\documentclass[12pt]{article}

\usepackage{amsmath,amssymb,amsthm,amscd,bm}
\usepackage{graphicx}
\usepackage[all]{xy}

\allowdisplaybreaks

\numberwithin{equation}{section}

\newcommand{\Fg}{\mathfrak{g}}
\newcommand{\Fh}{\mathfrak{h}}
\newcommand{\Fb}{\mathfrak{b}}
\newcommand{\Fn}{\mathfrak{n}}
\newcommand{\Fe}{\mathfrak{e}}
\newcommand{\FS}{\mathfrak{S}}

\newcommand{\BZ}{\mathbb{Z}}

\newcommand{\BR}{\mathbb{R}}
\newcommand{\BC}{\mathbb{C}}
\newcommand{\BB}{\mathbb{B}}
\newcommand{\BW}{\mathbb{W}}

\newcommand{\CL}{\mathcal{L}}
\newcommand{\CB}{\mathcal{B}}
\newcommand{\CS}{\mathcal{S}}

\newcommand{\ve}{\varepsilon}
\newcommand{\vp}{\varphi}
\newcommand{\vpi}{\varpi}
\newcommand{\vrho}{\omega}

\newcommand{\ba}{\mathbf{a}}
\newcommand{\bb}{\mathbf{b}}
\newcommand{\bx}{\mathbf{x}}

\newcommand{\bzero}{\bm{0}}
\newcommand{\brho}{\bm{\rho}}
\newcommand{\bchi}{\bm{\chi}}
\newcommand{\bvrho}{\bm{\omega}}

\newcommand{\q}{\mathsf{q}}

\newcommand{\J}{J}
\newcommand{\K}{K}
\newcommand{\QJ}{Q_{J}}

\newcommand{\QJv}{Q_{J}^{\vee}}
\newcommand{\QKv}{Q_{K}^{\vee}}
\newcommand{\QJvp}{Q_{J}^{\vee,+}}

\newcommand{\DeJ}{\Delta_{J}}
\newcommand{\DeK}{\Delta_{K}}
\newcommand{\DeS}{\Delta_{S}}
\newcommand{\PJ}{\Pi^{J}}
\newcommand{\PK}{\Pi^{K}}
\newcommand{\PS}{\Pi^{S}}
\newcommand{\WJ}{W_{J}}
\newcommand{\WK}{W_{K}}
\newcommand{\WS}{W_{S}}
\newcommand{\WJu}{W^{J}}
\newcommand{\WKu}{W^{K}}
\newcommand{\WSu}{W^{S}}

\newcommand{\ZP}{\mathbb{Z}[P]}

\newcommand{\si}{\frac{\infty}{2}}
\newcommand{\sell}{\ell^{\si}}
\newcommand{\sil}{\prec}
\newcommand{\sile}{\preceq}
\newcommand{\sig}{\succ}
\newcommand{\sige}{\succeq}

\newcommand{\SB}{\mathrm{BG}^{\si}(W_{\af})}
\newcommand{\SBJ}{\mathrm{BG}^{\si}\bigl((\WJu)_{\af}\bigr)}

\newcommand{\SBa}{\mathrm{BG}^{\si}_{\lambda,a}\bigl((\WJu)_{\af}\bigr)}
\newcommand{\SBb}[1]{\mathrm{BG}^{\si}_{\lambda,#1}\bigl((\WJu)_{\af}\bigr)}

\newcommand{\SLS}{\mathbb{B}^{\si}}
\newcommand{\SM}{\mathbb{S}^{\frac{\infty}{2}}}
\newcommand{\DC}{\mathbb{D}^{\frac{\infty}{2}}}
\newcommand{\QLS}{\mathrm{QLS}}

\newcommand{\Deo}[2]{\iota(#1,\,#2)}
\newcommand{\Lif}[2]{\mathrm{Lift}_{\sige #1}(#2)}
\newcommand{\Lift}{\mathrm{Lift}}

\newcommand{\cl}{\mathop{\rm cl}\nolimits}
\newcommand{\id}{\mathop{\rm id}\nolimits}
\newcommand{\wt}{\mathop{\rm wt}\nolimits}
\newcommand{\fwt}{\mathop{\rm fin}\nolimits}
\newcommand{\qwt}{\mathop{\rm nul}\nolimits}
\newcommand{\ch}{\mathop{\rm ch}\nolimits}
\newcommand{\gch}{\mathop{\rm gch}\nolimits}
\newcommand{\gr}{\mathop{\rm gr}\nolimits}
\newcommand{\Hom}{\mathop{\rm Hom}\nolimits}
\newcommand{\af}{\mathrm{af}}

\newcommand{\Par}{\mathop{\rm Par}\nolimits}
\newcommand{\Turn}{\mathop{\rm Turn}\nolimits}
\newcommand{\Conn}{\mathop{\rm Conn}\nolimits}

\newcommand{\rr}{\Delta_{\af}}
\newcommand{\prr}{\Delta_{\af}^{+}}

\newcommand{\mcr}[1]{\lfloor #1 \rfloor}
\newcommand{\edge}[1]{\xrightarrow{\,#1\,}}

\newcommand{\pair}[2]{\langle #1,\,#2 \rangle}

\newcommand{\ol}[1]{\overline{#1}}
\newcommand{\ti}[1]{\widetilde{#1}}
\newcommand{\ha}[1]{\widehat{#1}}

\newcommand{\BP}{\mathbb{P}}
\newcommand{\BK}{\mathbb{K}}
\newcommand{\BA}{\mathbb{A}}

\newcommand{\FQ}{\mathfrak{Q}}
\newcommand{\FI}{\mathfrak{I}}
\newcommand{\CO}{\mathcal{O}}
\newcommand{\CE}{\mathcal{E}}
\newcommand{\CP}{\mathcal{P}}
\newcommand{\CH}{\mathcal{H}}
\newcommand{\CU}{\mathcal{U}}
\newcommand{\CW}{\mathcal{W}}

\newcommand{\bQ}{\mathbf{Q}}
\newcommand{\bO}{\mathbf{O}}
\newcommand{\bD}{\mathbf{D}}
\newcommand{\bi}{\mathbf{i}}
\newcommand{\bj}{\mathbf{j}}
\newcommand{\be}{\mathbf{e}}
\newcommand{\bq}{\mathbf{q}}

\newcommand{\te}{\mathtt{e}}
\newcommand{\tT}{\mathtt{T}}

\newcommand{\sfm}{\mathsf{m}}

\newcommand{\Iw}{\mathbf{I}}
\newcommand{\Gm}{\mathbb{G}_{m}}
\newcommand{\Wm}{\mathbf{W}_{\af}}

\newcommand{\gdim}{\mathop{\rm gdim}\nolimits}
\newcommand{\Img}{\mathop{\rm Image}\nolimits}
\newcommand{\coker}{\mathop{\rm Coker}\nolimits}
\newcommand{\Pic}{\mathop{\rm Pic}\nolimits}

\newcommand{\bra}[1]{[\hspace{-1pt}[#1]\hspace{-1pt}]}
\newcommand{\sbra}[1]{[\hspace{-0.5pt}[#1]\hspace{-0.5pt}]}
\newcommand{\pra}[1]{(\hspace{-1pt}(#1)\hspace{-1pt})}

\newcommand{\Spec}{\mathop{\rm Spec}\nolimits}
\newcommand{\Lie}{\mathop{\rm Lie}\nolimits}

\newcommand{\Fun}{\mathop{\rm Fun}\nolimits}

\newcommand{\ev}{\mathop{\tt ev}\nolimits}
\newcommand{\Proj}{\mathop{\rm Proj}\nolimits}
\newcommand{\pr}{\mathop{\rm pr}\nolimits}

\newcommand{\trs}{\tau}

\newcommand{\DP}{\mathbf{Q}_{G}}
\newcommand{\tDP}{\widetilde{\mathbf{Q}}_{G}}
\newcommand{\hDP}{\widehat{\mathbf{Q}}_{G}}
\newcommand{\cDP}{\mathring{\mathbf{Q}}_{G}}
\newcommand{\RDP}{\mathbf{Q}_{G}^{\mathrm{rat}}}

\newcommand{\dH}{{\mathcal{H}\hspace{-5.6pt}\mathcal{H}}}

\theoremstyle{plain}
\newtheorem{thm}{Theorem}[section]
\newtheorem{lem}[thm]{Lemma}
\newtheorem{prop}[thm]{Proposition}
\newtheorem{cor}[thm]{Corollary}

\newtheorem{claim}{Claim}[thm]

\newtheorem{ithm}{Theorem}
\newtheorem{iprop}[ithm]{Proposition}

\theoremstyle{definition}
\newtheorem{dfn}[thm]{Definition}

\theoremstyle{remark}
\newtheorem{rem}[thm]{Remark}

\newenvironment{enu}{%
 \begin{enumerate}%
}{\end{enumerate}}

\newcommand{\bqed}{\quad \hbox{\rule[-0.5pt]{3pt}{8pt}}}

\newcommand{\vsp}{\vspace{3mm}}

\begin{document}

\setlength{\baselineskip}{14pt}

\title{\LARGE\bf 
Equivariant $K$-theory of \\[2mm]
semi-infinite flag manifolds \\[2mm]
and Pieri-Chevalley formula%
}
\author{%
Syu Kato \\
 \small Department of Mathematics, Kyoto University, \\
 \small Oiwake Kita-Shirakawa, Sakyo, Kyoto 606-8502, Japan \\
 \small (e-mail: {\tt syuchan@math.kyoto-u.ac.jp}) \\[3mm]
Satoshi Naito \\ 
 \small Department of Mathematics, Tokyo Institute of Technology, \\
 \small 2-12-1 Oh-okayama, Meguro-ku, Tokyo 152-8551, Japan \\
 \small (e-mail: {\tt naito@math.titech.ac.jp}) \\[3mm]
and \\[3mm]
Daisuke Sagaki \\ 
 \small Institute of Mathematics, University of Tsukuba, \\
 \small 1-1-1 Tennodai, Tsukuba, Ibaraki 305-8571, Japan \\
 \small (e-mail: {\tt sagaki@math.tsukuba.ac.jp})
}
\date{}
\maketitle

%
\begin{abstract} \setlength{\baselineskip}{12pt}
We propose a definition of equivariant (with respect to an Iwahori subgroup) $K$-theory of 
the formal power series model $\DP$ of semi-infinite flag manifold and prove 
the Pieri-Chevalley formula, which describes the product, in the $K$-theory of $\DP$, 
of the structure sheaf of a semi-infinite Schubert variety with a line bundle 
(associated to a dominant integral weight) over $\DP$.
In order to achieve this, we provide a number of fundamental results on $\DP$ and 
its Schubert subvarieties including the Borel-Weil-Bott theory, 
whose special case is conjectured in \cite{BF14c}. 
One more ingredient of this paper besides the geometric results above is 
(a combinatorial version of) standard monomial theory for 
level-zero extremal weight modules over quantum affine algebras, 
which is described in terms of semi-infinite Lakshmibai-Seshadri paths.
In fact, in our Pieri-Chevalley formula, the positivity of structure coefficients is 
proved by giving an explicit representation-theoretic meaning 
through semi-infinite Lakshmibai-Seshadri paths.
\end{abstract}
%
%
\section{Introduction.} 
\label{sec:intro}

Let $G$ be a connected and simply-connected simple algebraic group over $\BC$, 
and let $X$ be the flag variety of $G$. The torus-equivariant Grothendieck group $K_H ( X )$ of $X$ 
affords rich structures from the perspective of geometry and representation theory. 
One of the highlights there is the positivity of the structure constants of the products 
among natural classes (called the Schubert classes; see Anderson-Griffeth-Miller \cite{AGM}, 
and Baldwin-Kumar \cite{BK}), which serves as a basis of its interaction with the eigenvalue problems 
\cite{Kly98} and Gaudin models \cite{MTV}. There is a variant of this theme 
(called the Pieri-Chevalley formula), namely the structure constants of 
the products between Schubert classes and (ample) line bundles in $K_H ( X )$, 
which is also known to be positive by Mathieu \cite{Mat00} and Brion \cite{Br02}.

Pittie and Ram \cite{PR99} initiated a program to describe 
such a positive structure constant by relating them 
with the standard monomial theory (SMT for short). 
In particular, they gave an explicit meaning of each structure coefficient 
in the Pieri-Chevalley formula in terms of Lakshmibai-Seshadri paths 
(LS paths for short; see, e.g., \cite{Lit95}), 
which carries almost all information about simple $G$-modules. 
Their program is subsequently completed by Littelmann-Seshadri \cite{LiSe03} and 
Lenart-Shimozono \cite{LeSh14} (see also Lenart-Postnikov \cite{LeP07}). 

Peterson \cite{Pet97} noticed that the quantum $K$-theory of $X$ should be 
intimately connected with the $K$-theory of the ``affine version" of $X$ 
(see Lam-Shimozono \cite{LS10} and Lam-Li-Mihalcea-Shimozono \cite{LLMS}). 
In view of Givental-Lee \cite{GL03} and Braverman-Finkelberg \cite{BF14a,BF14c}, 
the quantum $K$-theory of $X$ can be defined through the space of quasi-maps, 
whose union forms a dense subset of the formal power series model $\bQ_G$ of 
semi-infinite flag manifolds (cf. Finkelberg-Mirkovi\'c \cite{FM99}).

Therefore, it is quite natural to make some rigorous sense of $K_H ( \DP )$ and 
provide the Pieri-Chevalley formula using SMT, which is compatible with 
the pictures provided by Pittie-Ram and Peterson. This is what we perform in this paper 
by affording two new theories: {\bf 1)} the Borel-Weil-Bott theory of $\DP$ 
that enables us to define and calculate a version of $K_H ( \DP )$, and {\bf 2)} 
the SMT of level-zero modules over quantum affine algebras. 
We remark that the level-zero modules over quantum affine algebras admit 
an interpretation through the geometry of affine Grassmannian (of Langlands dual type), 
which is the ``affine version'' of $X$ (see, e.g., Lenart-Naito-Sagaki-Schilling-Shimozono 
\cite[Introduction]{LNSSS2}).

In order to explain our results, theories, and ideas more precisely, 
we need some notation. Let $\Fg$ denote the Lie algebra of $G$, 
and let $\Fg_\af$ denote the associated untwisted affine Lie algebra; 
we fix a Borel subgroup $B$ of $G$ and a maximal torus $H \subset B$, 
and set $N := [B,B]$. Let $W = N_{G}(H)/H$ be the Weyl group, 
which is generated by the simple reflections $s_{i}$, $i \in I$; 
$W$ can be thought of as acting on the dual space $\Fh^{\ast}$ of 
the Cartan subalgebra $\Fh := \Lie(H)$. We set $W_{\af} := W \ltimes Q^{\vee}$, 
with $Q^{\vee, +} := \sum_{i \in I} \BZ_{\geq 0} \alpha_{i}^{\vee} \subset 
Q^{\vee} := \bigoplus_{i \in I} \BZ \alpha_{i}^{\vee}$ (the coroot lattice). 
Let $P = \bigoplus_{i \in I} \BZ \vpi_{i} \subset \Fh^{\ast}$ be 
the weight lattice generated by the fundamental weights $\vpi_{i}$, $i \in I$, 
and set $P^+ := \sum_{i \in I} \BZ_{\ge 0} \vpi_{i}$. 

For an algebraic group $E$ over $\BC$, 
we denote by $E\pra{z}$ and $E\bra{z}$ 
the space of $\BC\pra{z}$-valued points and the space of $\BC\bra{z}$-valued points of $E$, respectively, 
viewed as an (ind-)scheme over $\BC$. Let $\ev_{0} : G\bra{z} \rightarrow G$ be 
the evaluation map at $z = 0$, and set $\Iw := \ev_{0}^{-1}(B)$, 
an Iwahori subgroup of $G\bra{z}$; we also set $\tilde{\Iw} := \Iw \rtimes \BC^{\ast}$, 
the semi-direct product group, 
where the group $\BC^{\ast}$ (of loop rotations) acts on $\Iw$ as the dilation on $z$.
Now, we define $\RDP := G\pra{z} / H N\pra{z}$, which is a pure ind-scheme of infinite type. 
Then the set of $\Iw$-orbits is in natural bijection with $W_{\af}$; 
let $\DP ( x )$ denote the $\Iw$-orbit closure corresponding to $x \in W_{\af}$. 
We define $\DP := \DP ( e )$.

Our first main result is the following:
%
%
\begin{ithm}[$\doteq$ Theorem~\ref{normal} and Corollary~\ref{pcr}] \label{fnormal}
For each $x \in W_\af$, the scheme $\DP ( x )$ is normal. 
In addition, there is an explicit $P^{+}$-graded algebra $R_G$ such that 
$\DP = \Proj R_G$; here our $\Proj$ is the $P^+$-graded one.
\end{ithm}

We remark that Theorem~\ref{fnormal} affirmatively answers 
\cite[Conjecture 2.1]{BF14c} and relevant speculations therein. 
Also, as we see below, the scheme $\DP$ is far from being ``compact'' 
(cf. \cite[Theorem~A]{Kat17} and \cite[(7.1)]{FGT}). 
In order to prove Theorem \ref{fnormal} naturally, 
we introduce a ``semi-infinite" Bott-Samelson-Demazure-Hansen tower 
that yields a normal ring $R_G$. From the construction, 
$R_G$ contains the projective coordinate ring of $\DP$. 
Moreover, on the basis of the fact that $R_G$ is generated by 
the primitive degree terms, a detailed comparison with the computation 
for the dense subset in \cite{BF14c} implies that the inclusion must be an isomorphism.

For each $x \in W_\af$ and $\lambda \in P$, 
we have an associated $G\bra{z}$-equivariant line bundle $\CO_{\DP(x)}(\lambda)$ over $\DP(x)$. 
Also, for each $x \in W$ and $\lambda \in P^+$, 
we have a Demazure submodule $V^{-}_{x} ( \lambda )$, in the sense of \cite{Kas05}, 
of the level-zero extremal weight module $V ( \lambda )$ (of extremal weight $\lambda$) 
over the quantum affine algebra $U_{\q}(\Fg_{\af})$ associated to $\Fg_{\af}$.
%
%
\begin{ithm}[$\doteq$ Theorem~\ref{coh-e}] \label{fBWB}
For each $x \in W_\af$ and $\lambda \in P$, we have
\begin{equation*}
\gch H^{i} ( \DP ( x ), \CO_{\DP(x)}( \lambda ) ) = 
  \begin{cases} 
    \gch V^{-}_{x} ( - w_{\circ} \lambda ) & \text{\rm if $i = 0$ and $\lambda \in P^{+}$}, \\[1.5mm]
    0 & \text{\rm otherwise}, 
  \end{cases}
\end{equation*}
where $\gch$ denotes the character taking values in 
$(\BZ\pra{q^{-1}})[P]$, and $w_{\circ} \in W$ is the longest element.
\end{ithm}

The higher cohomology vanishing part of Theorem~\ref{fBWB} 
is based on the fact that the ring $R_G$ is free over a polynomial ring 
with infinitely many variables (Theorem~\ref{R-free}), 
which is also an interesting result in its own. 
We should mention that Theorem \ref{fBWB} have an ind-model counterpart in \cite{BF14c}, 
but there are no implications between these and the two proofs are totally different.
%
%
\begin{iprop}[$\doteq$ Proposition~\ref{line-bundles}] \label{fLB}
For each $x \in W_\af$, every $\ti{\Iw}$-equivariant line bundle 
over the scheme $\DP ( x )$ is isomorphic to some $\CO_{\DP(x)}( \lambda )$ up to character twist.
\end{iprop}

Although $R_G$ itself is highly infinite-dimensional (it is {\it not} even finitely generated), 
it admits a grading such that it is almost like an Artin algebra in a graded sense. 
Moreover, Proposition \ref{fLB} supplies ``graded indecomposable projectives'' of $R_G$. 
These two facts, combined with Theorem \ref{fBWB}, assert that the category of 
$\ti{\Iw}$-equivariant sheaves on $\DP$ (and on $\RDP$) behaves almost like the category of 
coherent sheaves on an affine scheme.

This series of observations enables us 
to define a reasonable variant of an equivariant $K$-group $K^{\prime}_{\ti{\Iw}}(\DP)$ of 
$\DP$ (and $K_{\ti{\Iw}}(\RDP)$ of $\RDP$) with respect to $\tilde{\Iw}$; 
see Section~\ref{sec:K-SiFl} for details. They are rather involved, 
partly because we need to specify a class of formal power series 
that is large enough to afford the Pieri-Chevalley rule, and 
at the same time is small enough so that the Euler character map is injective. 
Nevertheless, we can prove that $K^{\prime}_{\ti{\Iw}}(\DP)$ contains 
(the classes of) the sheaves $[\CO_{\DP(y)} ( \lambda )]$ 
for each $\lambda \in P$ and relevant $y \in W_{\af}$. 
We also prove that our $K_{\ti{\Iw}}(\DP^{\mathrm{rat}})$ is natural enough 
so that it admits a nil-DAHA action as an analog of Kostant-Kumar \cite{KK90} 
for $\RDP$ (see Section~\ref{sec:nDAHA} for details).

Here we recall that in Ishii-Naito-Sagaki \cite{INS} and Naito-Sagaki \cite{NS16}, 
the semi-infinite path model of the crystal basis of $V_{x}^{-} ( \lambda )$ is 
constructed for every $\lambda \in P^+$ and $x \in W_{\af}$; 
it is a specific subset of the set of ``semi-infinite'' LS paths $\SLS ( \lambda )$ of 
shape $\lambda$ parametrizing the global crystal basis of $V ( \lambda )$. 
Note that it is endowed with three functions
\begin{equation*}
\iota,\,\kappa : \SLS ( \lambda ) \longrightarrow W_\af 
\quad \text{and} \quad
\wt : \SLS ( \lambda ) \longrightarrow P \oplus \BZ \delta,
\end{equation*}
which are called the initial/final directions and the weight, respectively. We set
\begin{equation*}
\SLS_{\sige x} ( \lambda ) := 
 \bigl\{ \eta \in \SLS ( \lambda ) \mid \kappa ( \eta ) \sige x \bigr\}.
\end{equation*}

In order to make use of the path model above to 
derive the Pieri-Chevalley formula for $\RDP$, 
we additionally need a combinatorial version of the semi-infinite SMT. 
This consists of the definition of the initial direction $\Deo{\eta}{x} \in W_{\af}$ of 
a semi-infinite LS path $\eta$ with respect to $x$ (based on the existence of 
the semi-infinite analog of the so-called Deodhar lift), and 
of the description of tensor product decomposition of crystals in terms of $\Deo{\bullet}{x}$ 
(Theorem~\ref{thm:SMT} and Theorem~\ref{thm:Dem}). We remark that our $\Deo{\eta}{x}$ is 
an analogue of the one in \cite{LiSe03, LeSh14} in the setting of 
level-zero extremal weight modules over $U_{\q}(\Fg_{\af})$. 
Using them, we obtain our Pieri-Chevalley formula:
%
%
\begin{ithm}[$\doteq$ Theorem \ref{thm:PC}]\label{fPCf}
For $\lambda \in P^{+}$ and $x \in W_{\af}^{\ge 0}:=W \times Q^{\vee,+}$, we have
\begin{equation} \label{fPC}
[\CO_{\DP} ( \lambda )] \cdot [\CO_{\DP ( x )}] = 
 \sum_{\eta \in \SLS_{\sige x} (- w_{\circ} \lambda )} 
 e^{\fwt(\eta)}q^{\qwt(\eta)} \cdot [ \CO_{\DP ( \Deo{\eta}{x} )}] \quad \in K^{\prime}_{\ti{\Iw}}(\DP),
\end{equation}
where $\fwt(\eta) \in P$ and $\qwt(\eta) \in \BZ$ 
for $\eta \in \SLS (- w_{\circ} \lambda )$ are defined by:
\begin{equation*}
\wt(\eta) = \fwt(\eta) + \qwt(\eta) \delta.
\end{equation*}
\end{ithm}

Once generalities on $K^{\prime}_{\ti{\Iw}}(\DP)$ and 
the semi-infinite SMT are given, our strategy 
for the proof of Theorem \ref{fPCf} is 
along the line of \cite{LiSe03}. Namely, we compare the functionals
\begin{equation*}
P \ni \lambda \mapsto \sum_{i \ge 0}(-1)^{i} 
  \gch H^{i}( \DP, \CE \otimes_{\CO_{\DP}} \CO_{\DP} ( \lambda ) ) \in \BC [P]\pra{q^{-1}},
\end{equation*}
where $[\CE]$ is taken from the both sides of \eqref{fPC}.

This paper is organized as follows.
In Section~\ref{sec:notation}, 
we fix our notation for untwisted affine Lie algebras,
and then recall some basic facts about semi-infinite LS paths, extremal
weight modules, and their Demazure submodules.
In Section~\ref{sec:SMT}, we state a combinatorial version of 
standard monomial theory for level-zero extremal weight modules, 
and also its refinement for Demazure submodules; 
the proofs of these results are given in 
Sections~\ref{sec:prf-SMT}, \ref{sec:prf-DC}, and \ref{sec:prf-Dem}.
In Section~\ref{sec:SiSch}, we first review the formal power series model $\DP$ 
of semi-infinite flag manifold, and then introduce
a semi-infinite version of Bott-Samelson-Demazure-Hansen tower for $\DP$.
Then, we study the cohomology spaces of line bundles over $\DP$,
and prove the higher cohomology vanishing; also, 
we describe the spaces of global sections in terms of
Demazure submodules of extremal weight modules.
As an application, we prove
the normality of the semi-infinite Schubert varieties $\DP(x)$, $x \in W_{\af}^{\geq 0}$.
In Section~\ref{sec:K-SiFl}, after giving a definition of 
$\ti{\Iw}$-equivariant $K$-group $K^{\prime}_{\ti{\Iw}}(\DP)$ 
of $\DP$ (and $K_{\ti{\Iw}}(\RDP)$ of $\RDP$),
we establish the Pieri-Chevalley formula (Theorem \ref{thm:PC}) 
by combining our geometric results with the semi-infinite SMT.
Also, in Section~\ref{sec:nDAHA}, we show that our $K$-group $K_{\ti{\Iw}}(\RDP)$ 
admits a natural nil-DAHA action.
Appendices mainly contain some technical results concerning 
the semi-infinite Bruhat order; in particular, 
we prove the existence of analogs of Deodhar lifts for the semi-infinite Bruhat order.
%
%
\subsection*{Acknowledgments.}
We thank Michael Finkelberg for sending us unpublished manuscripts. 
S.K. was partially supported by 
JSPS Grant-in-Aid for Scientific Research (B) 26287004 and 
Kyoto University Jung-Mung program. 
S.N. was partially supported by 
JSPS Grant-in-Aid for Scientific Research (B) 16H03920. 
D.S. was partially supported by 
JSPS Grant-in-Aid for Scientific Research (C) 15K04803. 

%
\section{Algebraic setting.}
\label{sec:notation}
%
%
\subsection{Affine Lie algebras.}
\label{subsec:liealg}

A graded vector space is a $\BZ$-graded vector space over $\BC$ 
all of whose homogeneous subspaces are finite-dimensional. 
Let $V = \bigoplus_{m \in \BZ} V_{m}$ be a graded vector space 
with $V_{m}$ its subspace of degree $m$. We define
\begin{equation*}
\gdim V := \sum_{m \in \BZ} \bigl( \dim V_{m} \bigr) q^{m}. 
\end{equation*}
Also, we denote by $V^{\vee}$ (resp., $V^{\ast}$) 
the full (resp., restricted) dual of $V$; 
note that $V^{\ast} := \bigoplus _{m \in \BZ} ( V^{\ast} )_{m}$, 
with $(V^{\ast})_{m}:=(V_{-m})^{\ast}$. 
In addition, we set $\ha{V}:=\prod_{m \in \BZ} V_{m}$, 
which is a completion of $V$.

Let $G$ be a connected, 
simply-connected simple algebraic group over $\BC$, 
and $B$ a Borel subgroup with unipotent radical $N$. 
We fix a maximal torus $H \subset B$, and take the 
opposite Borel subgroup $B^{-}$ of $G$ that contains $H$. 
In the following, for an (arbitrary) algebraic group $E$ over $\BC$, 
we denote its Lie algebra $\Lie(E)$ by the corresponding German letter $\Fe$; 
in particular, we write $\Fg=\Lie(G)$, $\Fb=\Lie(B)$, $\Fn=\Lie(N)$, and 
$\Fh=\Lie(H)$. 
Thus, $\Fg$ is a finite-dimensional simple Lie algebra over $\BC$
with Cartan subalgebra $\Fh$. 
Denote by $\{ \alpha_{i}^{\vee} \}_{i \in I}$ and 
$\{ \alpha_{i} \}_{i \in I}$ the set of simple coroots and 
simple roots of $\Fg$, respectively, and set
$Q := \bigoplus_{i \in I} \BZ \alpha_i$, 
$Q^{+} := \sum_{i \in I} \BZ_{\ge 0} \alpha_i$, and 
$Q^{\vee} := \bigoplus_{i \in I} \BZ \alpha_i^{\vee}$, 
$Q^{\vee,+} := \sum_{i \in I} \BZ_{\ge 0} \alpha_i^{\vee}$; 
for $\xi,\,\zeta \in Q^{\vee}$, we write $\xi \ge \zeta$ if $\xi-\zeta \in Q^{\vee,+}$. 
Let $\Delta$ and $\Delta^{+}$ be the set of roots and positive roots of $\Fg$, respectively, 
with $\theta \in \Delta^{+}$ the highest root of $\Fg$. 
For a root $\alpha \in \Delta$, we denote by $\alpha^{\vee}$ its dual root. We set
$\rho:=(1/2) \sum_{\alpha \in \Delta^{+}} \alpha$ and 
$\rho^{\vee}:= (1/2) \sum_{\alpha \in \Delta^{+}} \alpha^{\vee}$. 
Also, let $\vpi_{i}$, $i \in I$, denote the fundamental weights for $\Fg$, and set
%
%
\begin{equation} \label{eq:P-fin}
P:=\bigoplus_{i \in I} \BZ \vpi_{i}, \qquad 
P^{+} := \sum_{i \in I} \BZ_{\ge 0} \vpi_{i}. 
\end{equation} 

Let $\Fg_{\af} = \bigl(\Fg \otimes \BC[z,z^{-1}]\bigr) \oplus \BC c \oplus \BC d$ be 
the untwisted affine Lie algebra over $\BC$ associated to $\Fg$, 
where $c$ is the canonical central element, and $d$ is 
the scaling element (or the degree operator), 
with Cartan subalgebra $\Fh_{\af} = \Fh \oplus \BC c \oplus \BC d$. 
We regard an element $\mu \in \Fh^{\ast}:=\Hom_{\BC}(\Fh,\,\BC)$ as an element of 
$\Fh_{\af}^{\ast}$ by setting $\pair{\mu}{c}=\pair{\mu}{d}:=0$, where 
$\pair{\cdot\,}{\cdot}:\Fh_{\af}^{\ast} \times \Fh_{\af} \rightarrow \BC$ is 
the canonical pairing of $\Fh_{\af}^{\ast}:=\Hom_{\BC}(\Fh_{\af},\,\BC)$ and $\Fh_{\af}$. 
Let $\{ \alpha_{i}^{\vee} \}_{i \in I_{\af}} \subset \Fh_{\af}$ and 
$\{ \alpha_{i} \}_{i \in I_{\af}} \subset \Fh_{\af}^{\ast}$ be the set of 
simple coroots and simple roots of $\Fg_{\af}$, respectively, 
where $I_{\af}:=I \sqcup \{0\}$; note that 
$\pair{\alpha_{i}}{c}=0$ and $\pair{\alpha_{i}}{d}=\delta_{i0}$ 
for $i \in I_{\af}$. 
Denote by $\delta \in \Fh_{\af}^{\ast}$ the null root of $\Fg_{\af}$; 
recall that $\alpha_{0}=\delta-\theta$. 
Also, let $\Lambda_{i} \in \Fh_{\af}^{\ast}$, $i \in I_{\af}$, 
denote the fundamental weights for $\Fg_{\af}$ such that $\pair{\Lambda_{i}}{d}=0$, 
and set 
%
%
\begin{equation} \label{eq:P}
P_{\af} := 
  \left(\bigoplus_{i \in I_{\af}} \BZ \Lambda_{i}\right) \oplus 
   \BZ \delta \subset \Fh^{\ast}, \qquad 
P_{\af}^{0}:=\bigl\{\mu \in P_{\af} \mid \pair{\mu}{c}=0\bigr\};
\end{equation}
notice that $P_{\af}^{0}=P \oplus \BZ \delta$, and that
\begin{equation}
\pair{\mu}{\alpha_{0}^{\vee}} = - \pair{\mu}{\theta^{\vee}} \quad 
\text{for $\mu \in P_{\af}^{0}$}. 
\end{equation}

Let $W := \langle s_{i} \mid i \in I \rangle$ and 
$W_{\af} := \langle s_{i} \mid i \in I_{\af} \rangle$ be the (finite) Weyl group of $\Fg$ and 
the (affine) Weyl group of $\Fg_{\af}$, respectively, 
where $s_{i}$ is the simple reflection with respect to $\alpha_{i}$ 
for each $i \in I_{\af}$, with length function $\ell:W_{\af} \rightarrow \BZ_{\ge 0}$, 
which gives the one on $W$ by restriction; we denote by $e \in W_{\af}$ 
the identity element, and by $w_{\circ} \in W$ the longest element. 
For each $\xi \in Q^{\vee}$, let $t_{\xi} \in W_{\af}$ denote 
the translation in $\Fh_{\af}^{\ast}$ by $\xi$ (see \cite[Sect.~6.5]{Kac}); 
for $\xi \in Q^{\vee}$, we have 
%
%
\begin{equation}\label{eq:wtmu}
t_{\xi} \mu = \mu - \pair{\mu}{\xi}\delta \quad 
\text{if $\mu \in \Fh_{\af}^{\ast}$ satisfies $\pair{\mu}{c}=0$}.
\end{equation}
Then, $\bigl\{ t_{\xi} \mid \xi \in Q^{\vee} \bigr\}$ forms 
an abelian normal subgroup of $W_{\af}$, in which $t_{\xi} t_{\zeta} = t_{\xi + \zeta}$ 
holds for $\xi,\,\zeta \in Q^{\vee}$. Moreover, we know from \cite[Proposition 6.5]{Kac} that
\begin{equation*}
W_{\af} \cong 
 W \ltimes \bigl\{ t_{\xi} \mid \xi \in Q^{\vee} \bigr\} \cong W \ltimes Q^{\vee}; 
\end{equation*}
we also set
%
%
\begin{equation} \label{eq:Wge0}
W_{\af}^{\ge 0}:=\bigl\{wt_{\xi} \mid w \in W,\, \xi \in Q^{\vee,+} \bigr\} \subset W_{\af}. 
\end{equation}

Denote by $\rr$ the set of real roots of $\Fg_{\af}$, and 
by $\prr \subset \rr$ the set of positive real roots; 
we know from \cite[Proposition 6.3]{Kac} that
$\rr = 
\bigl\{ \alpha + n \delta \mid \alpha \in \Delta,\, n \in \BZ \bigr\}$, 
and 
$\prr = 
\Delta^{+} \sqcup 
\bigl\{ \alpha + n \delta \mid \alpha \in \Delta,\, n \in \BZ_{> 0}\bigr\}$. 
For $\beta \in \rr$, we denote by $\beta^{\vee} \in \Fh_{\af}$ 
its dual root, and $s_{\beta} \in W_{\af}$ the corresponding reflection; 
if $\beta \in \rr$ is of the form $\beta = \alpha + n \delta$ 
with $\alpha \in \Delta$ and $n \in \BZ$, then 
$s_{\beta} =s_{\alpha} t_{n\alpha^{\vee}} \in W \ltimes Q^{\vee}$.

Finally, let $U_{\q}(\Fg_{\af})$ denote the quantized universal enveloping algebra 
over $\BC(\q)$ associated to $\Fg_{\af}$, 
with $E_{i}$ and $F_{i}$, $i \in I_{\af}$, the Chevalley generators 
corresponding to $\alpha_{i}$ and $-\alpha_{i}$, respectively.  
We denote by $U_{\q}^{-}(\Fg_{\af})$ 
the negative part of $U_{\q}(\Fg_{\af})$, that is, 
the $\BC(\q)$-subalgebra of $U_{\q}(\Fg_{\af})$ generated by $F_{i}$, $i \in I_{\af}$. 
%
%
\subsection{Parabolic semi-infinite Bruhat graph.}
\label{subsec:SiBG}

In this subsection, we fix a subset $J \subset I$. 
We set 
$\QJ := \bigoplus_{i \in \J} \BZ \alpha_i$, 
$\QJv := \bigoplus_{i \in \J} \BZ \alpha_i^{\vee}$,  
$\QJvp := \sum_{i \in \J} \BZ_{\ge 0} \alpha_i^{\vee}$, 
$\DeJ := \Delta \cap \QJ$, 
$\DeJ^{+} := \Delta^{+} \cap \QJ$, and 
$\WJ := \langle s_{i} \mid i \in \J \rangle$.
Also, we denote by 
%
%
\begin{equation} \label{eq:prj}
[\,\cdot\,]_{\J} : 
Q^{\vee} \twoheadrightarrow \QJv \quad 
\text{(resp., $[\,\cdot\,]^{\J} : Q^{\vee} \twoheadrightarrow Q_{I \setminus \J}^{\vee}$)}
\end{equation}
the projection from $Q^{\vee}=Q_{I \setminus \J}^{\vee} \oplus \QJv$
onto $\QJv$ (resp., $Q_{I \setminus \J}^{\vee}$) 
with kernel $Q_{I \setminus \J}^{\vee}$ (resp., $\QJv$). 
Let $\WJu$ denote the set of minimal(-length) coset representatives 
for the cosets in $W/\WJ$; we know from \cite[Sect.~2.4]{BB} that 
%
%
\begin{equation} \label{eq:mcr}
\WJu = \bigl\{ w \in W \mid 
\text{$w \alpha \in \Delta^{+}$ for all $\alpha \in \DeJ^{+}$}\bigr\}.
\end{equation}
For $w \in W$, we denote by $\mcr{w}=\mcr{w}^{\J} \in \WJu$ 
the minimal coset representative for the coset $w \WJ$ in $W/\WJ$.
Also, following \cite{Pet97} 
(see also \cite[Sect.~10]{LS10}), we set
\begin{align}
(\DeJ)_{\af} 
  & := \bigl\{ \alpha + n \delta \mid 
  \alpha \in \DeJ,\,n \in \BZ \bigr\} \subset \Delta_{\af}, \\
(\DeJ)_{\af}^{+}
  &:= (\DeJ)_{\af} \cap \prr = 
  \DeJ^{+} \sqcup \bigl\{ \alpha + n \delta \mid 
  \alpha \in \DeJ,\, n \in \BZ_{> 0} \bigr\}, \\
\label{eq:stabilizer}
(\WJ)_{\af} 
 & := \WJ \ltimes \bigl\{ t_{\xi} \mid \xi \in \QJv \bigr\}
   = \bigl\langle s_{\beta} \mid \beta \in (\DeJ)_{\af}^{+} \bigr\rangle, \\
\label{eq:Pet}
(\WJu)_{\af}
 &:= \bigl\{ x \in W_{\af} \mid 
 \text{$x\beta \in \prr$ for all $\beta \in (\DeJ)_{\af}^{+}$} \bigr\};
\end{align}
note that if $\J = \emptyset$, then 
$(W^{\emptyset})_{\af}=W_{\af}$ and $(W_{\emptyset})_{\af}=\bigl\{e\bigr\}$. 
We know from \cite{Pet97} (see also \cite[Lemma~10.6]{LS10}) that 
for each $x \in W_{\af}$, there exist a unique 
$x_1 \in (\WJu)_{\af}$ and a unique $x_2 \in (\WJ)_{\af}$ 
such that $x = x_1 x_2$; we define a (surjective) map 
%
%
\begin{equation} \label{eq:PiJ}
\PJ : W_{\af} \twoheadrightarrow (\WJu)_{\af}, \quad x \mapsto x_{1}, 
\end{equation}
where $x= x_1 x_2$ with $x_1 \in (\WJu)_{\af}$ and $x_2 \in (\WJ)_{\af}$. 

%
\begin{dfn} \label{dfn:sell}
Let $x \in W_{\af}$, and 
write it as $x = w t_{\xi}$ for $w \in W$ and $\xi \in Q^{\vee}$. 
We define the semi-infinite length $\sell(x)$ of $x$ by:
$\sell (x) = \ell (w) + 2 \pair{\rho}{\xi}$. 
\end{dfn}
%
%
\begin{dfn}[\cite{Lu80}, \cite{Lu97}; see also \cite{Pet97}] \label{dfn:SiB}
\mbox{}
\begin{enu}
\item The (parabolic) semi-infinite Bruhat graph $\SBJ$ 
is the $\prr$-labeled, directed graph with vertex set $(\WJu)_{\af}$ 
whose directed edges are of the following form:
$x \edge{\beta} s_{\beta} x$ for $x \in (\WJu)_{\af}$ and $\beta \in \prr$, 
where $s_{\beta } x \in (\WJu)_{\af}$ and 
$\sell (s_{\beta} x) = \sell (x) + 1$. 
When $J=\emptyset$, we write $\SB$ for 
$\mathrm{BG}^{\si}\bigl((W^{\emptyset})_{\af}\bigr)$. 

\item 
The semi-infinite Bruhat order is a partial order 
$\sile$ on $(\WJu)_{\af}$ defined as follows: 
for $x,\,y \in (\WJu)_{\af}$, we write $x \sile y$ 
if there exists a directed path from $x$ to $y$ in $\SBJ$; 
we write $x \sil y$ if $x \sile y$ and $x \ne y$. 
\end{enu}
\end{dfn}

\begin{rem}
In the case $J = \emptyset$, the semi-infinite Bruhat order on $W_{\af}$ is 
just the generic Bruhat order introduced in \cite{Lu80}; 
see \cite[Appendix~A.3]{INS} for details. Also, for a general $J$, 
the parabolic semi-infinite Bruhat order on $(\WJu)_{\af}$
is nothing but the partial order on $J$-alcoves introduced in
\cite{Lu97} when we take a special point to be the origin.
\end{rem}

In Appendix~\ref{sec:basic}, we recall some of 
the basic properties of the semi-infinite Bruhat order. 

For $x \in (\WJu)_{\af}$, 
let $\Lift(x)$ denote the set of lifts of $x$ in $W_{\af}$ 
with respect to the map $\PJ:W_{\af} \twoheadrightarrow (\WJu)_{\af}$, that is, 
%
%
\begin{equation} \label{eq:lift}
\Lift(x):=
\bigl\{ x' \in W_{\af} \mid \PJ(x') = x \bigr\}; 
\end{equation}
for an explicit description of $\Lift(x)$, see Lemma~\ref{lem:lift}. 
The following proposition will be proved in 
Appendix~\ref{sec:prf-Deo}. 
%
%
\begin{prop} \label{prop:Deo}
If $x \in W_{\af}$ and $y \in (\WJu)_{\af}$ satisfy the condition that 
$y \sige \PJ(x)$, then the set
\begin{equation}
\Lif{x}{y}:=
\bigl\{ y' \in \Lift(y) \mid y' \sige x \bigr\}
\end{equation}
has the minimum element with respect to the semi-infinite Bruhat order on $W_{\af}${\rm;}
we denote this element by $\min \Lif{x}{y}$. 
\end{prop}
%
%
\subsection{Semi-infinite Lakshmibai-Seshadri paths.}
\label{subsec:SLS}

In this subsection, we fix $\lambda \in P^{+} \subset P_{\af}^{0}$ 
(see \eqref{eq:P-fin} and \eqref{eq:P}), and set 
%
%
\begin{equation} \label{eq:J}
J:= \bigl\{ i \in I \mid \pair{\lambda}{\alpha_i^{\vee}}=0 \bigr\} \subset I.
\end{equation}
%
%
\begin{dfn} \label{dfn:SBa}
For a rational number $0 < a < 1$, 
we define $\SBa$ to be the subgraph of $\SBJ$ 
with the same vertex set but having only the edges of the form
$x \edge{\beta} y$ with 
$a\pair{x\lambda}{\beta^{\vee}} \in \BZ$.
\end{dfn}
%
%
\begin{dfn}\label{dfn:SLS}
A semi-infinite Lakshmibai-Seshadri (LS for short) path of 
shape $\lambda $ is a pair 
%
%
\begin{equation} \label{eq:SLS}
\pi = (\bx \,;\, \ba) 
     = (x_{1},\,\dots,\,x_{s} \,;\, a_{0},\,a_{1},\,\dots,\,a_{s}), \quad s \ge 1, 
\end{equation}
of a strictly decreasing sequence $\bx : x_1 \sig \cdots \sig x_s$ 
of elements in $(\WJu)_{\af}$ and an increasing sequence 
$\ba : 0 = a_0 < a_1 < \cdots  < a_s =1$ of rational numbers 
satisfying the condition that there exists a directed path 
from $x_{u+1}$ to  $x_{u}$ in $\SBb{a_{u}}$ 
for each $u = 1,\,2,\,\dots,\,s-1$. 
We denote by $\SLS(\lambda)$ 
the set of all semi-infinite LS paths of shape $\lambda$.
\end{dfn}

Following \cite[Sect.~3.1]{INS} (see also \cite[Sect.~2.4]{NS16}), 
we endow the set $\SLS(\lambda)$ 
with a crystal structure with weights in $P_{\af}$ by 
the map $\wt:\SLS(\lambda) \rightarrow P_{\af}$ 
and the root operators $e_{i}$, $f_{i}$, $i \in I_{\af}$; 
for details, see Appendix~\ref{sec:crystal}. 
We denote by $\SLS_{0}(\lambda)$ the connected component of 
$\SLS(\lambda)$ containing 
$\pi_{\lambda}:=(e\,;\,0,1) \in \SLS(\lambda)$. 

If $\pi \in \SLS(\lambda)$ is of the form \eqref{eq:SLS}, 
then we set 
%
%
\begin{equation} \label{eq:dir}
\iota(\pi):=x_{1} \in (\WJu)_{\af} \qquad 
 \text{(resp., $\kappa(\pi):=x_{s} \in (\WJu)_{\af}$)}; 
\end{equation}
we call $\iota(\pi)$ (resp., $\kappa(\pi)$) 
the initial (resp., final) direction of $\pi$. 
For $x \in W_{\af}$, we set
%
%
\begin{equation} \label{eq:SLS-dem}
\SLS_{\sige x}(\lambda) := 
 \bigl\{
   \pi \in \SLS(\lambda) \mid \kappa(\pi) \sige \PJ(x)
 \bigr\}. 
\end{equation}
%
%
\subsection{Extremal weight modules and their Demazure submodules.}
\label{subsec:extremal}

In this subsection, we fix $\lambda \in P^{+} \subset P_{\af}^{0}$ 
(see \eqref{eq:P-fin} and \eqref{eq:P}). 
Let $V(\lambda)$ denote the extremal weight module of 
extremal weight $\lambda$ over $U_{\q}(\Fg_{\af})$, 
which is an integrable $U_{\q}(\Fg_{\af})$-module generated by 
a single element $v_{\lambda}$ with 
the defining relation that $v_{\lambda}$ is 
an ``extremal weight vector'' of weight $\lambda$; 
recall from \cite[Sect.~3.1]{Kas02} and \cite[Sect.~2.6]{Kas05} that 
$v_{\lambda}$ is an extremal weight vector of weight $\lambda$ 
if and only if ($v_{\lambda}$ is a weight vector of weight $\lambda$ and) 
there exists a family $\{ v_{x} \}_{x \in W_{\af}}$ 
of weight vectors in $V(\lambda)$ such that $v_{e}=v_{\lambda}$, 
and such that for every $i \in I_{\af}$ and $x \in W_{\af}$ with 
$n:=\pair{x\lambda}{\alpha_{i}^{\vee}} \ge 0$ (resp., $\le 0$),
the equalities $E_{i}v_{x}=0$ and $F_{i}^{(n)}v_{x}=v_{s_{i}x}$ 
(resp., $F_{i}v_{x}=0$ and $E_{i}^{(-n)}v_{x}=v_{s_{i}x}$) hold, 
where for $i \in I_{\af}$ and $k \in \BZ_{\ge 0}$, 
the $E_{i}^{(k)}$ and $F_{i}^{(k)}$ are the $k$-th divided powers of 
$E_{i}$ and $F_{i}$, respectively;
note that the weight of $v_{x}$ is $x\lambda$. 
Also, for each $x \in W_{\af}$, we define 
the Demazure submodule $V_{x}^{-}(\lambda)$ of $V(\lambda)$ by
%
%
\begin{equation} \label{eq:dem}
V_{x}^{-}(\lambda):=U_{\q}^{-}(\Fg_{\af})v_{x}. 
\end{equation}

We know from \cite[Proposition~8.2.2]{Kas94} that $V(\lambda)$ has 
a crystal basis $\CB(\lambda)$ and the corresponding 
global basis $\bigl\{G(b) \mid b \in \CB(\lambda)\bigr\}$; 
we denote by $u_{\lambda}$ the element of $\CB(\lambda)$ 
such that $G(u_{\lambda})=v_{\lambda}$, and 
by $\CB_{0}(\lambda)$ the connected component of $\CB(\lambda)$ 
containing $u_{\lambda}$. %
Also, we know from \cite[Sect.~2.8]{Kas05} (see also \cite[Sect.~4.1]{NS16}) that 
$V_{x}^{-}(\lambda) \subset V(\lambda)$ is compatible 
with the global basis of $V(\lambda)$, that is, 
there exists a subset $\CB_{x}^{-}(\lambda)$ of 
the crystal basis $\CB(\lambda)$ such that 
%
%
\begin{equation} \label{eq:deme}
V_{x}^{-}(\lambda) = 
\bigoplus_{b \in \CB_{x}^{-}(\lambda)} \BC(\q) G(b) 
\subset
V(\lambda) = 
\bigoplus_{b \in \CB(\lambda)} \BC(\q) G(b).
\end{equation}

\begin{rem}[{\cite[Lemma~4.1.2]{NS16}}]
For every $x \in W_{\af}$, we have 
$V_{x}^{-}(\lambda) = V_{\PJ(x)}^{-}(\lambda)$ and 
$\CB_{x}^{-}(\lambda) = \CB_{\PJ(x)}^{-}(\lambda)$.
\end{rem}
We know the following from 
\cite[Theorem~3.2.1]{INS} and \cite[Theorem~4.2.1]{NS16}. 
%
%
\begin{thm} \label{thm:isom}
There exists an isomorphism 
$\Phi_{\lambda}:\CB(\lambda) \stackrel{\sim}{\rightarrow} \SLS(\lambda)$ 
of crystals such that $\Phi(u_{\lambda}) = \pi_{\lambda}$ and 
such that $\Phi_{\lambda}(\CB_{x}^{-}(\lambda)) = \SLS_{\sige x}(\lambda)$ 
for all $x \in W_{\af}${\rm;} in particular, we have 
$\Phi_{\lambda}(\CB_{0}(\lambda)) = \SLS_{0}(\lambda)$. 
\end{thm}

Let $x \in W_{\af}$. If $x$ is of the form
$x = wt_{\xi}$ for some $w \in W$ and $\xi \in Q^{\vee}$, 
then $v_{x} \in V(\lambda)$ is a weight vector of weight $x\lambda = 
w\lambda-\pair{\lambda}{\xi}\delta$; note that $w\lambda \in \lambda-Q^{+}$. 
Also, for $i \in I$ (resp., $i=0 \in I_{\af}$), 
the Chevalley generator $F_{i}$ (resp., $F_{0}$) of $U_{\q}(\Fg_{\af})$ 
acts on $V(\lambda)$ as a (linear) operator of 
weight $-\alpha_{i} \in Q$ (resp., $-\alpha_{0}=\theta-\delta \in Q + \BZ_{< 0}\delta$). 
Therefore, the Demazure submodule 
$V_{x}^{-}(\lambda) = U_{\q}^{-}(\Fg_{\af})v_{x}$ has 
the weight space decomposition of the form: 
\begin{equation*}
V_{x}^{-}(\lambda) = 
 \bigoplus_{k \in \BZ}
\Biggl(
\underbrace{
   \bigoplus_{\gamma \in Q} 
   V_{x}^{-}(\lambda)_{\lambda+\gamma+k\delta}}_{=:V_{x}^{-}(\lambda)_{k}}
\Biggr),
\end{equation*}
where $V_{x}^{-}(\lambda)_{k} = \bigl\{ 0 \bigr\}$ 
for all $k > - \pair{\lambda}{\xi}$;
in addition, by Theorem~\ref{thm:isom}, 
together with the definition of the map 
$\wt:\SLS(\lambda) \rightarrow P_{\af}$ (see \eqref{eq:wt}), 
we see that if $\gamma \notin -Q^{+}$, then 
$V_{x}^{-}(\lambda)_{\lambda+\gamma+k\delta} = \bigl\{0\bigr\}$ 
for all $k \in \BZ$, since $W_{\af}\lambda \subset \lambda - Q^{+} +\BZ \delta$ 
by the assumption that $\lambda \in P^{+}$.
Here we claim that $V_{x}^{-}(\lambda)_{k}$ is finite-dimensional 
for all $k \in \BZ$ with $k \le -\pair{\lambda}{\xi}$; 
we show this assertion by descending induction on $k$. 
Let $U_{\q}^{-}(\Fg)$ denote the $\BC(\q)$-subalgebra of 
$U_{\q}^{-}(\Fg_{\af})$ generated by $F_{i}$, $i \in I$. 
If $k=-\pair{\lambda}{\xi}$, then the assertion is obvious since 
$V_{x}^{-}(\lambda)_{-\pair{\lambda}{\xi}} = U_{\q}^{-}(\Fg)v_{x}$ and 
$V(\lambda)$ is an integrable $U_{\q}(\Fg_{\af})$-module.
Assume that $k < -\pair{\lambda}{\xi}$. 
Observe that $V_{x}^{-}(\lambda)_{k}$ is 
a $U_{\q}^{-}(\Fg)$-module generated by $F_{0}V_{x}^{-}(\lambda)_{k+1}$. 
Because $F_{0}V_{x}^{-}(\lambda)_{k+1}$ is finite-dimensional by 
our induction hypothesis, and $V(\lambda)$ is 
an integrable $U_{q}(\Fg_{\af})$-module, we deduce that 
$V_{x}^{-}(\lambda)_{k} = U_{\q}^{-}(\Fg)(F_{0}V_{x}^{-}(\lambda)_{k+1})$ is 
also finite-dimensional, as desired. 

Now, we define the graded character $\gch V_{x}^{-}(\lambda)$ of 
$V_{x}^{-}(\lambda)$ to be 
\begin{equation} \label{eq:gch}
\gch V_{x}^{-}(\lambda) : = 
 \sum_{k \in \BZ}
\Biggl(
 \sum_{\gamma \in Q} 
 \dim \bigl( V_{x}^{-}(\lambda)_{\lambda+\gamma+k\delta} \bigr) 
 e^{\lambda+\gamma}
\Biggr) q^{k};
\end{equation}
observe that 
%
%
\begin{equation} \label{eq:gch1}
\gch V_{x}^{-}(\lambda) \in 
\bigl(\underbrace{\BZ[e^{\nu} \mid \nu \in P]}_{=\ZP}\bigr)
\bra{q^{-1}}q^{-\pair{\lambda}{\xi}}.
\end{equation}
For $\gamma \in Q$ and $k \in \BZ$, we set 
$\fwt(\lambda+\gamma+k\delta):=\lambda+\gamma \in P$ and 
$\qwt(\lambda+\gamma+k\delta):=k \in \BZ$. Then, by Theorem~\ref{thm:isom}, 
we have
\begin{equation} \label{eq:gch2}
\gch V_{x}^{-}(\lambda) = \sum_{\pi \in \SLS_{\sige x}(\lambda)} 
e^{\fwt(\wt(\pi))} q^{\qwt(\wt(\pi))}. 
\end{equation}
%
%
\section{Combinatorial standard monomial theory for semi-infinite LS paths.}
\label{sec:SMT}

In this section, we fix $\lambda,\,\mu \in P^{+} \subset P_{\af}^{0}$ 
(see \eqref{eq:P-fin} and \eqref{eq:P}), and set
\begin{equation*}
J:=\bigl\{i \in I \mid \pair{\lambda}{\alpha_{i}^{\vee}}=0 \bigr\}, \qquad 
K:=\bigl\{i \in I \mid \pair{\mu}{\alpha_{i}^{\vee}}=0 \bigr\}. 
\end{equation*}

%
\subsection{Standard paths.}
\label{subsec:SP}

We consider the following condition \eqref{eq:SM} 
on $\pi \otimes \eta \in \SLS(\lambda) \otimes \SLS(\mu)$: 
%
%
\begin{equation} \label{eq:SM}
\begin{cases}
\text{there exist $x,\,y \in W_{\af}$ such that 
$x \sige y$ in $W_{\af}$}, \\[1mm]
\text{and such that $\PJ(x)=\kappa(\pi)$, $\PK(y)=\iota(\eta)$};
\end{cases} \tag{\sf SP}
\end{equation}
we set
\begin{equation*}
\SM(\lambda+\mu):=
  \bigl\{\pi \otimes \eta \in \SLS(\lambda) \otimes \SLS(\mu) \mid 
  \text{$\pi \otimes \eta$ satisfies condition \eqref{eq:SM}}
  \bigr\}.
\end{equation*}
%
%
\begin{thm} \label{thm:SMT}
The set $\SM(\lambda+\mu) \sqcup \{\bzero\}$ is stable under the action of 
the Kashiwara {\rm(}or, root{\rm)} operators $e_{i}$, $f_{i}$, $i \in I_{\af}$, on 
$\SLS(\lambda) \otimes \SLS(\mu)${\rm;} in particular, $\SM(\lambda+\mu)$ is 
a crystal with weights in $P_{\af}$. Moreover, $\SM(\lambda+\mu)$ is 
isomorphic as a crystal to $\SLS(\lambda+\mu)$. 
\end{thm}
\noindent
We will give a proof of Theorem~\ref{thm:SMT} in Section~\ref{sec:prf-SMT}. 

%
\subsection{Defining chains.}
\label{subsec:DC}
%
%
\begin{dfn} \label{dfn:DC}
Let $\pi=(x_{1},\,\dots,\,x_{s}\,;\,\ba) \in \SLS(\lambda)$ and 
$\eta=(y_{1},\,\dots,\,y_{p}\,;\,\bb) \in \SLS(\mu)$. 
A defining chain for $\pi \otimes \eta$ is 
a sequence $x_{1}',\,\dots,\,x_{s}',\,y_{1}',\,\dots,\,y_{p}'$
of elements in $W_{\af}$ satisfying the condition: 
%
%
\begin{equation} \label{eq:DC}
\begin{cases}
x_{1}' \sige \cdots \sige x_{s}' \sige y_{1}' \sige \cdots \sige y_{p}' 
 \quad \text{\rm in $W_{\af}$}; \\[1.5mm]
\PJ(x_{u}')=x_{u} 
 \quad \text{\rm for $1 \le u \le s$}; \\[1.5mm]
\PK(y_{q}')=y_{q}
 \quad \text{\rm for $1 \le q \le p$};
\end{cases}
\tag{\sf DC}
\end{equation}
we call $x_{1}'$ (resp., $y_{p}'$) the initial element 
(resp., the final element) of this defining chain. 
\end{dfn}
%
%
\begin{prop} \label{prop:DC}
Let $\pi \in \SLS(\lambda)$ and $\eta \in \SLS(\mu)$. 
Then, $\pi \otimes \eta \in \SM(\lambda+\mu)$ if and only if 
there exists a defining chain for 
$\pi \otimes \eta \in \SLS(\lambda) \otimes \SLS(\mu)$.
\end{prop}
\noindent
We will give a proof of Proposition~\ref{prop:DC} in Section~\ref{subsec:prf-DC}. 

Now, let $\eta=(y_{1},\,\dots,\,y_{p}\,;\,\bb) \in \SLS(\mu)$. 
For each $x \in W_{\af}$ such that 
$\kappa(\eta) = y_{p} \sige \PK(x)$, 
we define a specific lift $\Deo{\eta}{x} \in W_{\af}$ of 
$\iota(\eta) = y_{1} \in (\WKu)_{\af}$ as follows. 
Since $y_{p} \sige \PK(x)$ by the assumption, 
it follows from Proposition~\ref{prop:Deo} that 
$\Lif{x}{y_{p}}$ has the minimum element 
$\min \Lif{x}{y_{p}}=:\ti{y}_{p}$.
Similarly, since 
$y_{p-1} \sige y_{p} = \PK(\ti{y}_{p})$, 
it follows again from Proposition~\ref{prop:Deo} that 
$\Lif{\ti{y}_{p}}{y_{p-1}}$ has 
the minimum element $\min \Lif{\ti{y}_{p}}{y_{p-1}} = :\ti{y}_{p-1}$. 
Continuing in this way, we obtain 
$\ti{y}_{p}$, $\ti{y}_{p-1}$, $\dots$, $\ti{y}_{1}$. 
Namely, these elements are defined by 
the following recursive procedure (from $p$ to $1$): 
%
%
\begin{equation} \label{eq:lift0}
\begin{cases}
\ti{y}_{p} := \min \Lif{x}{y_{p}}, & \\[1mm]
\ti{y}_{q} := \min \Lif{\ti{y}_{q+1}}{y_{q}}
& \text{for $1 \le q \le p-1$}.
\end{cases}
\end{equation}
Finally, we set 
\begin{equation} \label{eq:lift0a}
\Deo{\eta}{x}:=\ti{y}_{1};
\end{equation}
this element is sometimes called the initial direction of $\eta$ 
with respect to $x$. 
%
%
\begin{prop} \label{prop:DC2}
Let $\pi \in \SLS(\lambda)$ and $\eta \in \SLS(\lambda)$. 
Then, $\pi \otimes \eta \in \SM(\lambda+\mu)$ {\rm(}or equivalently, 
there exists a defining chain for $\pi \otimes \eta \in 
\SLS(\lambda) \otimes \SLS(\mu)$ by Proposition~\ref{prop:DC}{\rm)}
if and only if $\kappa(\pi) \sige 
\PJ(\Deo{\eta}{x})$ for some 
$x \in W_{\af}$ such that 
$\kappa(\eta) \sige \PK(x)$. 
\end{prop}
\noindent
We will give a proof of Proposition~\ref{prop:DC2} in Section~\ref{subsec:prf-DC2}. 
%
%
\subsection{Demazure crystals in terms of standard paths.}
\label{subsec:dem-SP}

We set $S:=\bigl\{i \in I \mid \pair{\lambda+\mu}{\alpha_{i}^{\vee}}=0 \bigr\}=\J \cap \K$. 
For each $x \in W_{\af}$, we define $\SM_{\sige x}(\lambda+\mu) 
\subset \SM(\lambda+\mu)$ to be the image of 
$\SLS_{\sige x}(\lambda+\mu)=
 \bigl\{ \psi \in \SLS(\lambda+\mu) \mid \kappa(\psi) \sige \PS(x) \bigr\}$
under the isomorphism $\SLS(\lambda + \mu) \cong \SM(\lambda+\mu)$ 
in Theorem~\ref{thm:SMT}. 
%
%
\begin{thm} \label{thm:Dem}
Let $x \in W_{\af}$. 
For $\pi \otimes \eta \in \SLS(\lambda) \otimes \SLS(\mu)$, 
the following conditions \eqref{eq:Da}, \eqref{eq:Db}, and \eqref{eq:Dc} 
are equivalent{\rm:}
%
%
\begin{equation} \label{eq:Da}
\pi \otimes \eta \in \SM_{\sige x}(\lambda+\mu); \tag{\sf D1}
\end{equation}
%
%
\begin{equation} \label{eq:Db}
\begin{cases}
\text{\rm there exists a defining chain for $\pi \otimes \eta$ whose final element, say $y$, } \\[1mm]
\text{\rm satisfies the condition that $\PS(y) \sige \PS(x)$};
\end{cases} \tag{\sf D2}
\end{equation}
%
%
\begin{equation} \label{eq:Dc}
\kappa(\eta) \sige \PK(x) \quad \text{\rm and} \quad
\kappa(\pi) \sige \PJ(\Deo{\eta}{x}). \tag{\sf D3}
\end{equation}
Therefore, we have
\begin{equation*}
\SM_{\sige x}(\lambda+\mu) = 
\bigl\{ \pi \otimes \eta \mid 
 \text{\rm $\eta \in \SLS_{\sige x}(\mu)$ and 
 $\pi \in \SLS_{\sige \Deo{\eta}{x}}(\lambda)$}
\bigr\},
\end{equation*}
and hence {\rm(}see \eqref{eq:gch2}{\rm)}
\begin{equation}
\gch V_{x}^{-}(\lambda+\mu) = 
\sum_{ \eta \in \SLS_{\sige x}(\mu) } e^{\fwt(\wt(\eta))} q^{\qwt(\wt(\eta))} 
\underbrace{%
\sum_{ \pi \in \SLS_{\sige \Deo{\eta}{x}}(\lambda) } e^{\fwt(\wt(\pi))} q^{\qwt(\wt(\pi))}
}_{ \text{\rm $=\gch V_{\Deo{\eta}{x}}^{-}(\lambda)$ by \eqref{eq:gch2}} }. 
\end{equation}
\end{thm}
\noindent
We will give a proof of Theorem~\ref{thm:Dem} 
in Section~\ref{sec:prf-Dem}. 
%
%
\section{Semi-infinite Schubert varieties and their resolutions.}
\label{sec:SiSch}
%
%
\subsection{Geometric setting.}
\label{subsec:geo}

An (algebraic) variety is an integral separated scheme of finite type over $\BC$. 
Also, a pro-affine space is a product of finitely many copies of 
$\Spec \BC [ x_{m} \mid m \ge 0 ]$, equipped with a truncation morphism 
$\Spec \BC [ x_{m} \mid m \ge 0 ] \rightarrow \Spec \BC [ x_{m} \mid 0 \le m \le n ]$ 
for $n \gg 0$; by a morphism of pro-affine spaces, 
we mean a morphism of schemes that is also continuous with respect to 
the topology induced by the truncation morphisms 
(this topology itself is irrelevant to the Zariski topology).
For a $\BC$-vector space $V$, 
we set $\BP(V) := (V \setminus \{ 0 \}) / \BC^{\times}$. 
We usually regard $\BP(V)$ as an algebraic variety over $\BC$ 
when $\dim V < \infty$, or as an ind/pro-scheme when $\dim V = \infty$ 
in accordance with the topology of $V$.

For an algebraic group $E$, 
let $E\bra{z}$, $E\pra{z}$, and $E [z]$
denote the set of $\BC\bra{z}$-valued points, 
$\BC\pra{z}$-valued points, and $\BC[z]$-valued points of $E$, respectively; 
the corresponding Lie algebras are denoted by 
$\Fe\bra{z}$, $\Fe\pra{z}$, and $\Fe[z]$, 
respectively, with $E$ replaced by its German letter $\Fe=\Lie(E)$.
Also, we denote by $R(E)$ the representation ring of $E$. 

Recall that $G$ is a connected, 
simply-connected simple algebraic group over $\BC$; 
concerning the Lie algebra $\Fg=\Lie(G)$ and 
its untwisted affinization $\Fg_{\af}$, 
we use the notation of Section~\ref{sec:notation}. 
%
%

We have an evaluation map 
$\ev_{0} : G\bra{z} \longrightarrow G$ at $z=0$. 
Let $\Iw := \ev_{0}^{-1} ( B )$ be an Iwahori subgroup of $G\bra{z}$. 
Also, for each $i \in I_{\af}$, we have a minimal parahoric subgroup 
$\Iw \subset \Iw(i) \subset G\bra{z}$ corresponding to $\alpha_{i}$, 
so that $\Iw(i) / \Iw \cong \BP^{1}$. 
Note that both $G\bra{z}$ and $\Iw$ admit an action of $\Gm$ 
obtained by the scalar dilation on $z$; 
we denote the resulting semi-direct product groups by 
$\ti{G}\bra{z}$ and $\ti{\Iw}$, respectively.
The (finite) Weyl group $W$ of $\Fg$ is 
isomorphic to $N_{G}(H)/H$, and $Q^{\vee}$ is isomorphic to 
$H\pra{z}/H\bra{z}$, both of which fit in the following commutative 
diagram involving the (affine) Weyl group $W_{\af}$ of $\Fg_{\af}$: 
\begin{equation*}
\begin{CD}
0 @>>> Q^{\vee} @>>> W_{\af} @>>> W @>>> e \\
@.  @A{\cong}AA  @A{\cong}AA @A{\cong}AA \\
@. H\pra{z}/H\bra{z} @>>> N_{G\pra{z}}(H)/H\bra{z} @>>> 
   \underbrace{N_{G\pra{z}}(H)/H\pra{z}}_{= N_{G}(H)/H}, \\
\end{CD}
\end{equation*}
where the first row is exact, and the rightmost isomorphism 
in the second row holds since $N_{G\pra{z}}(H) \cong 
N_{G\pra{z}}(H\pra{z}) \cong (N_{G}(H))\pra{z}$. 
In particular, we have a lift $\dot{w} \in N_{G\pra{z}}(H)$ 
for each $w \in W_{\af}$. 
Now, for $x \in W_{\af}$ and $i \in I_{\af}$, we set
\begin{equation}
s_{i} \ast x := 
\begin{cases} 
 x & \text{if $s_{i} x < x$}, \\
 s_{i} x & \text{if $s_{i} x > x$}, 
\end{cases}
\end{equation}
where we denote by $>$ the ordinary Bruhat order on $W_{\af}$. 
%
Then, the set $W_{\af}$ becomes a monoid, 
which we denote by $\Wm$, under the product $\ast$; 
this monoid is also obtained as a subset of 
the generic Hecke algebra associated to $(W_{\af},\,I_{\af})$ 
by setting $a_{i} = 1$ and $b_{i} = 0$ for $i \in I_{\af}$ 
in \cite[Sect.~7.1, Theorem]{Hum90}. 
%
%
%
\subsection{Semi-infinite flag manifolds.}
\label{subsec:SiFl}
Here we review two models of semi-infinite flag manifold associated to $G$, 
for which the basic references are \cite{FM99} and \cite{FFKM99}.

Let $L(\lambda)$ denote the (finite-dimensional) irreducible highest weight 
$\Fg$-module of highest weight $\lambda \in P^{+}$. 
Recall that for each $\lambda,\,\mu \in P^{+}$, 
we have a canonical embedding of 
irreducible highest weight $\Fg$-modules (and hence of $G$-modules) up to scalars:
%
%
\begin{equation} \label{HWtensor}
L(\lambda + \mu) \hookrightarrow L(\lambda) \otimes_{\BC} L(\mu).
\end{equation}
The embedding \eqref{HWtensor} induces an embedding
%
%
\begin{equation} \label{eq:emb}
L(\lambda+\mu) \otimes_{\BC} R \hookrightarrow  
  \bigl( L(\lambda) \otimes_{\BC} L(\mu) \bigr) \otimes_{\BC} R \cong 
  (L(\lambda) \otimes_{\BC} R) \otimes_{R} (L(\mu) \otimes_{\BC} R)
\end{equation}
for every commutative, associative $\BC$-algebra $R$. 
%
%
\begin{thm}[{\cite[1.1.2]{BG02}}] \label{thm:BG}
Let $\BK$ be a field containing $\BC$. 
The set of collections $\{ \ell^{\lambda} \}_{\lambda \in P^{+}}$ of 
one-dimensional $\BK$-vector subspaces $\ell^{\lambda}$ 
in $L(\lambda) \otimes_{\BC} \BK$ such that 
$\ell^{\lambda} \otimes_{\BK} \ell^{\mu} = \ell^{\lambda+\mu}$ 
for every $\lambda,\,\mu \in P^{+}$ 
{\rm(}under the embedding \eqref{eq:emb}{\rm)} is 
in bijection with the set of closed $\BK$-points of $G/B$.  
\end{thm}
For a $\Fg$-module $V$, we set $V\bra{z} := V \otimes_{\BC} \BC\bra{z}$ 
and $V\pra{z} := V \otimes_{\BC} \BC\pra{z}$. 
%
%
%
\begin{dfn} \label{dfn:DP}
Consider a collection $L = \{ L^{\lambda} \}_{\lambda \in P^{+}}$ of 
one-dimensional $\BC$-vector subspaces $L^{\lambda}$ in 
$L(\lambda)\bra{z} = L(\lambda) \otimes_{\BC} \BC\bra{z}$ 
(resp., $L(\lambda)\pra{z} = L(\lambda) \otimes_{\BC} \BC\pra{z}$). 
The datum $L$ is called a formal (resp., rational) 
Drinfeld-Pl\"ucker (DP for short) datum if 
for every $\lambda,\,\mu \in P^{+}$, the equality
\begin{equation} \label{eq:DP}
L^{\lambda + \mu} = L^{\lambda} \otimes_{\BC} L^{\mu}
\end{equation}
holds under the embedding \eqref{eq:emb}, 
where $L^{\lambda} \otimes_{\BC} L^{\mu}$ is considered to be 
its image under the map $L(\lambda)\bra{z} \otimes_{\BC} L(\mu)\bra{z} 
\rightarrow L(\lambda)\bra{z} \otimes_{\BC\bra{z}} L(\mu)\bra{z}$ 
(resp., $L(\lambda)\pra{z} \otimes_{\BC} L(\mu)\pra{z} 
\rightarrow L(\lambda)\pra{z} \otimes_{\BC\pra{z}} L(\mu)\pra{z}$); 
we sometimes refer to a collection $\{u^{\lambda}\}_{\lambda \in P^{+}}$ of 
nonzero elements $u^{\lambda} \in L^{\lambda}$, $\lambda \in P^{+}$, 
as a formal (resp., rational) DP datum. 
Let $\DP$ (resp., $\RDP$) denote the set of formal (resp., rational) DP data. 
\end{dfn}
%
%
\begin{rem} \label{rem:DP}
By the compatibility condition \eqref{eq:DP}, 
a DP datum $\{ L^{\lambda} \}_{\lambda \in P^{+}}$ is 
determined completely by a collection $\{u^{i}\}_{i \in I}$ 
of nonzero elements $u^{i} \in L^{\vpi_{i}}$ for $i \in I$. 
We call this collection $\{u^{i}\}_{i \in I}$ DP coordinates. 
\end{rem}

Let $L = \{ L^{\lambda} \}_{\lambda \in P^{+}} \in \DP$. 
We define $\deg L^{\lambda}$ to be the degree of 
a nonzero element in $L^{\lambda}$, 
viewed as an $L(\lambda)$-valued formal power series 
(if it is bounded). 
For each $\xi \in Q^{\vee,+}$, 
a DP datum of degree $\xi$ is a formal DP datum
$L = \{ L^{\lambda} \}_{\lambda \in P^{+}}$ 
such that $\deg L^{\lambda} \le \pair{\lambda}{\xi}$ 
for all $\lambda \in P^{+}$. 
For each $\xi \in Q^{\vee,+}$, 
let $\FQ_{G} (\xi)$ denote the set of formal DP data of degree $\xi$. 
Here we note that 
if $\xi,\,\zeta \in Q^{\vee,+}$ satisfy $\xi \le \zeta$, 
i.e., $\zeta - \xi \in Q^{\vee,+}$, 
then $\FQ_{G}(\xi) \subset \FQ_{G}(\zeta)$. 
We set $\FQ_{G} := \bigcup_{\xi \in Q^{\vee,+}} \FQ_{G}(\xi)$; 
observe that 
\begin{equation*}
\FQ_{G}(\xi) \subset \FQ_{G} \subset \DP \quad 
 \text{for each $\xi \in Q^{\vee,+}$}. 
\end{equation*}
Also, we remark that $\DP$ has a natural $G\bra{z}$-action, and 
that its subsets $\FQ_{G}(\xi)$, $\xi \in Q^{\vee,+}$, and $\FQ_{G}$ 
are stable under the action of $G$ on $\DP$. 
%
%
\begin{lem}\label{defbQ}
We have an embedding
\begin{equation} \label{Pemb}
\DP \ni 
  \{ L^{\lambda} \}_{\lambda \in P^{+}} \mapsto 
  \{ [ L^{\vpi_{i}} ] \} _{i \in I} \in 
  \prod_{i \in I} \BP (L(\vpi_{i})\bra{z}), 
\end{equation}
which gives the set $\DP$ a {\rm(}reduced{\rm)}
structure of an infinite type 
closed subscheme of 
$\prod_{i \in I} \BP (L(\vpi_{i})\bra{z})$. 
In particular, $\DP$ is separated.
\end{lem}

\begin{proof}
Because a DP datum $\{ L^{\lambda} \}_{\lambda \in P^{+}}$ is 
determined uniquely by $\{ L^{\vpi_{i}} \} _{i \in I}$ 
(see Remark~\ref{rem:DP}), the map above is injective. 
Also, condition \eqref{eq:DP} is equivalent to saying 
that the image of $L^{\lambda} \otimes_{\BC} L^{\mu}$ lies in 
the $\BC$-vector subspace $L(\lambda+\mu)\bra{z} \subset 
\bigl(L(\lambda) \otimes_{\BC} L(\mu) \bigr)\bra{z}$ 
for all $\lambda,\,\mu \in P^{+}$. 
This condition imposes infinitely many equations on 
$\prod _{i \in I} \BP(L(\vpi_{i})\bra{z})$ 
that define $\DP$ as its closed subscheme. 
Since each $\BP(L(\vpi_{i})\bra{z})$ is separated, 
so is $\DP$. This proves the lemma. 
\end{proof}
%
%
\begin{thm}[\cite{FM99}] \label{orbbQ}
The set of $G\bra{z}$-orbits in $\DP$ is 
labeled by $Q^{\vee,+}$. The codimension of 
the $G\bra{z}$-orbit corresponding to $\xi \in Q^{\vee,+}$ is 
equal to $2 \pair{\rho}{\xi}$. 
\end{thm}

\begin{cor} \label{orbits}
The set of $\Iw$-orbits in $\DP$ 
is in bijection with the set $W_{\af}^{\ge 0} = 
\bigl\{w t_{\xi} \mid w \in W,\,\xi \in Q^{\vee,+}\bigr\}$. 
\end{cor}

\begin{proof}
Apply (the consequence of) the Bruhat decomposition 
$G\bra{z} = \bigsqcup_{w \in W} \Iw \dot{w} \Iw$ 
to each $G\bra{z}$-orbit in $\DP$ described in Theorem~\ref{orbbQ}.
\end{proof}

For each $x = w t_{\xi} \in W_{\af}^{\ge 0}$, 
we denote by $\bO(x)$ the $\Iw$-orbit of $\DP$ that 
contains a unique $(H \times \Gm)$-fixed point corresponding to 
$\{z^{ \pair{\vpi_{i}}{\xi} } v_{w w_{\circ} \vpi_{i}} \} _{i \in I}$ 
(see Lemma~\ref{defbQ}), 
where for each $\lambda \in P^{+}$ and $w \in W$, 
we take and fix a nonzero vector $v_{w\lambda}$ 
of weight $w\lambda$ in $L(\lambda)$; 
note that the codimension of $\bO(x) \subset \ol{\bO(e)}$ is given by $\sell(x)$. 
We set $\DP ( x ) := \overline{\bO ( x )} \subset \DP$. 
For $x,\,y \in W_{\af}^{\ge 0}$, we have 
\begin{align*}
\DP (x) \subset \DP ( y ) 
 & \iff x \sige y \quad \text{(\cite[Sect.~5.1]{FFKM99})} \\
 & \iff x w_{\circ} \sile  yw_{\circ} \quad 
   \text{(see \cite[Lecture 13]{Pet97})}.
\end{align*}
Also, we have $\DP(e) = \DP$ by inspection; in fact, 
$e \in W_{\af}^{\ge 0}$ is the minimum element 
in the semi-infinite Bruhat order restricted to $W_{\af}^{\ge 0}$.

For $\xi \in Q^{\vee,+}$ and $x \in W_{\af}^{\ge 0}$, 
we set $\FQ_G (x,\,\xi) := \FQ_G (\xi) \cap \DP(x)$, 
and for $x \in W_{\af}^{\ge 0}$, we set 
$\FQ_G (x) := \bigcup_{\xi \in Q^{\vee,+}} \FQ_G (x,\,\xi)$.

For each $\xi \in Q^{\vee,+}$, we have an embedding
\begin{equation*}
\imath_{\xi} : 
 \DP \ni 
 \{ L^{\lambda} \}_{\lambda \in P^{+}} 
 \bigl(= \{ u^{\lambda} \}_{\lambda \in P^{+}} \bigr) \mapsto 
 \{ z ^{ \pair{\lambda}{\xi} } u^{\lambda} \}_{\lambda \in P^{+}} \in \DP.
\end{equation*}
Thus we have its direct limit $\RDP \cong \varinjlim_{\xi} \DP$, 
on which we have an action of $G\pra{z}$.
By its construction, the embedding $\imath_{\xi}$ is 
$G \bra{z}$-equivariant, and sends 
the $\Iw$-orbit $\bO(x)$ to $\bO(xt_{\xi})$.

Now, by Lemma~\ref{defbQ}, we have a $G \bra{z}$-equivariant line bundle 
$\CO_{\DP}(\vpi_{i})$ obtained by the pullback of 
the $i$-th $\CO(1)$ through \eqref{Pemb}. 
For $\lambda = \sum_{i \in I} m_{i} \vpi_{i} \in P$, 
we set $\CO_{\DP}(\lambda) := 
\bigotimes_{i \in I} \CO_{\DP}(\vpi_{i})^{\otimes m_{i}}$ 
(as a tensor product of line bundles). 
Also, for each $x \in W_{\af}^{\ge 0}$, 
we have the restriction $\CO_{\DP(x)}(\lambda)$ 
obtained through \eqref{Pemb}, which is $\ti{\Iw}$-equivariant. 
Similarly, we have $( B \times \Gm )$-equivariant line bundles 
$\CO_{\FQ_G(x,\xi)} ( \lambda )$ and $\CO_{\FQ_G(x)} ( \lambda )$ 
by further pullbacks (the latter is $G [z]$-equivariant 
whenever $x = t_{\zeta}$ for some $\zeta \in Q^{\vee,+}$); we set
%
%
\begin{equation} \label{eq:Hn}
H^{n} ( \FQ_G(x),\, \CO_{\FQ_G(x)}(\lambda) ) := 
 \varprojlim _{\xi} H ^{n} ( \FQ_G(x,\,\xi),\,\CO_{\FQ_G(x,\,\xi)}(\lambda) )
 \quad \text{for $n \in \BZ_{\ge 0}$}. 
\end{equation}

Let $\lambda \in P^{+}$. 
As explained in \cite[Theorem 1.6]{Kat16}, 
the restricted dual of the Demazure submodule $V_{e}^{-} (- w_{\circ} \lambda )$ 
(see \eqref{eq:dem}) of the extremal weight 
$U_{\q}(\Fg_{\af})$-module $V(- w_{\circ} \lambda )$ 
of extremal weight $- w_{\circ} \lambda$ gives rise 
to an integrable $\Fg [z]$-module (by taking the classical limit $\q \to 1$), 
called the global Weyl module; we denote it by $W(\lambda)$. 
Here we note that global Weyl modules carry natural gradings 
arising from the dilation of the $z$-variable. 
%
%
%
\begin{thm}[{\cite[Proposition~5.1]{BF14c}}] \label{BFmain}
For $\lambda \in P$, we have
\begin{equation*}
H^{0} ( \cDP, \, \CO_{\DP}(\lambda) ) \cong 
\begin{cases}  
 W(\lambda)^{\ast} & \text{\rm if $\lambda \in P^{+}$}, \\[1mm]
 \{ 0 \} & \text{\rm otherwise}, 
\end{cases}
\end{equation*}
where $\cDP$ denotes the open dense $G\bra{z}$-orbit in $\DP$.
\end{thm}
%
%
\begin{thm}[{\cite[Theorem~3.1]{BF14a} and \cite[Theorem~4.12]{Kat16}}] \label{Kmain}
For each $\lambda \in P$ and $x \in W_{\af}^{\ge 0}$, 
we have
\begin{equation*}
\gch \dot{H}^{n}( \FQ_G ( x ),\,\CO_{\FQ_G(x)} ( \lambda ) ) = 
\begin{cases}
\gch V_{x}^{-}( - w_{\circ} \lambda ) 
  & \text{\rm if $\lambda \in P^{+}$ and $n=0$}, \\
0 & \text{\rm otherwise},
\end{cases}
\end{equation*}
where $\dot{H}^{n}( \FQ_G ( x ),\,\CO_{\FQ_G(x)} ( \lambda ) )$ 
denotes the space of $\Gm$-finite vectors in \eqref{eq:Hn}.
\end{thm}
%
%
\begin{rem} \label{Kmain-rem}
The assertion of \cite[Theorem 4.12]{Kat16} is only for $x \in W$. 
The assertion of the form above easily follows from the isomorphism 
$\DP ( x ) \cong \DP ( x t_{\xi} )$ 
for $x \in W$ and $\xi \in Q^{\vee, +}$, 
which is obtained by means of $\imath_{\xi}$.
\end{rem}
%
%
\begin{lem} \label{vample}
For each $\lambda \in P^{++}:=\sum_{i \in I} \BZ_{> 0} \vpi_{i}$, 
the sheaf $\CO_{\DP}(\lambda)$ is very ample.
\end{lem}

\begin{proof}
The proof is exactly the same as 
that of \cite[Corollary 2.7]{Kat16}.
\end{proof}
%
%
\subsection{Bott-Samelson-Demazure-Hansen towers.}
\label{subsec:BSDH}

In this subsection, we construct two kind of (pro-)schemes of 
infinite type, which can be thought of as ``resolutions" of 
$\DP(x)$ for $x \in W_{\af}^{\ge 0}$, and study their properties.
%
%
\begin{lem}[{\cite[(2.4.1)]{Mac03}}] \label{lem:Mac}
If $\xi \in Q^{\vee}$ is a strictly antidominant coweight, i.e., 
$\pair{\alpha_{i}}{\xi} < 0$ for all $i \in I$, 
then $\ell(t_{\xi})=-2\pair{\rho}{\xi}$, 
and $\ell(wt_{\xi})=\ell(t_{\xi}) - \ell(w)$ for all $w \in W${\rm;}
hence we have $\ell(wt_{\xi})=-\sell(wt_{\xi})$ 
for all $w \in W$. 
\end{lem}

%
\begin{lem} \label{tweakw}
\mbox{}
\begin{enu}
\item 
$\sell ( x t_{- 2 m \rho^{\vee}} ) + 
 \sell ( t _{2 m \rho^{\vee}} ) = \sell ( x )$ 
for all $x \in W_{\af}$ and $m \in \BZ$. 

\item 
There exists $m_{0} \ge 0$ such that 
$-\sell ( x t_{- 2 m \rho^{\vee}} ) = 
 \ell ( x t_{- 2 m \rho^{\vee}} )$ 
for all $m \ge m_{0}$.
\end{enu}
\end{lem}

\begin{proof}
Part (1) is obvious from the definition of $\sell(\,\cdot\,)$. 
For part (2), write $x$ as $x =wt_{\xi} \in W_{\af}$ 
for some $w \in W$ and $\xi \in Q^{\vee}$, and 
take $m_{0} \ge 0$ such that $\xi-2m_{0}\rho^{\vee}$ is 
strictly antidominant. Then we see from Lemma~\ref{lem:Mac}
that $-\sell ( x t_{- 2 m \rho^{\vee}} ) = 
 \ell ( x t_{- 2 m \rho^{\vee}} )$ for all $m \ge m_{0}$. 
This proves the lemma. 
\end{proof}
%
%
\begin{rem} \label{rem:ell}
Keep the setting of Lemma~\ref{tweakw}\,(2). We have 
\begin{equation*}
\ell ( x t_{- 2 (m_{0}+m) \rho^{\vee}} ) = 
\ell ( x t_{- 2 m_{0} \rho^{\vee}} ) + m\ell(t_{-2\rho^{\vee}}) = 
\ell ( x t_{- 2 m_{0} \rho^{\vee}} ) + \ell(t_{-2m\rho^{\vee}})
\end{equation*}
for all $m \ge 0$. 
\end{rem}

In what follows, we fix $x \in W_{\af}^{\ge 0}$ unless stated otherwise.
For this $x$, we take $m_{0} \ge 0$ as in Lemma~\ref{tweakw}\,(2), 
and fix reduced expressions
%
%
\begin{equation} \label{eq:reduced}
x t_{- 2 m_{0} \rho^{\vee}} = 
 s_{i'_1} s_{i'_2} \cdots s_{i'_{\ell'}} \qquad \text{and} \qquad
t_{- 2 \rho^{\vee}} = 
 s_{i''_1} s_{i''_2} \cdots s_{i''_{\ell}},
\end{equation}
where $i'_1,\,\ldots,\,i'_{\ell'},\,i''_1,\,\ldots,\,i''_{\ell} \in I_{\af}$, 
with $\ell'=\ell(x t_{- 2 m_{0} \rho^{\vee}})$ and $\ell=\ell(t_{-2\rho^{\vee}})$. 
We concatenate these sequences periodically to obtain an infinite sequence
\begin{equation} \label{eq:bi}
\bi = (
\underbrace{i'_1,\,i'_2,\,i'_3,\,\ldots,\,i'_{\ell'}}_{ \text{for $x t_{- 2 m_{0} \rho^{\vee}}$} },\,
\underbrace{i''_1,\,i''_2,\,\ldots,\,i''_{\ell}}_{\text{for $t_{-2\rho^{\vee}}$}},\,
\underbrace{i''_1,\,i''_2,\,\ldots,\,i''_{\ell}}_{\text{for $t_{-2\rho^{\vee}}$}},\,i''_1,\,\ldots) 
\in I_{\af}^{\infty}, 
\end{equation}
and write it as: $\bi= (i_{1},\,i_{2},\,\ldots) \in I_{\af}^{\infty}$; 
remark that $s_{i_{1}}s_{i_{2}} \cdots s_{i_{k}}$ 
is reduced for all $k \ge 0$. For $k \in \BZ_{\ge 0}$, 
we set $\bi_{k} := (i_1,\,i_2,\,\ldots,\,i_{k})$. 

Let $k \in \BZ_{\ge 0}$, and let $\bj = (i_{j_1},\,\dots,\,i_{j_t})$ 
be a subsequence of $\bi_{k}$, where $1 \le j_{1} < \cdots < j_{t} \le k$. 
We set $\sigma (\bj) = \sigma_{k}(\bj) : = 
\bigl\{1,\,2,\,\ldots,\,k\bigr\} \setminus 
\bigl\{j_{1},\,\dots,\,j_{t}\bigr\}$.  
We identify a subsequence $\bj$ of $\bi_{k}$ 
with a subsequence $\bj'$ of $\bi_{k'}$ if and only if 
$\sigma_{k}(\bj) = \sigma_{k'}(\bj')$ (as subsets of $\BZ_{>0}$); 
namely, if $k' \ge k$, then 
$\bj = (i_{j_1},\,\dots,\,i_{j_t}) \subset \bi_{k}$ and 
$\bj' = (i_{j_1},\,\dots,\,i_{j_t},\,i_{k+1},\,\dots,\,i_{k'}) 
\subset \bi_{k'}$ are identified. 
Thus, we identify a subsequence $\bj$ of $\bi_{k}$ 
with a subsequence of $\bi$ by taking the limit in 
$\varinjlim_{k} \bi_{k} = \bi$. 

Let $k \in \BZ_{\ge 0}$, and let 
$\bj = (i_{j_1},\,i_{j_2},\,\ldots,\,i_{j_{t}})$ 
be a subsequence of $\bi_{k}$. We set 
\begin{equation*}
x(\bj ; k) := 
s_{i_{j_1}} \ast s_{i_{j_2}} \ast \cdots \ast s_{i_{j_t}} \in \Wm.
\end{equation*}
\begin{rem}
Let $k' \in \BZ_{\ge 0}$ be such that $k' \ge k$. 
Because the sequence $\bj$ above is identified with the subsequence 
$(i_{j_1},\,\dots,\,i_{j_t},\,i_{k+1},\,\dots,\,i_{k'})$ of $\bi_{k'}$, 
we have
\begin{equation*}
x(\bj ; k') := 
\underbrace{s_{i_{j_1}} \ast \cdots \ast s_{i_{j_t}}}_{=x(\bj ; k)} \ast 
s_{i_{k+1}} \ast \cdots \ast s_{i_{k'}} \in \Wm.
\end{equation*}
\end{rem}

%
\begin{lem} \label{stab}
Let $m \in \BZ_{\ge 0}$, and let $\bj$ be 
a subsequence of $\bi_{ \ell' + m \ell}$. 
Then, there exists $m_1 \ge m$ such that
$x ( \bj ; \ell' + m' \ell ) = 
 x ( \bj ; \ell' + m'' \ell) \cdot t_{-2 (m' -m'') \rho^{\vee}}$
for every $m' \ge m'' \ge m_{1}$. In particular, the element
$x ( \bj ) := x ( \bj ; \ell' + m' \ell ) \cdot 
 t_{2(m'+m_{0})\rho^{\vee}}  \in W_{\af}$ 
 does not depend on the choice of $m' \ge m_{1}$.
\end{lem}

\begin{proof}
We first note that 
\begin{equation*}
x ( \bj ; \ell' + m' \ell ) = 
x ( \bj ; \ell' + m'' \ell) \ast 
\overbrace{
\underbrace{ ( s_{i''_1} \ast s_{i''_2} \ast \cdots \ast s_{i''_{\ell}} ) }_{%
\text{corresponding to $t_{-2\rho^{\vee}}$}}
\ast \cdots \ast
\underbrace{ ( s_{i''_1} \ast s_{i''_2} \ast \cdots \ast s_{i''_{\ell}} ) }_{%
\text{corresponding to $t_{-2\rho^{\vee}}$}}
}^{ \text{$(m'-m'')$ times} }. 
\end{equation*}
Since $y \ast s_{i} = y s_{i}$ if and only if 
$\ell ( y s_{i} ) =  \ell ( y ) + 1$ for 
$y \in W_{\af}$ and $i \in I_{\af}$, 
it suffices to show that there exists $m_{1} \ge m$ such that 
\begin{equation} \label{eq:ell-0}
\ell ( x ( \bj ; \ell' + m'' \ell) \cdot t_{-2 n \rho^{\vee}} ) = 
\ell ( x ( \bj ; \ell' + m'' \ell) ) + \ell ( t_{-2 n \rho^{\vee}} ) 
\end{equation}
for all $n > 0$ and $m'' \ge m_{1}$. 
Let $k \in \BZ_{\ge 0}$ be such that $k \ge m$. 
Since $x ( \bj ; \ell' + k \ell) = 
x ( \bj ; \ell' + m \ell) \ast t_{-2(k-m)\rho^{\vee}}$
and $\ell(t_{-2(k-m)\rho^{\vee}}) = (k-m)\ell$, 
we see that 
%
%
\begin{equation} \label{eq:ell-1}
\ell ( x ( \bj ; \ell' + k \ell) ) \ge (k - m) \ell; 
\end{equation}
note that $\ell (y \ast y') \ge 
\max \bigl\{ \ell ( y ), \, \ell ( y' ) \bigr\}$ 
for $y,\,y' \in W_{\af}$, as is verified by induction. 
Now, for each integer $k \ge m$, we set 
\begin{equation*}
d_{k}:= \underbrace{\ell ( t_{- 2 \rho^{\vee}} )}_{=\ell} - 
\bigl\{
\ell ( x ( \bj ; \ell' + ( k + 1 ) \ell) ) - 
\ell ( x ( \bj ; \ell' + k \ell) ) \bigr\}; 
\end{equation*}
observe that $d_{k} \in \BZ_{\ge 0}$. 
Also, for $k \ge m$, we have
\begin{equation*}
(k-m)\ell = 
\ell ( x ( \bj ; \ell' + k \ell) ) - 
\ell ( x ( \bj ; \ell' + m \ell) ) + 
\sum_{k'=m}^{k-1} d_{k'}. 
\end{equation*}
If $d_{k} > 0$ for infinitely many $k \ge m$, 
then $(k-m)\ell > \ell ( x ( \bj ; \ell' + k \ell) )$ 
for $k \gg m$, which contradicts \eqref{eq:ell-1}. 
Hence we deduce that $d_{k} > 0$ only for finitely many $k \ge m$. 
Thus, if we set $m_{1}:=\max \bigl\{ k \ge m \mid d_{k} > 0 \bigr\}$, 
then \eqref{eq:ell-0} holds. This proves the lemma. 
\end{proof}
%
%
\begin{lem} \label{si-refl}
For each $y,\,y' \in W_{\af}$ such that $y \sile y'$, 
there exists $m_{2} \in \BZ_{\ge 0}$ such that 
$y t_{-2m\rho^{\vee}} \ge y' t_{-2m\rho^{\vee}}$
in the ordinary Bruhat order on $W_{\af}$ for all $m \ge m_{2}$. 
\end{lem}

\begin{proof}
It suffices to prove the assertion in the case that 
$y \edge{\beta} s_{\beta}y = y'$ for some $\beta \in \prr$. 
Here we see from \cite[Corollary 4.2.2]{INS} that 
$\beta$ is either of the following forms: 
\begin{center}
(i) $\beta = \alpha$ with $\alpha \in \Delta^{+}$; \qquad
(ii) $\beta = \alpha+\delta$ with $-\alpha \in \Delta^{+}$. 
\end{center}
Moreover, if $y = w t_{\xi}$ with $w \in W$ and $\xi \in Q^{\vee}$, 
then $\gamma:=w^{-1}\alpha \in \Delta^{+}$ in both cases above. 
Also, it follows from \cite[Proposition~A.1.2]{INS} that 
%
%
\begin{equation} \label{eq:ell}
\ell(w) = 
 \begin{cases}
  \ell(wr_{\gamma})-1 & \text{in case (i)}, \\[1mm]
  \ell(wr_{\gamma})-1+2 \pair{\rho}{\gamma^{\vee}} & \text{in case (ii)}. 
 \end{cases}
\end{equation}

If we set $\zeta:=\xi-2m\rho^{\vee}$ for $m \in \BZ$, then 
$y t_{- 2 m \rho^{\vee}} = w t _{\xi - 2 m \rho^{\vee}} = w t_{\zeta}$, and 
\begin{align}
s_{\beta} (y t_{- 2 m \rho^{\vee}}) & = 
 s_{\alpha+k\delta} w t_{\zeta}, 
 \quad \text{where $k=0$ in case (i), and $k=1$ in case (ii)}, \nonumber \\
 & = wt_{\zeta} s_{t_{-\zeta}w^{-1}\alpha+k\delta} 
   = (y t_{- 2 m \rho^{\vee}}) s_{\gamma + n \delta}, 
   \quad \text{with $n:=k+ \pair{\gamma}{\zeta}$}. \label{eq:n}
\end{align}
Therefore, in case (i) (resp., case (ii)), 
we deduce from \cite[Proposition~5.1\,(1) (resp., (2)) with $v=e$]{LNSSS}, 
together with equalities \eqref{eq:ell} and \eqref{eq:n} that 
$y't_{-2m\rho^{\vee}} = s_{\beta}y t_{-2m\rho^{\vee}} = 
(y t_{-2m\rho^{\vee}}) s_{\gamma + n \delta} < y t_{-2m \rho^{\vee}}$ for all $m \gg 0$. 
This proves the lemma. 
\end{proof}
%
%
\begin{lem} \label{BSDH-cover}
For each $y \in W_{\af}$ such that $y \sige x$, 
there exist $k \in \BZ_{\ge 0}$ and a subsequence $\bj$ of $\bi_{k}$ 
such that $y = x(\bj)$.
\end{lem}

\begin{proof}
By Lemma~\ref{si-refl}, there exists $m \gg 0$ such that 
the assertion of Lemma~\ref{tweakw}\,(2) holds for $m_{0}+m$ instead of $m_{0}$ 
(see also Remark~\ref{rem:ell}), and such that
\begin{equation*}
yt_{- 2 (m_{0}+m) \rho^{\vee}} < 
x t _{- 2 (m_{0}+m) \rho^{\vee}} = 
x t _{- 2 m_{0} \rho^{\vee}} \cdot t _{- 2 m \rho^{\vee}};
\end{equation*}
note that $\bi_{\ell'+m\ell}$ gives a reduced expression for 
$x t _{- 2 (m_{0}+m) \rho^{\vee}}$. By the Subword Property 
(see, e.g., \cite[Theorem~2.2.2]{BB}), 
$y t_{- 2 (m_{0}+m) \rho^{\vee}}$ is obtained 
as a subexpression of the reduced expression for
$x t_{- 2 (m_{0}+m) \rho^{\vee}}$ 
corresponding to $\bi_{\ell'+m\ell}$. 
Namely, there exists a subsequence $\bj=(i_{j_1},\,\dots,\,i_{j_t})$ 
of $\bi_{\ell'+m\ell}$ such that 
\begin{equation*}
y t_{- 2 (m_{0}+m) \rho^{\vee}} = 
s_{i_{j_1}} \cdots s_{i_{j_t}} = 
s_{i_{j_1}} \ast \cdots \ast s_{i_{j_t}} = 
x(\bj,\ell'+m\ell). 
\end{equation*}
Let us take $m_{1} \gg m$ as in Lemma~\ref{stab}. 
Then, we see by Remark~\ref{rem:ell} 
that for $m' \ge m_{1}$, 
\begin{align*}
x(\bj,\ell'+m'\ell) & = 
 x(\bj,\ell'+m\ell) \ast t_{-2(m'-m)\rho^{\vee}} \\
& = 
 x(\bj,\ell'+m\ell) \cdot t_{-2(m'-m)\rho^{\vee}} =
 y t_{- 2 (m'+m_{0}) \rho^{\vee}}.
\end{align*}
Thus, we obtain $y = x(\bj)$, as desired.
This proves the lemma. 
\end{proof}

For each $i \in I_{\af}$ and $y \in W_{\af}$, 
we define a map
%
%
\begin{equation} \label{eq:qiy}
q_{i,y} : \Iw(i) \times^{\Iw} \DP(y) \rightarrow \RDP, \quad 
(p,\,L) \mapsto pL. 
\end{equation}
%
%
\begin{lem}\label{one-step}
Let $y \in W_{\af}^{\ge 0}$, and $i \in I_{\af}$. 
If $s_{i} y \sige y$, then the map $q_{i,y}$ induces 
a $\BP^{1}$-fibration
$\Iw(i) \times^{\Iw} \DP(y) \rightarrow \DP(y)$, which 
we also denote by $q_{i,y}$. 
If $e \preceq s_i y \preceq y$, then 
the map $q_{i,y}$ induces a birational map 
$\Iw(i) \times^{\Iw} \DP(y) \rightarrow \DP(s_{i}y)$, 
which we also denote by $q_{i,y}$.
In both cases, the map $q_{i,y}$ is proper.
\end{lem}

\begin{proof}
If $s_{i} y \sige y$, then 
the action of $\Iw(i)$ stabilizes 
$\bO(y) \cup \bO ( s_{i} y ) \subset \DP ( y )$. 
Hence taking the closure in $\DP$ implies that $\DP ( y )$ admits 
an $\Iw ( i )$-action. Therefore, the assertions hold in this case 
since $\Iw ( i )/ \Iw \cong \BP^{1}$.

If $e \preceq s_{i} y \preceq y$, then 
$q_{i,y}^{-1} ( \bO ( s_{i} y ) )$ contains 
$\Iw \dot{s}_i \Iw \times ^{\Iw} \bO (y)$. 
Here we deduce from Lemmas~\ref{lem:si} and \ref{lem:dmd} that 
if $y' \in W_{\af}$ satisfies $y' \sig y$, 
then $s_{i}y' \sig s_{i}y$; in particular, 
we have $\sell( s_i y ) <  \sell( s_i y' )$. 
Hence we have $\Iw \dot{s}_i \Iw \times^{\Iw} \bO ( y ) = 
q_{i,y}^{-1} ( \bO ( s_{i} y ) )$. 
In addition, the unipotent one-parameter subgroup of $\Iw$ 
corresponding to $\alpha_{i}$ gives an isomorphism 
$\BA^{1} \times \dot{s}_i \bO ( y ) \cong \bO ( s_{i} y )$. 
Therefore, $q_{i,y}$ is birational. 
Also, by applying the same observation for $s_{i}y$ instead of $y$, 
we deduce that $q_{i,y}$ is obtained as the restriction of 
the $\BP^{1}$-fibration $q_{i,s_{i}y}$ to a closed subscheme.
Hence $q_{i,y}$ defines a proper map. 
This proves the lemma. 
\end{proof}

For each $k \in \BZ_{\ge 0}$, 
we set $x(k) := s_{i_{k}} s_{i_{k-1}} \cdots s_{i_{1}} x$. 
We claim that 
%
%
\begin{equation} \label{eq:xk}
\sell( x(k+1) ) = \sell(x (k) ) + 1 \quad \text{for all $k \ge 0$}.
\end{equation}
Indeed, since $x(k+1) = s_{i_{k}}x(k)$ for $k \ge 0$, 
we see by \eqref{eq:simple} that 
$\sell( x(k+1) ) = \sell(x (k) ) \pm 1$ for each $k \ge 0$. 
Therefore, it suffices to show that 
\begin{equation} \label{eq:xm}
\sell(x((m-m_{0})\ell + \ell')) = 
 \sell(x(0)) + (m-m_{0})\ell + \ell' \qquad 
 \text{for all $m \ge m_{0}$}; 
\end{equation}
note that $x(0) = x$. 
We see by the definition that
$x((m-m_{0})\ell + \ell') = (xt_{-2m\rho^{\vee}})^{-1}x = 
t_{2m\rho^{\vee}}$. Hence we compute: 
\begin{align}
\sell(x((m-m_{0})\ell + \ell')) & 
  = \sell(t_{2m\rho^{\vee}}) = 2\pair{\rho}{2m\rho^{\vee}} 
  = -m \sell(t_{-2\rho^{\vee}}) \nonumber \\ 
& = m \ell(t_{-2\rho^{\vee}}) = m\ell \quad \text{by Lemma~\ref{lem:Mac}}. 
\label{eq:xk1}
\end{align}
Also, by Lemma~\ref{tweakw}\,(1), we have
$\sell(xt_{-2m_{0}\rho^{\vee}}) + \sell(t_{2m_{0}\rho^{\vee}}) = \sell(x)$. 
Here we deduce that 
$\sell(xt_{-2m_{0}\rho^{\vee}}) = - \ell(xt_{-2m_{0}\rho^{\vee}}) = -\ell'$ 
by Lemma~\ref{tweakw}\,(2), and that 
$\sell(t_{2m_{0}\rho^{\vee}}) = m_{0}\ell$. 
Hence we obtain $\sell(x) = -\ell' + m_{0}\ell$. 
Combining this equality with \eqref{eq:xk1} shows 
\eqref{eq:xm}, as desired. 

Now, we set
%
%
\begin{equation} \label{defbQi}
\DP ( \bi_k ) := 
 \Iw ( i_1 ) \times^{\Iw} \Iw ( i_2 ) \times^{\Iw} \cdots 
 \times^{\Iw} \Iw ( i_k ) \times^{\Iw} \bO ( x ( k ) ).
\end{equation}
We also define its ambient space
\begin{equation*}
\bQ^{\#}_G ( \bi_k ) := 
 \Iw ( i_1 ) \times^{\Iw} \Iw ( i_2 ) \times^{\Iw} \cdots \times^{\Iw} 
 \Iw ( i_k ) \times^{\Iw} \DP ( x ( k ) ).
\end{equation*}
Since $s_{i_{k}} x ( k ) = x ( k - 1 )$ and 
$x ( k ) \sige x ( k - 1 )$ for each $k \ge 1$ (see \eqref{eq:xk}), 
we have an $\ti{\Iw}$-equivariant embedding 
$\bO ( x ( k - 1 ) ) \hookrightarrow \Iw ( i_k ) \times^{\Iw} \bO ( x ( k ) )$ 
by the latter case of Lemma~\ref{one-step}, 
and hence an $\ti{\Iw}$-equivariant embedding 
$\DP ( \bi_{k-1} ) \hookrightarrow \DP ( \bi_k )$ for each $k \ge 1$. 
By infinite repetition of these embeddings, 
we obtain a scheme of infinite type
%
%
\begin{equation} \label{Qunion}
\DP ( \bi ) := \varinjlim_k \DP ( \bi_k ) = \bigcup_{k \ge 0} \DP ( \bi_k ),
\end{equation}
endowed with an $\ti{\Iw}$-action. Also, the multiplication of 
components yields an $\ti{\Iw}$-equivariant morphism
\begin{equation*}
\sfm : \DP ( \bi ) \longrightarrow \DP ( x ). 
\end{equation*}
Similarly, we have 
$\sfm_{k} : \DP^{\#}( \bi_k ) \rightarrow \DP ( x )$ 
for each $k \ge 0$. The natural inclusion 
$\DP ( \bi_k ) \hookrightarrow \DP^{\#} ( \bi_k )$ 
yields a map $\DP(\bi) \to \DP^{\#} ( \bi_k )$ 
by taking the  product of factors at position $> k$ from 
the left in \eqref{defbQi}. 
Namely, we collect the maps
\begin{equation*}
\DP ( \bi_{k'} ) \ni 
( p_1,\,\ldots,\,p_{k'},\,L) 
 \mapsto 
(p_1,\,\ldots,\,p_k,\,p_{k+1} \cdots p_{k'}L) \in \DP^{\#}( \bi_k )
\end{equation*}
for $k'> k$ through \eqref{Qunion}, 
where $p_{j} \in \Iw ( i_j )$, $1 \le j \le k'$, and 
$L \in \bO (x(k'))$; here each closed point of $\DP ( \bi_{k'} )$ is 
an equivalence class with respect to the $\Iw^{k'}$-action, 
and the map above respects equivalence classes.
This yields a factorization of $\sfm$ 
through arbitrary $\sfm_{k}$ in such a way that
%
%
\begin{equation} \label{factorBSDH}
\DP ( \bi ) \rightarrow 
\DP^{\#}( \bi_{k'} ) \rightarrow 
\DP^{\#} ( \bi_k ) \rightarrow \DP ( x ) 
\end{equation}
for each $k' \ge k$. This also yields an inclusion
\begin{equation*}
\DP ( \bi ) \hookrightarrow 
\DP^{\#} ( \bi ) := \varprojlim_k \bQ^{\#}_G ( \bi_k ),
\end{equation*}
which fits into the following commutative diagram of 
$\ti{\Iw}$-equivariant morphisms for $k < k'$: 
\begin{equation*}
\xymatrix{
\DP(\bi) \ar@{^{(}->}[rr] \ar[d]_{\sfm} \ar[drr] \ar[drrr] & & 
\DP^{\#}(\bi) \ar[dll] \ar[d] \ar[dr] & \\
\DP(x) & & \DP^{\#}(\bi_k) \ar[ll]^{\sfm_k} & 
\DP^{\#}(\bi_{k'}). \ar[l]
}
\end{equation*}
%
%
\begin{lem} \label{up-normal}
The scheme $\DP ( \bi )$ is separated and normal.
\end{lem}

\begin{proof}
The scheme $\DP ( \bi )$ is an inclusive union of countably many 
open subschemes each of which is isomorphic to a pro-affine space bundle 
over a finite successive $\BP^{1}$-fibrations. 
Since each of such a space is separated, we deduce the desired separatedness. 

Also, by its construction, each $\DP ( \bi_{k} )$ is a union of 
pro-affine spaces labeled by subsequences of $\bi_k$ 
(so that $\DP (\bi_{k})$ is covered by a total of $2^{k}$-copies 
of open cover consisting of pro-affine spaces). 
Thus, $\DP ( \bi )$ is a union of countably many pro-affine spaces. 
Since all of these pieces are normal, we deduce the desired normality.
This proves the lemma. 
\end{proof}

For each subsequence $\bj \subset \bi_{k}$, 
we obtain an $\Iw$-stable pro-affine space 
$\bO ( \bj ) \subset \DP ( \bi_{k} )$ 
by replacing $\Iw ( i_{j} )$, $1 \le j \le k$, with 
$\Iw \dot{s}_{i_j} \Iw$ (resp., $\Iw$)
if $i_{j} \in \bj$ (resp., $i_{j} \not\in \bj$) in \eqref{defbQi}; 
we refer to $\bO ( \bj )$ as the stratum corresponding 
to a subsequence $\bj$ of $\bi_k$.
%
%
\begin{lem}\label{k-stratum}
\mbox{}
\begin{enu}
\item For each $k \in \BZ_{\ge 1}$, it holds that 
$\DP ( \bi_k ) = \bigsqcup _{\bj \subset \bi_k} \bO ( \bj )$. 
\item 
Let $k \in \BZ_{\ge 1}$, and 
let $\bj$ and $\bj'$ be subsequences of $\bi_k$. 
Then, $\bO ( \bj ) \subset \ol{\bO(\bj')}$ in $\DP ( \bi_{k} )$ 
if and only if $\bj \subset \bj' (\subset \bi_k)$.
\end{enu}
\end{lem}

\begin{proof}
The proofs of the assertions are straightforward by the definitions.
\end{proof}

If we regard $\bO ( \bj )$ as a locally closed subscheme of 
$\DP ( \bi )$ via \eqref{Qunion} for $\bj \subset \bi_k$, 
then its images in $\DP^{\#}( \bi_k )$ and $\DP^{\#} ( \bi_{k'})$ 
for $k < k'$ are isomorphic through \eqref{factorBSDH}. 
Therefore, we can freely replace $\sfm$ with $\sfm_{k}$ 
when we analyze a single stratum in $\DP(\bi)$. 
We set $\DP(\bj) := \overline{\bO(\bj)}$, 
where the closure is taken in $\DP ( \bi )$.

\begin{lem}
We have $\DP ( \bi ) = \bigsqcup_{\bj} \bO ( \bj )$, 
where $\bj$ runs over all subsequences of $\bi$ 
such that $|\sigma(\bj)| < \infty$.
\end{lem}

\begin{proof}
Since $\DP (\bi) = \bigcup_{k \ge 1} \DP ( \bi_k )$, 
the assertion is a consequence of Lemma \ref{k-stratum} and 
the consideration just above.
\end{proof}

\begin{lem}
The map $\sfm$ is surjective.
\end{lem}

\begin{proof}
By Lemma~\ref{BSDH-cover}, each fiber of $\sfm$ 
along an $(H \times \Gm)$-fixed point is nonempty. 
Because applying the $\tilde{\Iw}$-action to 
the set of $(H \times \Gm)$-fixed points exhausts $\DP ( x )$ 
by Corollary~\ref{orbits}, we conclude that $\sfm$ is 
surjective, as required. This proves the lemma. 
\end{proof}

\begin{lem}
The map $\sfm_k$ is $\ti{\Iw}$-equivariant, 
birational, and proper for each $k \ge 1$.
\end{lem}

\begin{proof}
Since $\sell ( x ( k ) ) - \sell ( x ) = k$ for each $k \ge 1$, 
repeated application of Lemma~\ref{one-step} shows 
that $\sfm_k$ is birational and proper for each $k \ge 1$; 
this is because $\sfm_k$ is obtained as 
the base change of the composite of 
morphisms of type $q_{i_j,\,x(j)}$ for $1 \le j \le k$.
This proves the lemma. 
\end{proof}

The map $\sfm$ is also $\ti{\Iw}$-equivariant and 
birational since the embedding $\DP(\bi) \subset \DP^{\#} ( \bi )$ is open.
%
%
\subsection{Cohomology of line bundles over $\DP$.}
\label{subsec:coh}

In this subsection, keeping the setting of 
the previous subsection, 
we also assume that $x=e$, the identity element; 
in this case, we can (and do) take the $m_{0}$ 
(in Lemma~\ref{tweakw}\,(2)) to be $0$, so that we have $\ell'=0$ 
(see \eqref{eq:reduced}). 
In particular, we have $\DP ( x ) = \DP ( e ) = \DP$.

For each $\lambda \in P$, 
the $\ti{\Iw}$-equivariant line bundle $\CO_{\DP ( x( k ) )} ( \lambda )$ 
induces a line bundle over $\DP(\bi_{k})$ 
for each $k \in \BZ_{\ge 0}$. 
%
%
\begin{prop} \label{H-range}
For each $\lambda \in P$, we have an isomorphism
\begin{equation*}
H^{0}( \DP(\bi), \,\CO_{\DP(\bi)}(\lambda) ) \cong W(\lambda)^{\ast}
\end{equation*}
of $\ti{\Iw}$-modules. In particular, 
the left-hand side carries a natural structure of graded $\Fg[z]$-module.
\end{prop}

\begin{proof}
We adopt the notation of the previous subsection, 
with $x=e$, $m_{0}=0$, and $\ell' = 0$. Also, since 
$\ell ( t_{- 2 \rho^{\vee}} ) = 
\ell ( w_{\circ} ) + \ell ( w_{\circ} t_{- 2 \rho^{\vee}} )$,
we can rearrange the reduced expression 
$(i''_1,\,\ldots,\,i''_{\ell}) = (i_1,\,\ldots,\,i_{\ell})$ 
for $t_{- 2 \rho^{\vee}}$, if necessary, 
in such a way that the first $\ell(w_{\circ})$-entries 
$(i_1,\,\ldots,\,i_{\ell(w_{\circ})})$ 
give rise to a reduced expression for $w_{\circ}$. 

The map $\sfm$ factors as $\sfm' \circ \sfm''$, 
where $\sfm''$ is the map obtained by collapsing 
the first $\ell ( w_{\circ} )$-factors in $\DP ( \bi )$ as: 
%
%
\begin{equation} \label{eq:mapm}
\DP ( \bi ) = \Iw ( i_1 ) \times^{\Iw} \cdots \times^{\Iw} 
\Iw (i_{\ell ( w_{\circ} )}) \times ^{\Iw} X 
\rightarrow G \bra{z} \times ^{\Iw} X,
\end{equation}
with $X$ a certain scheme admitting an $\ti{\Iw}$-action, 
and $\sfm'$ is the natural map 
$G\bra{z} \times^{\Iw} X \rightarrow \DP$ 
induced by the action. 
In particular, we deduce that 
$\sfm_{\ast} \CO _{\DP ( \bi )} ( \lambda )$ 
admits a $\ti{G}\bra{z}$-action, and hence 
that the space of global sections of 
$\sfm_{\ast} \CO _{\DP ( \bi )} ( \lambda )$ is 
a $\Fg[z]$-module. 

Now, for each $m \in \BZ_{\ge 0}$, we have 
$\DP ( \bi_{\ell(w_{\circ}) + \ell m} ) \subset \DP (\bi)$.
Also, for each $k \in \BZ_{\ge 0}$, 
we have an inclusion $\DP ( \bi_k ) \hookrightarrow \DP ( \bi_{k + \ell} )$ 
by sending the stratum $\bO(\bj)$ corresponding to $\bj \subset \bi_k$ 
to the one corresponding to $\bj_{+} \subset \bi_{k + \ell}$, 
which is obtained by shifting all the entries by $\ell$ 
(so that $\{1,\,2,\,\ldots,\,\ell\} \cap \bj_{+} = \emptyset$ and 
$|\sigma_{k+\ell} (\bj_{+})| = |\sigma_{k} (\bj)| + \ell$); 
this corresponds to twisting $\Img \sfm = \DP$ to 
$\DP ( t_{2 \rho^{\vee}} )$. It follows that the $\Gm$-action on 
the line bundle $\CO_{\DP ( \bi_k )} ( \lambda )$ is 
twisted by $t_{2 \rho^{\vee}}$, whose actual effect is 
of degree $- 2 \pair{\lambda}{\rho^{\vee}}$. 
In particular, the structure map 
$\pi_{m} : \DP ( \bi_{\ell(w_{\circ})+\ell m} ) 
\rightarrow \bigl\{ \mathrm{pt} \bigr\}$ for $m \in \BZ_{\ge 0}$ 
factors as: 
%
%
\begin{equation} \label{factor-pi}
\pi_{m} \cong 
    \bigl( ( \pi_{1})^{\ast} \bigr)^{m} ( \pi_{0} ),
\end{equation}
where the maps $\pi_0$ and $\pi_1$ are projections
\begin{align}
\pi_{0} & : 
 \Iw ( i_{1} ) \times^{\Iw} \Iw ( i_{2} ) \times^{\Iw} \cdots \times^{\Iw} 
 \Iw ( i_{\ell(w_{\circ})}) \times^{\Iw} 
 \bO( w_{\circ} ) \rightarrow \bigl\{ \mathrm{pt} \bigr\}, \nonumber \\
\pi_{1} & : 
 \Iw ( i_{1} ) \times^{\Iw} \Iw ( i_{2} ) \times^{\Iw} \cdots \times^{\Iw} 
 \Iw ( i_{\ell} ) / \Iw 
 \rightarrow \bigl\{ \mathrm{pt} \bigr\}, \label{pi1}
\end{align}
and $(\pi_{1})^{\ast}$ means the pullback obtained 
by identifying $\bigl\{\mathrm{pt}\bigr\}$ with $\Iw / \Iw \in \Iw ( i_{\ell} ) / \Iw$ 
in \eqref{pi1}.

For each $i \in I_{\af}$ and an $\Iw$-module $M$, we define
$\bD_{i}(M)$ to be the space of $(H \times \Gm)$-finite vectors in 
$H^{0} ( \Iw (i) / \Iw, \, \Iw ( i ) \times^{\Iw} M^{\ast} )$; 
this can be thought as a left exact endo-functor 
in the category of $(H \times \Gm)$-semi-simple $\FI$-modules. We set
\begin{equation*}
\bD_{- 2 \rho^{\vee}} := 
\bD_{s_{i_1}} \circ \bD_{s_{i_2}} \circ \cdots \circ \bD_{s_{i_{\ell}}}
\quad \text{and} \quad 
\bD_{w_{\circ}} := 
\bD_{s_{i_1}} \circ \bD_{s_{i_2}} \circ \cdots \circ \bD_{s_{i_{\ell ( w_{\circ} )}}}.
\end{equation*}
Then, successive application of 
the Leray spectral sequence to \eqref{factor-pi}, 
together with the fact that the $\Gm$-twist commutes 
with the whole construction, yields
\begin{equation*}
H^0 ( \DP (\bi) ,\, \CO _{\DP ( \bi )} ( \lambda )) 
\stackrel{\cong}{\longrightarrow} 
\varprojlim_{m} \bD_{- 2 \rho^{\vee}}^{m} 
 \left( \bD_{w_{\circ}} 
   \left( \BC [ \bO( w_{\circ} t_{2 m \rho^{\vee}} ) ] \otimes 
          \BC _{- t_{- 2 m \rho^{\vee}} \lambda} \right) \right), 
\end{equation*}
where we have used the equality 
$w_{\circ} t_{2 m \rho^{\vee}} w_{\circ} 
= t_{- 2 m \rho^{\vee}}$.  

In view of \eqref{eq:mapm} with $X$ 
replaced by $\bO( w_{\circ} t_{2 m \rho^{\vee}})$, 
Theorem~\ref{BFmain} implies
\begin{equation*}
\bD_{w_{\circ}} \left( 
 \BC [ \bO ( w_{\circ} t_{2 m \rho^{\vee}} ) ] \otimes 
 \BC_{- t_{- 2 m \rho^{\vee}} \lambda} \right) \cong W ( \lambda )^{\ast},
\end{equation*}
where the grading of the right hand side is shifted by 
$- 2m \pair{\lambda}{\rho^{\vee}}$. 
Hence, \cite[Theorem~4.13]{Kat16} (cf. \cite[Lemma~2.8]{Kas05} and 
\cite[Corollary~4.8]{Kat16}) implies
\begin{equation*}
\bD_{- 2 \rho^{\vee}}^m \left( \bD_{w_{\circ}} \left( 
  \BC [ \bO ( w_{\circ} t_{2 m \rho^{\vee}} ) ] \otimes 
  \BC_{- t_{- 2 m \rho^{\vee}} \lambda} \right) \right) \cong W(\lambda)^{\ast}
\end{equation*}
without grading shift. Therefore, we conclude
\begin{equation*}
H^{0} ( \DP ( \bi ), \CO _{\DP(\bi)}(\lambda)) \cong W(\lambda)^{\ast}
\end{equation*}
for the choice of $\bi$ fixed at the beginning of the proof.
The assertion for a general $\bi$ follows 
from the fact that 
\begin{equation*}
\bD_{- 2 \rho^{\vee}}^m \cong \bD_{s_{j_1}} \circ \cdots \circ \bD_{s_{j_{m\ell}}}
\end{equation*}
holds for an arbitrary reduced expression $s_{j_1} \cdots s_{j_{m\ell}}$ for
$t_{- 2 m \rho^{\vee}}$. This completes the proof of the proposition. 
\end{proof}
%
%
\begin{thm}[Kneser-Platonov; see, e.g., {\cite[Proof of Theorem 5.8]{Gille}}] \label{w-approx} 
The subset $G[z] \subset G\bra{z}$ is dense.
\end{thm}

\begin{proof}
The original claim is that $G\pra{z}$ is generated by 
the set of $\BC\pra{z}$-valued points of the unipotent groups of $G$. 
Because the set of $\BC(z)$-valued points in a one-dimensional unipotent group is 
dense in the set of $\BC\pra{z}$-valued points in the sense that
we can approximate the latter by the former up to an arbitrary order of $z$, 
it follows that we can approximate $G\pra{z}$ by elements of $G(z)$ 
up to an arbitrary order of $z$. Since such an approximation of an element of 
$G\bra{z}$ is achieved by elements that are regular at $z = 0$, 
we conclude the assertion. 
\end{proof}
%
%
\begin{thm} \label{normal}
We have $\sfm_{\ast} \CO _{\DP(\bi)} \cong \CO_{\DP}$. 
In particular, the scheme $\DP$ is normal.
\end{thm}

\begin{proof}
We (can) employ the same reduced expression for $t_{-2\rho^{\vee}}$ 
as in the proof of Proposition~\ref{H-range}; recall the last sentence of the proof. 
The pullback defines a map $\sfm^{\ast} \CO_{\DP} \rightarrow \CO_{\DP(\bi)}$, 
whose adjunction in turn yields $\CO _{\DP} \to \sfm_{\ast} \CO_{\DP(\bi)}$. 
From this, by twisting by $\CO _{\DP} ( \lambda )$ for some $\lambda \in P^{+}$, 
we obtain the following short exact sequence: 
\begin{equation*}
0 \rightarrow \CO_{\DP} (\lambda) \rightarrow 
  \sfm_{\ast} \CO_{\DP(\bi)} (\lambda) \rightarrow 
  \coker \rightarrow 0, 
\end{equation*}
from which we deduce a $\Fg [z]$-module inclusion
\begin{equation*}
H^0 ( \DP, \CO _{\DP} ( \lambda ) ) \hookrightarrow 
H^0 ( \DP, \sfm_{\ast} \CO _{\DP ( \bi )} ( \lambda ) ) \cong 
H^0 ( \DP ( \bi ), \CO _{\DP ( \bi )} ( \lambda ) ),
\end{equation*}
by taking their global sections. 
The rightmost one is isomorphic to $W(\lambda)^{\ast}$ 
by Proposition~\ref{H-range}. 
In particular, we have algebra homomorphisms:
\begin{align*}
\bigoplus _{\lambda \in P^{+}} 
  \Gamma ( \DP, \CO_{\DP} ( \lambda ) ) 
& \subset 
\bigoplus _{\lambda \in P^{+}} 
 \Gamma ( \DP, \sfm_{\ast} \CO_{\DP (\bi)} ( \lambda ) ) \\
& \cong \bigoplus _{\lambda \in P^{+}} \Gamma ( \DP (\bi), \CO_{\DP (\bi)} ( \lambda ) ) \\
& \cong \bigoplus _{\lambda \in P^{+}} W ( \lambda )^{\ast};
\end{align*}
let us denote the leftmost one by $R'_{G}$ and the rightmost one by $R_{G}$. 
Since $\DP(\bi)$ is normal, we deduce that $R_{G}$ is normal 
when localized with respect to homogeneous elements in $W(\vpi_{i})^{\ast}$ 
for each $i \in I$. For the same reason, $R_{G}$ is an integral domain. 
The ring structure of $R_{G}$ is induced by the unique (up to scalar) 
$\Fg [z]$-module map
%
%
\begin{equation} \label{d-mult}
W(\lambda + \mu) \longrightarrow W (\lambda) \otimes_{\BC} W ( \mu ), \qquad 
\lambda,\,\mu \in P^{+}, 
\end{equation}
of degree zero. In view of \cite[Proof of Theorem 3.3]{Kat16}, 
we deduce that the multiplication map $W(\lambda)^{\ast} \otimes W (\mu)^{\ast} \rightarrow 
W (\lambda + \mu)^{\ast}$ is surjective since \eqref{d-mult} is injective. 
Therefore, $R_{G}$ is a normal ring generated by terms of primitive degree 
(see, e.g., \cite[Chap.~I\hspace{-1pt}I, Exerc.~5.14]{Har77}; 
cf. \cite[Proof of Theorem 3.3]{Kat16}).

From the above, it suffices to prove $R_{G}' = R_{G}$. For this purpose, 
it is enough to prove that the associated graded ring of 
the projective coordinate ring $R''_{G}$ of $\DP$, 
which is arising from its structure of a closed subscheme of 
$\prod_{i \in I} \BP(L(\vpi_{i})\bra{z})$, contains $R_{G}$ 
(see, e.g., \cite[Sect.~2.6]{EGA1} for convention). 
Recall that the projective coordinate ring of 
$\BP ( L(\vpi_{i})\bra{z} )$ is $\bigoplus_{n \ge 0} S^{n} \bigl( L(\vpi_{i})[z]^{\ast}\bigr)$, 
where $S^{n}V$ denotes the $n$-th symmetric power of a vector space $V$. 
Thanks to the surjectivity of multiplication map of $R_G$, it is further reduced to 
seeing that for each $i \in I$, every element of the part $W(\vpi_{i})^{\ast}$ 
of degree $\vpi_{i}$ in $R_{G}$ is written as the quotient of an element of 
$\prod_{i \in I} S^{ \pair{\lambda}{\alpha_{i}^{\vee}} } \bigl( L ( \vpi_{i} ) [z] ^{\ast} \bigr)$ 
by some power of monomials in elements of 
$L ( \vpi_{j} )[z]^{\ast} \subset W ( \vpi_{j} )^{\ast}$, $j \in I$ 
(note that this condition is particularly apparent in types $A$ and $C$ 
since $W ( \vpi_{i} ) = L ( \vpi_{i} ) [z]$ for each $i \in I$).

By \cite[Proof of Theorem 3.1]{BF14c}, 
each $\FQ_{G}(\xi)$, $\xi \in Q^{\vee, +}$, is a projective variety with rational singularities. 
By the Serre vanishing theorem \cite[Th\'eor\`eme 2.2.1]{EGA3-1} applied to 
the ideal sheaf that defines $\FQ_{G}(\xi)$ 
(inside the product of finite-dimensional projective spaces in 
$\BP ( L(\vpi_{i})\bra{z} )$ obtained by bounding the degree; cf. \cite[(2.1)]{Kat16}), 
the restriction map
%
%
\begin{equation} \label{surj-rest}
\bigotimes_{i \in I} H^{0} \bigl( \BP ( L(\vpi_{i})\bra{z}) , 
\CO(\pair{\lambda}{\alpha_{i}^{\vee}}) \bigr) \rightarrow 
H^{0} ( \FQ_{G}(\xi), \CO_{\FQ_{G}(\xi)} (\lambda) ) 
\end{equation}
is surjective for sufficiently large $\lambda \in P^{+}$, 
where we have used the fact that
\begin{equation*}
\bigotimes _{i \in I} 
  H^{0}( \BP ( L(\vpi_{i})\bra{z} ), \mathcal{F}_{i}) \cong 
  H^{0}\left( \prod_{i \in I} \BP ( L (\vpi_{i})\bra{z} ), \boxtimes_{i \in I} \mathcal{F}_{i} \right)
\end{equation*}
holds for vector bundles $\mathcal{F}_{i}$ of finite rank on $\BP( L(\vpi_{i})\bra{z} )$.
%
%
\begin{claim} \label{H-inj}
For a given degree $n \in \BZ_{\le 0}$, 
we can choose $\lambda \in P^{+}$ and $\xi \in Q^{\vee,+}$ sufficiently large 
in such a way that for every $m \in \BZ_{>0}$, the restriction map
\begin{equation*}
W ( m \lambda )^{\ast} \subset W ( m \lambda )^{\vee} = 
 H^{0} ( \FQ_{G}, \CO _{\FQ_{G}} (m \lambda ) ) \longrightarrow 
 H^{0} ( \FQ_{G} ( \xi ), \CO_{\FQ_{G} (\xi)} ( m \lambda ) )
\end{equation*}
is injective at degree greater than or equal to $n$.
\end{claim}

\noindent
{\it Proof of Claim~\ref{H-inj}.}
Let us denote by $W (\vpi_{i} )^{\ast}_{\ge n} \subset W(\vpi_{i})^{\ast}$ and 
$W (\vpi_{i})_{\le -n} \subset W (\vpi_{i})$ 
the direct sum of the homogeneous components of $W(\vpi_{i})^{\ast}$ of 
degree greater or equal to $n$, and the the direct sum of 
the homogeneous components of $W ( \vpi_{i} )$
of degree less than or equal to $-n$, respectively. 
Also, let $R_{G}^{n}$ be the subring of $R_{G}$ generated by 
the $W (\vpi_{i} )^{\ast}_{\ge n}$, $i \in I$; every homogeneous component of $R_{G}$ of 
degree greater than or equal to $n$
is contained in $R_{G}^{n}$ by the surjectivity of multiplication map 
and the fact that $W(\lambda)^{\ast}$ is concentrated in nonpositive degrees. 
The value of a section of a line bundle over $\Proj R_{G}$ 
(our $\Proj$ here is the $P^{+}$-graded proj, by which we mean that 
the $H$-quotient of the subset of the affine spectrum in 
$\prod_{i \in I} \bigl( W(\vpi_{i} ) \setminus \{0\} \bigr)$) 
arising from $R_{G}^{n}$ at a point is determined completely 
by its image under the projection
\begin{equation*}
\pr : 
 \prod_{i \in I} \BP \left( \ha{W(\vpi_{i})} \right) \setminus Z 
 \longrightarrow 
 \prod_{i \in I} \BP \bigl( W (\vpi_{i} )_{\le - n} \bigr)
\end{equation*}
induced by the $\Fg [z]$-module surjection $W (\vpi_{i}) \to W (\vpi_{i})_{\le - n}$, 
where $Z$ denotes the loci in which $\pr$ is not well-defined; 
for the notation $\ha{W(\vpi_{i})}$, see Section~\ref{subsec:liealg}.
Thanks to Theorem~\ref{w-approx}, we deduce that
\begin{equation*}
\pr ( \FQ_{G} \setminus Z ) = \pr ( \DP \setminus Z )
\end{equation*}
as the set of closed points. 
Since the restriction of $\pr$ to $\FQ_{G}(\xi) \setminus Z$ 
for each $\xi \in Q^{\vee,+}$ is a morphism of Noetherian schemes, 
it follows that $\pr ( \FQ_{G}(\xi) \setminus Z )$ is 
a constructible subset of $\prod_{i \in I} \BP \bigl( W (\vpi_{i} )_{\le -n} \bigr)$. 
Moreover, the irreducibility of $\FQ_{G}( \xi )$ forces 
$\overline{\pr ( \FQ_{G} ( \xi ) \setminus Z )}$ to be irreducible. 
Therefore, the equality
\begin{equation*}
\pr ( \FQ_{G} \setminus Z ) = \bigcup_{\xi \in Q^{\vee,+}} \pr ( \FQ_{G}(\xi) \setminus Z )
\end{equation*}
implies that there exists some $\xi \in Q^{\vee,+}$ such that
%
%
\begin{equation} \label{Z-dense}
\pr ( \FQ_{G}(\xi) \setminus Z ) \subset \pr( \DP \setminus Z )
\end{equation}
is Zariski dense, since $\prod_{i \in I} \BP \left( W (\vpi_{i} )_{\le - n} \right)$ is a Noetherian scheme.

Thanks to \cite[Proposition 5.1]{BF14c} (cf. Theorem~\ref{Kmain}), 
we can find $\lambda$ (by replacing $\xi$ with a larger one if necessary) 
such that the assertion holds for $m = 1$. 
Now we assume the contrary to deduce the assertion for $m > 1$. 
Then, we have an additional equation 
on $\pr( \DP \setminus Z ) \subset 
\prod_{i \in I} \BP \bigl( W (\vpi_{i} )_{\le - n} \bigr)$
vanishing along $\FQ_{G}(\xi)$ 
by (taking sum of) the multiplication of $W(\lambda)^{\ast}_{\ge n}$. 
However, this is impossible in view of \eqref{Z-dense} 
since $\Proj R_{G}$ is reduced (and hence $R_{G}^{n}$ is integral). 
Thus we have proved Claim~\ref{H-inj}. \bqed

\vspace{3mm}

We return to the proof of Theorem~\ref{normal}. 
We fix $n \in \BZ_{\le 0}$ and $\beta \in Q^{\vee,+}$ such that Claim \ref{H-inj} holds. 
By replacing $\lambda$ if necessary to guarantee the surjectivity of the restriction map \eqref{surj-rest} 
with keeping the situation of Claim \ref{H-inj}, we deduce that all the maps
\begin{equation} \label{diag}
\begin{split}
\xymatrix{
\bigotimes _{i \in I} H^{0} \bigl( \BP(L(\vpi_{i})\bra{z}), \CO(\pair{\lambda}{\alpha_{i}^{\vee}}) \bigr) 
  \ar@{=}[d] \ar[r] 
  & H^{0}(\DP, \CO_{\DP}(\lambda)) 
  \ar@{^{(}->}[d] \\
\bigotimes _{i \in I} S^{ \pair{\lambda}{\alpha_{i}^{\vee}} } 
  \bigl( L(\vpi_{i})[z]^{\ast} \bigr) 
  \ar[r] \ar@{=}[d] 
  & W (\lambda)^{\ast} 
  \ar[d] \\
\bigotimes_{i \in I} H^{0} \bigl( \BP(L(\vpi_{i})\bra{z}), \CO( \pair{\lambda}{\alpha_{i}^{\vee}}) \bigr) 
 \ar[r] & 
 H^{0} ( \FQ_{G}(\xi), \CO_{\FQ_{G}(\xi)} ( \lambda ) ) 
}
\end{split}
\end{equation}
are surjective at degree $n$ from the commutativity of the diagram and 
the surjectivity of the bottom horizontal map. 
For a degree $n$ element $f \in W (\vpi_{i})^{\ast}$ and 
degree zero element $h_{j} \in L(\vpi_{j})^{\ast} \subset W(\vpi_{j})^{\ast}$ 
for each $j \in I$, we can choose sufficiently large integers $N_{i}$, $i \in I$, such that
\begin{equation*}
f \cdot \prod_{j \in I} h_{j}^{N_{j}} \in 
 S^{1+N_{i}} \bigl( L(\vpi_{i}) [z]^{\ast} \bigr) \cdot 
 \prod_{j \in I,\,j \ne i} S^{N_{j}} \bigl( L ( \vpi_{j} ) [z]^{\ast} \bigr) \subset R_{G}
\end{equation*}
as the corresponding claim is true after sending to the bottom line of \eqref{diag}.
This forces $W ( \vpi _{i})^{\ast}$ to be contained in the part of degree $\vpi_{i}$ of 
the associated graded ring of $R_{G}''$, as required. 
This completes the proof of the theorem.
\end{proof}
%
%
\begin{cor} \label{pcr}
The projective coordinate ring $R_{G}$ of $\DP$ arising from 
the embedding by means of the DP-coordinates is isomorphic to 
$\bigoplus_{\lambda \in P^{+}} W(\lambda)^{\ast}$.
\end{cor}
%
%
\begin{thm} \label{R-free}
The projective coordinate ring $R_{G}$ of $\DP$ in Corollary~\ref{pcr} is 
free over the polynomial algebra $A_{G}$ given by the lowest weight components 
with respect to the $H$-action.
\end{thm}

\begin{proof}
During this proof, we denote by $v_{\lambda} \in W ( \lambda )$ 
the unique $\Fh_{\af}$-eigenvector of weight $\lambda$ 
(which is determined up to a scalar, and is the specialization of 
the corresponding vector of $V ( \lambda )$ through $\mathsf{q} \to 1$).

For each $\lambda \in P^{+}$, we set
\begin{equation*}
\BC [\BA^{(\lambda)}] := \bigotimes_{i \in I} 
  S^{ \pair{\lambda}{\alpha_{i}^{\vee}} } \BC [z] \cong 
  \bigotimes_{i \in I} 
  \BC [ z_{1}^{(i)},\,\ldots,\,
  z_{ \pair{\lambda}{\alpha_{i}^{\vee}} }^{(i)} ]^{ \FS_{\pair{\lambda}{\alpha_{i}^{\vee}}} }.
\end{equation*}
By the results \cite{FL07,Naoi12}, due to Fourier-Littelmann and Naoi, 
we know that $W (\lambda)$ is a free module over $\BC [\BA^{(\lambda)}]$, 
and the $\lambda$-isotypical component of $W ( \lambda )$ is free of rank one. 
Here we define the polynomial algebra $A_{G}$ by collecting the $(-\lambda)$-isotypical component of 
$W ( \lambda )^{\ast}$ for all $\lambda \in P^{+}$. It follows that the ring $A_{G}$ is of the form
\begin{equation*}
\bigotimes_{i \in I} S^{\bullet} \BC [z]^{\ast} \cong 
\bigoplus_{\lambda \in P^{+}} 
\bigotimes_{i \in I} S^{ \pair{\lambda}{\alpha_{i}^{\vee}} } \BC [z]^{\ast} \cong 
\bigoplus_{\lambda \in P^{+}} \BC [\BA^{(\lambda)}]^{\ast}.
\end{equation*}

Let $\psi \in W ( \mu )^{\ast}$ and $\xi \in \BC[\BA^{(\lambda - \mu)}]^{\ast} \subset A_G$, 
where $\lambda, \mu, \lambda - \mu \in P^+$, and assume that both of them are 
homogeneous with respect to $P$-weights and degrees. 
Then, we find a product of $\Fh _{\af}$-eigen PBW basis element 
$F_1 \in U ( \Fn^- [z] )$ and monomials $f_1, f_2 \in U ( \Fh [z] z ) \cong 
S^{\bullet} ( \Fh [z] z )$ such that $\psi ( F _1 f_1 v_{\mu} ) \neq 0$ and 
$\xi ( f_2 v_{(\lambda - \mu)}) \neq 0$ by the Poincar\'e-Birkhoff-Witt theorem. 
It follows that $( \psi \cdot \xi) ( F_1 f_1 f_2 v_{\lambda} ) \neq 0$, 
since we need to collect the terms $F_1 m_1 v_{\mu} \otimes m_2 v_{\lambda - \mu}$, 
with $m_1 m_2 = f_1 f_2$, in the tensor product through 
the embedding $W ( \lambda ) \subset W ( \mu ) \otimes W ( \lambda - \mu )$. 
This means that $0 \neq \psi \cdot \xi \in W ( \lambda )^{\ast}$; 
in particular, the ring $R_G$ is torsion-free as an $A_G$-module.

Now, let us fix $i_0 \in I$ and $\lambda \in P^+$ 
so that $\pair{\lambda}{\alpha_{i_0}^{\vee}} = 0$, 
and set $\lambda_m := \lambda + m \vpi_{i_0}$ for each $m \in \BZ_{\ge 0}$. 
We also set $A_G^{i_0} := S^{\bullet} \BC [z]^{\ast} \subset A_G$ for the fixed $i_0$. 
For each $m, l \in \BZ_{\ge 0}$ such that $m \ge l$, 
we denote by $W ( \lambda; m,l )$ the space 
$(\BC [\BA^{(l\vpi_{i_0})}]^{\ast} \cdot W ( \lambda_{m-l} )^{\ast} )^{\ast}$, 
which is a $\Fb[z]$-submodule of $W ( \lambda_m )$. 
From this description, we have an inclusion
\begin{equation*}
W ( \lambda; m,l ) \subset W ( \lambda; m,l-1 )
\end{equation*}
when $l > 0$. In particular, $W ( \lambda; m, l )^{\ast}$ is 
a quotient of $W ( \lambda; m, l - 1 )^{\ast}$.

By repeated use of (the dual of) the surjectivity of 
the multiplication map of $R_G$, we have an embedding:
\begin{equation} \label{W-emb-fw}
\Phi : W ( \lambda_m ) \hookrightarrow W ( \varpi_{i_0} )^{\otimes m} \otimes W ( \lambda ).
\end{equation}
For $0 \le l \le m$, let $\mathbb W ( \lambda; m, l )$ denote 
the linear span of pure tensors
\begin{equation} \label{pure-tensor}
\left( \bigotimes_{j = 1}^{m} v_{i_0,j} \right) \otimes v \in 
 W ( \varpi_{i_0} )^{\otimes m} \otimes W ( \lambda )
\end{equation}
of $\Fh _{\af}$-eigenvectors in which at most $l$-elements of 
$\{v_{i_0,j}\}_{j = 1}^{m}$ is of the form $z^e v_{\vpi_{i_0}}$ 
for some $e \in \BZ_{\ge 0}$. If we denote by $W ( \lambda; m, l )'$ 
the preimage of $\BW ( \lambda; m, l )$ through $\Phi$, then we have
\begin{equation*}
W ( \lambda; m, l-1 )' \subset W ( \lambda; m, l )'
\end{equation*}
whenever $l > 0$. By construction, $\{ W ( \lambda; m, l )' \}_{0 \le l \le m}$ 
forms a $\BZ$-graded increasing filtration whose associated graded modules
\begin{equation*}
\gr_{l} W( \lambda_m ) := 
 W ( \lambda; m, l )' / W ( \lambda; m, l - 1 )', \qquad l \in \BZ_{\ge 0},
\end{equation*}
stratify $W ( \lambda_m )$.

Here, $W ( \varpi_{i_0} )^{\otimes m} \otimes W ( \lambda )$ admits 
a graded decomposition coming from the number of elements in 
$\{v_{i_0,j}\}_{j = 1}^{m}$ of the form $z^e v_{\varpi_{i_0}}$ 
for $e \in \BZ_{\ge 0}$ through \eqref{pure-tensor} and \eqref{W-emb-fw}. 
It follows that the space $\BW ( \lambda; m, l )$ is 
the annihilator of the subspace
\begin{equation} \label{expand-prod}
\sum_{w \in \FS_m} \sum_{a=l}^m w ( \BC [\BA^{(\varpi_{i_0})}]^{\otimes a} \otimes 
W ( \varpi_{i_0} )^{\otimes m-a} )^{\ast} \otimes 
W ( \lambda )^\ast \subset ( W ( \varpi_{i_0} ) ^{\otimes m} )^\ast \otimes 
W ( \lambda )^{\ast},
\end{equation}
where $\FS_m$ permutes the tensor factors of $W ( \varpi_{i_0} )^{\otimes m}$. 
Pulling back by $\Phi$, we deduce that $W ( \lambda; m, l-1 )'$ is 
the annihilator of the space \eqref{expand-prod} in $W ( \lambda_m )$ 
through the embedding \eqref{W-emb-fw}. Therefore, 
$W ( \lambda; m, l-1 )' \subset W ( \lambda_m )$ 
is exactly the annihilator of $\BC [\BA^{(l\vpi_{i_0})}]^{\ast} \cdot 
W ( \lambda_{m-l} )^{\ast} \subset W ( \lambda _m )^{*}$. 
It follows that we have a canonical isomorphism
\begin{equation*}
\gr_{l} W(\lambda_m) \cong W ( \lambda; m, l ) / W ( \lambda; m, l+1).
\end{equation*}
If we define a subquotient $M(\lambda; n)$ of $R_{G}$ 
for each $n \in \BZ_{\ge 0}$ by
\begin{equation*}
M( \lambda; n ) := \bigoplus _{l \ge 0} M ( \lambda; n )_l, \qquad 
M( \lambda; n )_l = \left( \gr_{l} W(\lambda_{n+l}) \right)^{\ast},
\end{equation*}
then $M(\lambda; n)$ admits an $A_G^{i_0}$-action.

From the construction of $M ( \lambda; n )_l$ through 
$\{W ( \lambda; m, l)\}_{m, l \ge 0}$, we deduce that 
$M ( \lambda; n )$ is generated by $M ( \lambda; n )_{0}$ 
as an $A_G^{i_0}$-module. Also, from the construction of $M ( \lambda; n )_l$ 
through $\{W ( \lambda; m, l)'\}_{m, l \ge 0}$, 
we deduce that the dual of the multiplication map is the natural map
\begin{equation*}
\BC [\BA^{(l \vpi_{i_0})}] \otimes \gr_{0} W(\lambda_{n}) \rightarrow \gr_{0} \, W(\lambda_{n+l})
\end{equation*}
of $\BC [\BA^{(\lambda_{n+l})}]$-modules. 
The $(\BC [\BA^{(m \vpi_{i_{0}})}] , \Fb [z] )$-modules $\gr_{l} W(\lambda_m)$, 
$0 \le l \le m$, stratify $W ( \lambda_{m} )$. In addition, 
the maximality of $W ( \lambda; m, l )'$ guarantees that 
each $\gr_{l} W(\lambda_m)$ is torsion-free as a $\BC [\BA^{(\lambda_m)}]$-module.

For each $\lambda \in P^{+}$, we can regard $W ( \lambda )$ 
as a module corresponding to a vector bundle (or a free sheaf) $\CW(\lambda)$ 
over $\BA^{(\lambda)}$, where $\BA^{(\lambda)}$ denotes $\Spec \BC [\BA^{(\lambda)}]$; 
its fiber is known to be the tensor product of local Weyl modules $W (\mu,\,x)$, 
where $\mu \in P^{+}$ and $x \in \BC$ runs over the configurations of
points determined by a point of $\BA^{(\lambda)}$ (see, e.g., \cite[Theorem 1.4]{Kat16}).

The spaces $\gr_{l} W(\lambda_m)$, $0 \le l \le m$, give 
torsion-free sheaves $\CW_l(\lambda_m)$ on $\BA^{(\lambda_{m})}$ that 
stratify $\CW( \lambda_m )$. Hence a section of $\CW_l (\lambda_m)$ is 
an equivalence class of the set of sections $\BA^{(\lambda_m)} \rightarrow \CW( \lambda_m )$ 
whose specialization to a general point gives an element of 
the tensor product of local Weyl modules
\begin{equation} \label{point-division}
v = \sum_{k} \bigotimes_{i,j} v_{i,j}^{(k)} \in 
 \bigotimes_{i \in I} \bigotimes_{j=1}^{\pair{\lambda}{\alpha_{i}^{\vee}}} 
 W ( \vpi_{i}, x_{i,j} )
\end{equation}
such that exactly $l$-elements 
in $\{v_{i_{0},j}^{(k)} \}_{j=1}^{m}$ are highest weight vectors, 
and the other vectors do not have a highest weight component for each $k$.

Since every two points in $\{x_{i,j}\}_{i,j}$ are 
generically distinct, each pure tensor in \eqref{point-division} divides 
$\{1,2,\ldots, m\}$ into two subsets $S_{1}$ and $S_{2}$, 
with $\# S_{1} = l$, so that $\{x_{i_{0},j}\}_{j\in S_{1}}$ carries 
a highest weight vector (i.e., $v_{i_{0},j}^{(k)} = z^{e}v_{\vpi_{i_{0}}}$ 
for some $e \in \BZ_{\ge 0}$) and $\{x_{i_{0},j}\}_{j\in S_{2}}$ carries 
a vector lying in non-highest-weight components (as a $\Fb$-module). 
The coordinates $\{x_{i_{0},j}\}_{j\in S_{1}}$ gives rise to 
the action of $\BC [\BA^{(l \vpi_{i_0} )}]$ on $\gr_{l} W(\lambda_m)$, 
while the coordinates $\{x_{i_{0},j}\}_{j\in S_{2}}$ 
gives rise to a $\BC [\BA^{(\lambda_{m-l})}]$-module structure on $\gr_{l} W(\lambda_m)$, 
and these two module structures are (mutually commutative and) distinct. 
It follows that every pair of elements of $\BC [\BA^{(l \vpi_{i_0} )}]$ and 
$\gr_{0} W(\lambda_{m-l})$ appears as a section in $\CW_l ( \lambda_{m} )$ 
after a generic localization. This particularly gives us the $\FS_{m}$-action on 
$\BC [\BA^{(l \vpi_{i_0} )}] \boxtimes \BC [\BA^{(\lambda_{m-l})}]$ and 
$\gr_{l} W(\lambda_m)$ that changes the order of the highest weight vectors 
and the one of non-highest weight vectors (or mixes up $S_{1}$ and $S_{2}$). 
Therefore, $\gr_{l} W(\lambda_m)$ itself is torsion-free 
as a $\BC [\BA^{(l \vpi_{i_0} )}] \boxtimes \BC [\BA^{(\lambda_{m-l})}]$-module.

Because $M ( \lambda; n )$ is generated by 
$M ( \lambda; n )_{0}$ as an $A_{G}^{i_{0}}$-module, 
we have an injective map
\begin{equation*}
\eta : \mathrm{gr}_{l} \, W(\lambda_{m}) \hookrightarrow 
  \BC [\BA^{(l \vpi_{i_0} )}] \otimes \mathrm{gr}_{0} \, W(\lambda_{m-l})
\end{equation*}
of $\BC [\BA^{(l \vpi_{i_0} )}] \boxtimes \BC [\BA^{(\lambda_{m-l})}]$-modules, 
which is an isomorphism after a localization to some Zariski open subset of 
$\BA^{(l \vpi_{i_0} )} \times \BA^{(\lambda_{m-l})}$.

Since $\gr_{l} W(\lambda_m)$, $0 \le l \le m$, stratifies 
$W ( \lambda_m )$, we deduce that the $M( \lambda; n )$'s, 
with $\lambda$ varying, give a stratification of 
the $A_{G}$-module $R_{G}$. Hence, in order to prove that $R_G$ is 
free over $A_{G}^{i_0}$, it suffices to verify that each $M ( \lambda; n )$ is 
a free $A_{G}^{i_0}$-module. By construction, the image of $\eta$ contains 
$\BC \otimes \gr_{0} W(\lambda_{m-l})$, which gives a $\BC [\BA^{(l \vpi_{i_0})}]$-module 
generator of $\gr_{l} \, W(\lambda_{m})$. Therefore, the map $\eta$ must be 
an isomorphism. As a consequence, we conclude that
\begin{equation*}
\BC [\BA^{(l \vpi_{i_0})}]^{\ast} \otimes M ( \lambda; n )_0 \cong M ( \lambda; n )_l
\end{equation*}
through the multiplication map. In other words, 
$M ( \lambda; n )$ is a free $A_{G}^{i_0}$-module.

Because the above argument is consistent with the filtrations and 
their associated graded modules arising from a different choice of $i_0 \in I$, 
we can vary $i_0 \in I$ and construct the associated graded modules inductively 
on a fixed total order on $I$. This gives a stratification of $R_{G}$ that is 
free over $A_{G}$. Hence the ring $R_G$ itself is free over $A_G$. 
This completes the proof of the theorem.
\end{proof}
%
%
\begin{thm} \label{coh-e}
For each $\lambda \in P$, we have
\begin{equation*}
\gch H^{n} ( \DP, \CO_{\DP} ( \lambda ) ) = 
  \begin{cases} 
    \gch V^{-}_{e}( - w_{\circ} \lambda ) & 
      \text{\rm if $\lambda \in P^{+}$ and $n = 0$}, \\[1.5mm]
      0 & 
      \text{\rm otherwise}.
  \end{cases}
\end{equation*}
\end{thm}

\newcommand{\Og}{\Omega} 

\begin{proof} 
We know that $\DP$ is a closed subscheme of 
$\prod_{i \in I} \BP ( L ( \vpi_{i} )\bra{z} )$ by \eqref{Pemb}. 
Therefore, we have a countable set $\Og$ of $I$-tuples of 
$( H \times \Gm )$-eigenvectors of $\bigsqcup_{i \in I} L( \vpi_{i} )\bra{z}$, 
one for each $i \in I$, so that it induces an affine open cover 
$\CU := \{ \CU_{S} \}_{S \subset \Og}$ of $\DP$ 
(where $\CU_{S} := \bigl\{ f \neq 0 \mid f \in S \bigr\}$) 
that is closed under intersection. 
%

Now, the maps $L(\vpi_i)\bra{z} \setminus \{ 0 \} \rightarrow 
\BP(L(\vpi_i)\bra{z} )$, $i \in I$, induce a (right) $H$-fibration $\tDP$ 
that defines an open scheme of $\prod_{i \in I} L( \vpi_i )\bra{z}$, 
which corresponds to specifying a nonzero vector $u^{\vpi_{i}}$ 
instead of a one-dimensional $\BC$-vector subspace $L^{\vpi_{i}} \ni u^{\vpi_{i}}$ 
in the definition of DP data; 
its closure $\hDP$, which corresponds to allowing $u^{\vpi_{i}} = 0$ in a DP datum, 
is an affine subscheme of $\prod_{i \in I} L( \vpi_i )\bra{z}$ 
of infinite type. We set $Z := \hDP \setminus \tDP$, 
which is a closed subscheme of $\hDP$. 
Also, the pullback $\ti{\CU}_S$ of $\CU_S$ to $\tDP$ defines 
an affine open subset of $\hDP$. 
By the finiteness of the defining functions, 
$\ti{\CU}_S \hookrightarrow \hDP$ is 
quasi-compact by \cite[Proposition~1.1.10]{EGA1}.
For each finite subset $S \subset \Og$, 
we set $\ti{\CU}^S :=  \{ \ti{\CU}_{T} \}_{T \subset S}$, 
which is again a collection of affine subschemes 
that is closed under intersections, and 
$\ti{U}^S := \bigcup_{T \subset S}\ti{\CU}_T$. 
In addition, we set $\ti{Z}_S := \hDP \setminus \ti{U}^S$.

Let us denote the natural projection 
$\tDP \rightarrow \DP$ by $\pi$. 
Since $\DP$ is a (right) free quotient of $\tDP$ by $H$, 
we deduce that $\CO_{\DP} ( \lambda ) = ( \pi_{*} \CO_{\tDP} )^{(H, \lambda)}$, 
where $\bullet^{(H, \lambda)}$ denotes the $\lambda$-isotypical component
with respect to the right $H$-action. 
Because discarding open sets (in such a way that the remaining ones are closed 
under intersection) in the \u{C}ech  complex defines a projective system of 
complexes satisfying the Mittag-Leffler condition, 
\cite[Proposition~13.2.3]{EGA3-1} yields an isomorphism 
%
%
\begin{equation} \label{cech-ref}
H^{n} ( \tDP, \CO_{\tDP} ) \cong
 \varprojlim_{S} H^{n} ( \ti{U}^S, \CO_{\tDP} )
 \quad \text{for each $n \in \BZ_{\ge 0}$}.
\end{equation}
We have
\begin{equation} \label{des-H}
H^{n} ( \DP, \CO_{\DP} ( \lambda ) ) \cong 
H^{n} ( \tDP, \CO_{\tDP} )^{(H, \lambda)}
\quad \text{for $n \in \BZ_{\ge 0}$}, 
\end{equation}
since the right $H$-action on $\tDP$ is free, and 
it induces a semi-simple action on the level of \u{C}ech complex.

The long exact sequence of local cohomologies 
(see \cite[Expos\'e~I, Corollaire~2.9]{SGA2}) yields:
\begin{equation} \label{l-les}
\cdots \rightarrow H^{n}_{\ti{Z}_S} ( \hDP, \CO_{\hDP} ) 
  \rightarrow H^{n} ( \hDP, \CO_{\hDP} ) 
  \rightarrow H^{n} ( \ti{U}^S, \CO_{\hDP} ) \rightarrow 
  H^{n+1}_{\ti{Z}_S} ( \hDP, \CO_{\hDP} ) \rightarrow \cdots. 
\end{equation}
Since $\hDP$ is affine, this induces
\begin{equation} \label{isom-local}
H^{n} ( \ti{U}^S, \CO_{\hDP} ) 
 \stackrel{\cong}{\longrightarrow} 
H^{n+1}_{\ti{Z}_S} ( \hDP, \CO_{\hDP} ) 
 \quad \text{for $n \ge 1$}
\end{equation}
by \cite[Th\'eor\`eme~1.3.1]{EGA3-1}.
Here the quasi-compactness of 
the embedding $\ti{U}^S \hookrightarrow \hDP$, together with 
\cite[Expos\'e~I\!I, Proposition~5]{SGA2}, implies that
\begin{equation} \label{Koszul}
H^{n}_{\ti{Z}_S} ( \hDP, \CO_{\hDP} ) \cong 
H^{n} ( K_S ( \BC [ \hDP ] ) ),
\end{equation}
where $K_S ( \BC [ \hDP ] )$ denotes 
the (cohomological) $\BC [ \hDP ]$-Koszul complex 
defined through $S \subset \Og$ (see \cite[(1.1.2)]{EGA3-1}).

In view of \eqref{l-les}, the comparison of \eqref{cech-ref} and \eqref{Koszul} 
via \eqref{isom-local} yields an isomorphism
\begin{equation} \label{genuine}
H_{Z}^{n} ( \hDP, \CO_{\hDP} ) \cong 
\varprojlim_S H^{n} ( K_S ( \BC [ \hDP ] ) ) \cong 
\varprojlim_S H^{n}_{\ti{Z}_S} ( \hDP, \CO_{\hDP} ). 
\end{equation}
We know from \cite[Proposition~1.1.4]{EGA3-1} that
the Koszul complex $K_S ( \BC [ \hDP ] )$ has 
trivial cohomology at degree $< n$ 
if $S$ contains a regular sequence of length $n$. 
Here we see from Corollary~\ref{pcr} that $\BC[\hDP] = 
\bigoplus_{\lambda \in P^{+}} W(\lambda)^{\ast}$. 
Also, by Theorem~\ref{R-free}, we can rearrange $\Og$ 
if necessary in such a way that for each $i \in I$, 
the set of the $i$-th components of the elements in $\Og$ 
contain a regular sequence of arbitrary length. 
Then, we deduce from \eqref{genuine} that 
\begin{equation*}
H_{Z}^{n} ( \hDP, \CO_{\hDP} ) = \{ 0 \} \quad 
 \text{for all $n \in \BZ_{\ge 0}$}. 
\end{equation*}
Therefore, \eqref{l-les} and the affinity of $\hDP$ imply that
\begin{equation*}
H^{n} ( \tDP, \CO_{\hDP} ) = 
H^{n} ( \tDP, \CO_{\tDP} ) = \{ 0 \} \quad
 \text{for all $n > 0$}.
\end{equation*}
In view of \eqref{des-H} and Theorem~\ref{BFmain}, 
we conclude the assertion of the theorem. 
\end{proof}
%
%
\begin{cor} \label{coh}
For each $x \in W_{\af}^{\ge 0}$, 
the scheme $\DP(x)$ is normal. 
Moreover, for each $\lambda \in P$ and $x \in W_{\af}^{\ge 0}$, we have
\begin{equation}
\gch H^{n} ( \DP(x), \CO_{\DP(x)}(\lambda) ) = 
 \begin{cases} 
  \gch V^{-}_{x} ( - w_{\circ} \lambda ) & 
    \text{\rm if $\lambda \in P^{+}$ and $n = 0$}, \\[1.5mm]
  0 & \text{\rm otherwise}. 
 \end{cases}.
\end{equation}
In particular, we have
\begin{equation}
\gch H^{0}( \DP(x), \CO _{\DP(x)}(\lambda) ) \in 
 (\BZ\pra{q^{-1}})[P] \subsetneq (\BZ [P])\pra{q^{-1}}.
\end{equation}
\end{cor}

\begin{proof}
Once we know the normality of $\DP$ and 
the cohomology vanishing result in Theorem~\ref{coh-e}, 
the same argument as in \cite[Theorem 4.7]{Kat16} 
(see Theorem~\ref{Kmain}) yields all the assertions 
except for the last one.

The last assertion on the character estimate 
follows from a result about extremal weight modules 
(\cite[Corollary~5.2]{Kas02}) and 
the fact that $U_{\q}^{-}(\Fg)$ is concentrated 
on subspaces of $q$-degree $\le 0$.
\end{proof}
%
%
\begin{cor} \label{w-push}
For an arbitrary $x \in W_{\af}^{\ge 0}$, 
we take $\bi$ as in \eqref{eq:bi}. 
Then we have $\sfm_{\ast} \CO _{\DP ( \bi )} \cong \CO _{\DP ( x )}$.
\end{cor}

\begin{proof}
We adopt the notation of Section~\ref{subsec:BSDH}. 
The map $\sfm$ factors as the composite of 
the map $\sfm$ for $t_{2 m_0 \rho^{\vee}}$ and 
a successive composite of the $q_{i_{k}, x ( k )}$ for $1 \le k \le \ell'$. 
The case $x = t_{\xi}$ for $\xi \in Q^{\vee,+}$ is clear 
since $\DP ( t_{\xi} ) \cong \DP$ (through $\imath_{\xi}$). 
Therefore, it suffices to show that
$( q_{i, x} )_{\ast} \CO_{\Iw(i) \times^{\Iw} \DP ( x )} 
 \cong \CO_{\DP ( s_{i} x )}$ 
for each $x \in W_{\af}^{\ge 0}$ and $i \in I_{\af}$ 
such that $s_{i} x \prec x$. 
This assertion itself follows from Corollary~\ref{coh} and \cite[Theorem 4.7]{Kat16} 
(in view of Lemma~\ref{vample}). Hence we obtain the assertion of the corollary. 
\end{proof}
%
%
\section{$K$-theory of semi-infinite flag manifolds.}
\label{sec:K-SiFl}

We keep the notation and setting of Section~\ref{subsec:BSDH}.  
%
%
\begin{prop} \label{line-bundles}
Every $\ti{\Iw}$-equivariant locally free sheaf of rank one 
{\rm(}i.e., line bundle{\rm)} on $\DP ( x )$ is of the form 
$\chi \otimes _{\BC} \CO_{\DP ( x )} ( \lambda )$ 
for some $\lambda \in P$ and an $\ti{\Iw}$-character $\chi$.
\end{prop}

\begin{proof}
For $x = e$, the boundary of the open $G\bra{z}$-orbit $\bO$ in $\DP$ is 
of codimension at least two, and 
the open $G \bra{z}$-orbit $\bO$ has 
a structure of pro-affine bundle over $G / B$. 
In particular, an $\ti{\Iw}$-equivariant line bundle over $\bO$ is 
the pullback of a $B \times \Gm$-equivariant line bundle over $G / B$. 
Because every line bundle over $G/B$ carries a unique $G$-equivariant 
structure by \cite[Sect.~3.3]{KKV}, 
and $B$-equivariant structures of the trivial 
line bundle $\CO_{G/B}$ are in bijection with $P$ 
(since $H^{0}(G/B,\,\CO_{G/B})=\BC$), we deduce that 
every $B$-equivariant line bundle over $G / B$ is 
an $H$-character twist of a $G$-equivariant line bundle, 
which is obtained as the restriction of some $\CO_{\DP} ( \lambda )$. 
Consequently, the assertion follows for $x = e$.

Now, for $y_{1},\,y_{2} \in W_{\af}^{\ge 0}$ such that $y_{1} \preceq y_{2}$, 
the restriction map transfers an $\ti{\Iw}$-equivariant line bundle over 
$\DP(y_{1})$ to an $\tilde{\Iw}$-equivariant line bundle over $\DP(y_{2})$.
Also, for an arbitrary $x \in W_{\af}^{\ge 0}$, 
we can find $\xi \in Q^{\vee,+}$ such that
%
%
\begin{equation} \label{eq:DPs}
\DP ( xt_{\xi} ) \subset \DP ( t_{\xi} ) \subset \DP ( x ) \subset \DP(e)=\DP
\end{equation}
by (the proof of) Lemma~\ref{BSDH-cover}, 
since $x = x ( \bj ) \preceq t_{2m\rho^{\vee}}$ for $m \gg 0$. 
Because we have $\DP ( t_{\xi} ) \cong \DP$ 
as schemes with an $\ti{\Iw}$-action, 
we conclude that a nonisomorphic pair of 
($\ti{\Iw}$-equivariant) line bundles over $\DP$ 
restricts to a nonisomorphic pair of 
($\ti{\Iw}$-equivariant) line bundles over $\DP(t_{\xi})$. 
Since $\DP ( x t_{\xi} ) \cong \DP ( x )$, 
the same is true for line bundles over $\DP ( x )$. 
Therefore, by means of \eqref{eq:DPs}, we deduce the assertion of the proposition
for an arbitrary $x \in W_{\af}^{\ge 0}$ from the case $x = e$.
This proves the proposition. 
\end{proof}

The following is an immediate consequence of 
Corollary~\ref{coh} and Proposition~\ref{line-bundles}. 
%
%
\begin{cor} \label{cvs}
For each $\ti{\Iw}$-equivariant line bundle $\CL$ over $\DP(x)$, 
we have
\begin{equation*}
H ^{n} (\DP(x),\,\CL ) = \{ 0 \} \quad \text{\rm for all $n > 0$}. 
\end{equation*}
\end{cor}
%
%
\begin{lem} \label{faith}
For each $\ti{\Iw}$-equivariant 
quasi-coherent sheaf $\CE$ on $\DP$ such that
\begin{equation*}
\Gamma ( \DP, \, \CE ( \lambda ) ) = \{ 0 \} \quad 
\text{\rm for all $\lambda \in P$, where 
$\CE ( \lambda ) = \CE \otimes_{\CO_{\DP}} \CO_{\DP}(\lambda)$},
\end{equation*}
we have $\CE = \{ 0 \}$.
\end{lem}

\begin{proof}
By the quasi-coherence and $\ti{\Iw}$-equivariance, every nonzero section of $\CE$ has 
an $\ti{\Iw}$-stable support, which must be a union of $\Iw$-orbits. 
In addition, it defines a regular section on a complement of 
finitely many hyperplanes having poles of finite order around the boundary points. 
Therefore, Lemma~\ref{vample} implies the desired result. 
\end{proof}
%
%
\begin{thm} \label{cvb}
Let $\CE$ be an $\ti{\Iw}$-equivariant quasi-coherent $\CO_{\DP}$-module
satisfying the following conditions{\rm:}
\begin{equation} \label{cvb1}
\gdim \Gamma ( \DP,\, \CE ( \lambda ) )^{\ast} \in \BZ\bra{q} 
\quad \text{\rm for every $\lambda \in P$};
\end{equation}
\begin{equation} \label{cvb2}
\begin{array}{l}
\text{\rm there exists $\lambda_{0} \in P$ such that 
$\Gamma ( \DP,\, \CE ( \lambda ) ) = \{ 0 \}$} \\[2mm]
\text{\rm for all $\lambda \in P$ with 
$\pair{\lambda}{\alpha_{i}^{\vee}} < \pair{\lambda_{0}}{\alpha_{i}^{\vee}}$ 
for some $i \in I$}. 
\end{array}
\end{equation}
Then, we have a resolution
$\cdots \rightarrow \CP^{2}_{\CE} 
  \rightarrow \CP^{1}_{\CE} 
  \rightarrow \CP^{0}_{\CE}
  \rightarrow \CE
  \rightarrow 0$
of $\ti{\Iw}$-equivariant $\CO_{\DP}$-modules such that
\begin{enu}
\item $\gdim \Gamma ( \DP,\,\CP^{k}_{\CE}(\lambda) )^{\ast} \in \BZ\bra{q}$ 
      for every $k \ge 0$ and $\lambda \in P${\rm;}
\item for each $k \ge 0$, the $\ti{\Iw}$-equivariant $\CO_{\DP}$-module 
      $\CP^{k}_{\CE}$ is a direct sum of line bundles 
      {\rm(}if we forget the $\Iw$-module structure{\rm);}
\item for each $m \in \BZ$ and $\lambda \in P$, 
      the number of direct summands $\CP^{k}_{\CE}(\lambda)$ of 
      $\bigoplus_{k \ge 0} \CP^{k}_{\CE}(\lambda)$ 
      contributing to the homogeneous subspace of degree $m$ of 
      $\Gamma ( \DP,\,\bigoplus _{k \ge 0} \CP^{k}_{\CE}(\lambda) )$ 
      is finite.
\end{enu}
Moreover, we have $H^{n} ( \DP,\,\CE ) = \{ 0 \}$ for all $n > 0$. 
\end{thm}

\begin{proof}
Let 
\begin{equation*}
R_{G} := \bigoplus_{\lambda \in P^{+}} W ( \lambda )^{\ast} 
  = \bigoplus_{\lambda \in P^{+}} H^{0} ( \DP,\,\CO_{\DP} ( \lambda ) )
\end{equation*}
be the projective coordinate ring. Thanks to Lemma~\ref{faith}, 
the sheaf $\CE$ is determined by the $R_G$-module
\begin{equation*}
M ( \CE ) := \bigoplus_{\lambda \in P} H^{0} ( \DP, \, \CE ( \lambda ) ).
\end{equation*}
Because $M(\CE)$ is nonpositively graded and 
each homogeneous subspace with respect to 
the $(P \times \BZ)$-grading is finite-dimensional, 
we obtain a surjection $P^{0}_{\CE} \rightarrow M (\CE)$, 
where $P^{0}_{\CE}$ is a direct sum of $(P \times \BZ)$-graded projective 
$R_G$-modules tensored with $\ti{\Iw}$-modules; 
indeed, we can construct the desired maps inductively 
by starting with $\lambda=\lambda_{0} \in P$, and then by 
adding the $\vpi_{i}$, $i \in I$, repeatedly, 
by means of the projectivity of $R_{G}$. 
Since a (graded) projective $R_G$-module is obtained from $R_{G}$ 
by a grading shift and an $\ti{\Iw}$-module twist, we deduce that
$P^{0}_{\CE} \cong M ( \CP^{0}_{\CE} )$
as ($P \times \BZ$)-graded $R_G$-modules 
for a certain direct sum $\CP^{0}_{\CE}$ of $\ti{\Iw}$-equivariant line bundles 
(with some twist of the $\ti{\Iw}$-equivariant structure). 
Here the surjectivity of $P^{0}_{\CE} \rightarrow M ( \CE )$ of $R_{G}$-modules
implies that $\CP^{0}_{\CE} \rightarrow \CE$ is also surjective. 
Also, by our character estimate, we deduce that
$\gdim \Gamma ( \DP,\,\CP^{0}_{\CE} ( \lambda ) )^{\ast} \in \BZ\bra{q}$.
Now, let $\Xi ( \CE )$ be the set of those pairs 
$(\lambda,\,m) \in P \times \BZ$ for which 
\begin{equation*}
\bigoplus_{ \mu \in \lambda-P^{+} }
\left(\bigoplus _{n > m} \Gamma ( \DP, \, \CE ( \mu ) )_{n}\right) \oplus 
\bigoplus _{\mu \in \lambda-P^{+}, \, \mu \ne \lambda} 
\Gamma ( \DP, \, \CE ( \mu ) )_{m}= \{0\}.
\end{equation*}
Then, we can rearrange $\CP^{0}_{\CE}$, if necessary, 
to assume that
\begin{equation} \label{eq:sp1}
\gdim \ker \bigl( 
 \Gamma ( \DP, \, \CP^{0}_{\CE} ( \lambda ) ) \rightarrow 
 \Gamma ( \DP, \CE ( \lambda ) ) 
\bigr)^{\ast} \in \BZ\bra{q} \quad 
\text{for all $\lambda \in P$}; 
\end{equation}
\begin{equation} \label{eq:sp2}
\Gamma ( \DP, \, \CP^{0}_{\CE} ( \lambda ) )^{\ast} = \{ 0 \} 
  \quad \text{\rm for all $\lambda \in P$ with
  $\pair{\lambda}{\alpha_{i}^{\vee}} < \pair{\lambda_{0}}{\alpha_{i}^{\vee}}$ 
  for some $i \in I$};
\end{equation}
\begin{equation} \label{eq:sp3}
\ker \bigl( 
  \Gamma ( \DP, \, \CP^0_{\CE} ( \lambda ) ) \rightarrow 
  \Gamma ( \DP, \CE ( \lambda ) ) 
\bigr)_{m} = \{ 0 \} 
\quad \text{for every $(\lambda, m) \in \Xi ( \CE )$}. 
\end{equation}
Thanks to \eqref{eq:sp1} and \eqref{eq:sp2}, 
we can replace $\CE$ with 
$\ker \bigl( \CP^{k}_{\CE} \rightarrow \CP^{k-1} _{\CE} \bigr)$ repeatedly 
(with the convention $\CP^{-1}_{\CE} = \CE$) to 
apply the procedure above in order to obtain $\CP^{k+1} _{\CE}$ 
for each $k \ge 0$. This yields an $\ti{\Iw}$-equivariant resolution
%
%
\begin{equation} \label{eq:res}
\cdots 
  \rightarrow \CP^{2}_{\CE} 
  \rightarrow \CP^{1}_{\CE} 
  \rightarrow \CP^{0}_{\CE} 
  \rightarrow \CE 
  \rightarrow 0, 
\end{equation}
in which each $\CP^{k}_{\CE}$ is a direct sum of 
$\ti{\Iw}$-equivariant line bundles 
(with some twist by $\ti{\Iw}$-modules). 
By the construction, we have
\begin{equation*}
\emptyset = 
 \bigcap_{k \ge 0} \Xi ( \ker d_{k} ) \subset P \times \BZ,
\end{equation*}
and hence the resolution \eqref{eq:res} satisfies 
the first two of the requirements. 
Also, taking into account \eqref{eq:sp2} and \eqref{eq:sp3}, 
we see that the resolution \eqref{eq:res} satisfies the third one 
of the requirements.

Finally, by applying Corollary~\ref{cvs}, 
we conclude the desired cohomology vanishing.
This completes the proof of the theorem. 
\end{proof}
%
%
\begin{cor} \label{conv}
Keep the setting of Theorem~\ref{cvb}. 
We have
\begin{equation*}
\sum_{k \ge 0} 
    \gdim \Gamma ( \DP, \, \CP^{k}_{\CE} ( \lambda ) )^{\ast} 
    \in \BZ\bra{q} \quad 
\text{\rm for all $\lambda \in P$}, 
\end{equation*}
and an {\rm(}unambiguously defined{\rm)} equality
\begin{equation*}
\gdim \Gamma ( \DP,\,\CE ( \lambda ) )^{\ast} 
  = \sum_{k \ge 0} 
    (-1)^{k} \gdim \Gamma ( \DP, \, \CP^{k}_{\CE} ( \lambda ) )^{\ast} 
    \in \BZ\bra{q} \quad 
\text{\rm for all $\lambda \in P$}.
\end{equation*}
\end{cor}

\begin{proof}
By Theorem~\ref{cvb}\,(3), there are only finitely many terms $\CP^{k}_{\CE}(\lambda)$
contributing to each homogeneous subspace of a fixed $q$-degree of 
$\Gamma(\DP,\,\CE(\lambda))^{\ast}$. Therefore, 
the projective resolution afforded in Theorem~\ref{cvb} implies the desired result.
This proves the corollary. 
\end{proof}

We say that a condition depending on $\lambda \in P$ holds for $\lambda \gg 0$ 
if there exists $\gamma \in P$ and the condition holds for every $\lambda \in \gamma + P^+$.

For an element $f \in (\BZ [P])\bra{q^{-1}}$, 
we define $|f| \in (\BZ_{\ge 0}[P])\bra{q^{-1}}$ as follows:
\begin{equation*}
f=\sum_{n \in \BZ_{\le 0}}
\underbrace{\left(\sum_{\nu \in P} c_{\nu,n} e^{\nu} \right)}_{\in \BZ[P]} q^{n}
\quad \Rightarrow \quad
|f|:=\sum_{n \in \BZ_{\le 0}}
\underbrace{\left(\sum_{\nu \in P} |c_{\nu,n}| e^{\nu} \right)}_{\in \BZ_{\ge 0}[P]} q^{n}. 
\end{equation*}
We now define $K'_{\ti{\Iw}} ( \DP )$ to be 
the following set of formal infinite sums, 
modulo equivalence relation $\sim$:
\begin{equation*}
\left\{ f=\sum_{ \lambda \in P } 
  f_{\lambda} \cdot [ \CO_{\DP} ( \lambda )] 
\ \Biggm| \  
\text{$f_{\lambda} \in (\BZ [P])\bra{q^{-1}}$, $\lambda \in P$, 
satisfy condition \eqref{eq:K}} \right\}\Biggm/\sim
\end{equation*}
where condition \eqref{eq:K} is given by:
%
%
\begin{equation} \label{eq:K}
\sum_{ \lambda \in P } 
| f_{\lambda} | \cdot 
\gch H^{0} ( \DP, \, \CO_{\DP} ( \lambda + \mu ) ) \in (\BZ_{\ge 0}[P])\bra{q^{-1}}
\quad \text{for every $\mu \in P$}, \tag{\#}
\end{equation}
and the equivalence relation $\sim$ is:
\begin{equation} \label{eq:finite-support}
f \sim 0 \iff \sum_{ \lambda \in P } f_{\lambda} \cdot 
\gch H^{0} ( \DP, \, \CO_{\DP} ( \lambda + \mu ) ) = 0
\quad \text{if $\mu \gg 0$}. \tag{$\sim$}
\end{equation}
By construction, 
$K'_{\ti{\Iw}} ( \DP )$ is topologically spanned by classes of 
$\Iw$-equivalent line bundles. 
Hence we deduce that the following map is well-defined: 
\begin{equation*}
\begin{split}
& \Pic^{\ti{\Iw}} \DP \times 
  K'_{\ti{\Iw}} ( \DP ) \rightarrow K'_{\ti{\Iw}} ( \DP ), \\ 
& \left( \CL, f = \sum f_{\lambda} [\CO ( \lambda)] \right) \mapsto 
  [\CL] \cdot f =  \sum f_{\lambda} [\CL \otimes \CO ( \lambda)].
\end{split}
\end{equation*}

By Corollary~\ref{conv}, 
each $\CE$ from Theorem~\ref{cvb} satisfies
\begin{equation*}
[\CE] := \sum_{k \ge 0} (-1)^{k} [ \CP^{k}_{\CE}] \in K' _{\ti{\Iw}} ( \DP ).
\end{equation*}
In particular, thanks to Corollary~\ref{coh}, 
we have $[\CO_{\DP ( x )} ( \lambda )] \in K' _{\ti{\Iw}} ( \DP )$ 
for every $x \in W_{\af}^{\ge 0}$ and $\lambda \in P$. 

Let $\Fun_{ (\BC[P])\bra{q^{-1}} } P$ denote the space of 
$(\BC [P])\bra{q^{-1}}$-valued functions on $P$, and let
\begin{equation*}
\Fun^f_{ (\BC[P])\bra{q^{-1}} } P \subset \Fun_{ (\BC[P])\bra{q^{-1}} } P
\end{equation*}
be the subset consisting of those functions that are
zero on $\gamma + P^{+}$ for some $\gamma \in P$. 
Then we form a $(\BC [P])\bra{q^{-1}}$-module quotient
\begin{equation*}
\Fun^{\mathrm{ess}}_{ (\BC[P])\bra{q^{-1}} } P 
:= \Fun_{ (\BC[P])\bra{q^{-1}} } P / \Fun^f_{ (\BC[P])\bra{q^{-1}} } P.
\end{equation*}
For each $\mu \in P$, we regard the assignment
\begin{equation*}
P \ni \lambda \mapsto 
\begin{cases}
\gch V^{-}_{e} ( - w_{\circ} ( \lambda + \mu ) ) 
  & \text{if $\lambda+\mu \in P^{+}$}, \\[1.5mm]
0 & \text{otherwise}, 
\end{cases}
\end{equation*}
as an element of $\Fun_{ (\BC[P])\sbra{q^{-1}} } P$, 
which we denote by $\Psi ( [\CO _{\DP} ( \mu )] )$. 
Passing to the quotient, 
we obtain a map $\Psi : \{ [\CO _{\DP} ( \mu )] \}_{\mu \in P} \ni [\CO _{\DP} ( \mu )] 
\mapsto \Psi ( [\CO _{\DP} ( \mu )] ) \in \Fun^{\mathrm{ess}}_{ (\BC [P])\sbra{q^{-1}} } P$.
%
%
\begin{thm} \label{Windep}
The map $\Psi$ extends to an injective $(\BZ [P])\bra{q^{-1}}$-linear map
$\Psi : K' _{\ti{\Iw}} ( \DP ) \rightarrow 
 \Fun^{\mathrm{ess}}_{ (\BC [P])\sbra{q^{-1}} } P$.
\end{thm}

\begin{proof}
We assume the contrary to deduce a contradiction. 
Let $C \in K'_{\ti{\Iw}} ( \DP )$, and expand the $C$ as: 
\begin{equation*}
C = \sum_{
  n \in \BZ_{\le 0},\,\nu,\,\mu \in P}
  c _{n,\nu,\mu} \, q^{n} e^{\nu} \cdot [\CO_{\DP}(\mu)], 
  \quad \text{with $c_{n,\nu,\mu} \in \BZ$}, 
\end{equation*}
inside $K'_{\ti{\Iw}} ( \DP )$. 
We have $\Psi ( C ) = 0$ if and only if 
there exists $\gamma \in P$ such that
\begin{equation*}
\sum_{
  n \in \BZ_{\le 0},\,\nu,\,\mu \in P}
  c _{n,\nu,\mu} \, q^{n} e^{\nu} \cdot 
  \gch V_e^- ( - w_{\circ} ( \lambda + \mu ) ) = 0 \qquad \text{for $\mu \in \gamma + P ^+$}. 
\end{equation*}
This is exactly the condition $C \sim 0$. 
Hence the map $\Psi$ defines an injective map. 
It is $(\BZ [P])\bra{q^{-1}}$-linear by construction.
\end{proof}

\newcommand{\Xp}{C_{p}} 

For countably many elements $\Xp$, $p \ge 0$, in $K_{\ti{\Iw}}' ( \DP )$ 
that represent the classes of 
$\ti{\Iw}$-equivariant quasi-coherent sheaves, we expand them as:
\begin{equation*}
\Xp = \sum_{ \lambda \in P } a_{\lambda} ( \Xp ) [\CO_{\DP} ( \lambda )], 
 \quad \text{with $a_{\lambda} ( \Xp ) \in (\BZ [P])\bra{q^{-1}}$} 
\end{equation*}
by using the procedure of Theorem~\ref{cvb};
we say that the sum $\sum _{p \ge 0} \Xp$ converges absolutely
to an element of $K'_{\Iw} ( \DP )$ if there exists some $\lambda_0 \in P$
(uniformly for all $p \ge 0$) such that $a_{\lambda} ( \Xp ) = 0$ 
for all $\lambda \in P$ with $\pair{\lambda}{\alpha_{i}^{\vee}} < 
\pair{\lambda_{0}}{\alpha_{i}^{\vee}}$ for some $i \in I$, and 
if the number of those $( \lambda, p ) \in 
P \times \BZ_{\ge 0}$ for which $a_{\lambda} ( \Xp )$ 
has a nonzero term of $q$-degree $m$ is finite for each $m \in \BZ$. 
It is straightforward to see that $\sum_{p \ge 0} \Xp$ defines 
an element of $K'_{\ti{\Iw}} ( \DP )$, which does not depend on 
the order of the $\Xp$'s. 

\begin{rem} \label{rem:conv}
Since the coefficients for $K'_{\ti{\Iw}} ( \DP )$ are in $\BZ$, 
the sum $\sum_{p \ge 0} \Xp$ must ``diverge'' or ``oscillate''
when it does not converge absolutely.
\end{rem}

\begin{prop}\label{indep-of-QGy}
Let $f_{y}\in (\BZ [P])\bra{q^{-1}}$, $y \in W_{\af} ^{\ge 0}$. 
Then the formal sum
\begin{equation} \label{fsum}
\sum _{y \in W_{\af}^{\ge 0}} f_{y} \cdot [\CO_{\DP ( y )}]
\end{equation}
converges absolutely to an element of $K'_{\Iw} ( \DP )$ 
if and only if $\sum_{y \in W_{\af} ^{\ge 0}} | f_{y} | 
\in (\BZ_{\ge 0}[P])\bra{q^{-1}}$. Moreover, in this case, 
the equation
\begin{equation*}
\sum _{y \in W_{\af}^{\ge 0}} f_{y} \cdot [\CO_{\DP ( y )}] = 0
\end{equation*}
implies $f_{y} = 0$ for all $y \in W_{\af}^{\ge 0}$.
\end{prop}

\begin{proof}
First, we remark that 
$[\CO_{\DP ( y )}] \in K'_{\ti{\Iw}} ( \DP )$ 
for each $y \in W_{\af}^{\ge 0}$ by Corollary~\ref{coh} and Theorem~\ref{cvb}. 
More precisely, by means of the cohomology vanishing: 
\begin{equation*}
H^{\ast}(\DP,\,\CO_{\DP(y)}(\mu))=\{0\} \quad \text{if $\mu \notin P^{+}$},
\end{equation*}
we can take $\lambda_{0}=0$ in Theorem~\ref{cvb} by setting $\CE=\CO_{\DP(y)}$. 
In addition, we have 
$H^{0} ( \DP,\,\CO_{\DP ( y )}) = \BC$. Hence the construction in 
Theorem~\ref{cvb} implies that 
%
%
\begin{equation} \label{v-exp}
[\CO_{\DP ( y )}] = 
 [\CO_{\DP}] + \sum_{ \lambda \in - P^{+} } 
 a_{y} ( \lambda ) [ \CO_{\DP} ( \lambda ) ] \in K'_{\ti{\Iw}} ( \DP )
\end{equation}
for some $a_{y}(\lambda) \in (\ZP)\bra{q^{-1}}$. 
Therefore, the coefficient of $[\CO_{\DP}]$ in \eqref{fsum} 
must be the sum $\sum_{y \in W_{\af} ^{\ge 0}} f_{y}$, 
which unambiguously defines an element of $(\ZP)\bra{q^{-1}}$ if and only if 
the coefficient ($\in \BZ$) of each $q^{n}$, $n \in \BZ_{\le 0}$, 
in the sum $\sum_{y \in W_{\af}^{\ge 0}} f_{y}$ converges absolutely 
(see Remark~\ref{rem:conv}). This proves the first assertion.

We prove the second assertion. 
Let us assume the contrary to deduce a contradiction. 
Let $S$ be the set of those $y \in W_{\af}^{\ge 0}$ 
for which $f_{y} \neq 0$; 
denote by $n_{0}$ the maximal $q$-degree of 
all $f_{y}$, $y \in S$.  
Also, let $S_{1}$ be the set of those $y \in S$ for which
$\trs ( y ) \not> \trs ( y')$ for any $y' \in S$, 
where for $y \in W_{\af}^{\ge 0}$ 
of the form $y = w t_{\xi}$ 
with $w \in W$ and $\xi \in Q^{\vee,+}$, 
we set $\trs ( y ):=\xi$; 
since a polynomial ring (of finite variables) is Noetherian, 
we deduce that $|S_{1}| < \infty$. 
We choose and fix $y_{0} \in S_{1}$ such that
\begin{equation*}
P^{\#} := \bigl\{ 
 \lambda \in P^{+} \mid 
 \text{$\pair{\lambda}{\trs(y_{0})} < \pair{\lambda}{\trs(y)}$ 
 for all $y \in S$ with $y \ne y_{0}$} \bigr\}
\end{equation*}
is Zariski dense in $\Fh^{\ast}$. 
Let $n_{1}$ denote the maximal $q$-degree of those 
$f_{y}$, $y \in S_{1}$, for which 
$\trs ( y ) = \trs ( y_{0} )$. Then, the subset
\begin{equation*}
P^{\#\#} := \bigl\{ 
 \lambda \in P^{+} \mid 
 \text{$\pair{\lambda}{\trs(y_{0})} < \pair{\lambda}{\trs(y)} + n_{0} - n_{1}$
 for all $y \in S$ with $y \neq y_{0}$}\bigr\} 
\end{equation*}
of $P^{\#}$ is still Zariski dense in $\Fh^{\ast}$.

For each $\lambda \in P^{\#\#}$, the coefficient of the part of 
degree $\bigl( n_{1} - \pair{\lambda}{\trs ( y_{0} )} \bigr)$ of
\begin{equation*}
\Psi \left( 
 \sum _{y \in W_{\af}^{\ge 0}} f_{y} \cdot [\CO_{\DP ( y )}] \right) 
( \lambda )
\end{equation*}
is equal to 
\begin{equation*}
\sum _{w \in W} f_{ w t_{\trs(y_0)} }^{(n_{1})} \cdot 
\ch L^{-}_{w} ( - w_{\circ} \lambda ),
\end{equation*}
where $f_{y}^{(n_{1})} \in \BZ [P]$ is the part of degree $n_{1}$ of $f_{y}$; 
here, for $w \in W$ and $\mu \in P^{+}$, 
$L^{-}_{w} ( \mu ):=U(\Fb^{-})L(\mu)_{w\mu}$ denotes 
the (opposite) Demazure submodule of $L(\mu)$. 
This defines a $\ZP$-valued function of $\lambda \in P^{\#\#}$; 
note that the above is a finite sum. Here we have the equality
$\ch L^{-}_{w} ( - w_{\circ} \lambda )^{\ast} = 
D_{w w_{\circ}} ( e ^{\lambda} )$ in terms of 
the Demazure operator $D_{ww_{\circ}}$ for each $w \in W$ 
(see \cite[Theorem 8.2.9]{Kum02}); 
recall that the Demazure operator $D_{i}=D_{s_{i}}$, $i \in I$, is defined by 
$D_{i}(e^{\mu}):=(e^{\mu}-e^{s_{i}\mu-\alpha_{i}})/(1-e^{-\alpha_{i}})$ for $\mu \in P$.
Also, we know by \cite[pp.\,28--29]{Mac91} that 
the operators $D_{w}$, $w \in W$, form a set of 
$\ZP$-linearly independent $\BZ$-linear operators acting on $\ZP$. 
Therefore, we obtain $f_{w t_{\trs ( y_{0} )}}^{(n_{1})} = 0$
for all $w \in W$. This is a contradiction, and 
hence we cannot take the $S_{1}$ above from the beginning. 
Thus, we conclude the desired result. 
This completes the proof of the proposition. 
\end{proof}
%
%
\begin{cor}[of Theorem~\ref{Windep}] \label{critK}
Let $x \in W_{\af}^{\ge 0}$ and $\lambda \in P^{+}$. 
Consider a collection 
$f_{y} ( \lambda ) \in (\ZP)\bra{q^{-1}}$, 
$y \in W_{\af} ^{\ge 0}$, such that 
$\sum_{y \in W_{\af}^{\ge 0}} f_{y} ( \lambda ) \cdot [\CO_{\DP ( y )}]$ 
converges absolutely in $K'_{\ti{\Iw}}(\DP)$. Then, 
%
%
\begin{equation} \label{class-eq}
[\CO_{\DP ( x )} ( \lambda )] = 
\sum _{y \in W_{\af}^{\ge 0}} f_{y} ( \lambda ) \cdot [\CO_{\DP ( y )}]
\end{equation}
if and only if
%
%
\begin{equation} \label{char-eq}
\gch V^{-}_{x} ( - w_{\circ} ( \lambda + \mu ) ) = 
  \sum _{y \in W_{\af}^{\ge 0}} f_{v} ( \lambda ) \cdot 
  \gch V^{-}_{y} ( - w_{\circ} \mu ) \quad 
  \text{\rm for $\mu \gg 0$}.
\end{equation}
\end{cor}

\begin{proof}
We have an expansion
\begin{equation*}
[\CO_{\DP ( x )} ( \lambda )] = 
\sum _{y \in W_{\af}^{\ge 0}} f_{y} ( \lambda ) \cdot [\CO_{\DP ( y )}]
\end{equation*}
inside $K'_{\ti{\Iw}} ( \DP )$. From this, by twisting 
by the line bundle $\CO ( \mu )$ for $\mu \in P^{+}$, 
we obtain
\begin{equation*}
[\CO_{\DP ( x )} ( \lambda + \mu )] = 
\sum_{y \in W_{\af}^{\ge 0}} 
 f_{y} ( \lambda ) \cdot [\CO_{\DP ( y )} ( \mu )].
\end{equation*}
By Corollary~\ref{coh}, 
this equation in turn implies \eqref{char-eq}, 
which proves the ``only if'' part of the assertion.

We now assume \eqref{char-eq}. Then we have
\begin{equation*}
\Psi ( [\CO_{\DP ( x )} ( \lambda )] )  = 
 \sum _{y \in W_{\af}^{\ge 0}} f_{y} ( \lambda ) \cdot \Psi ( [\CO_{\DP ( y )}] )
\end{equation*}
by Corollary~\ref{coh}. Therefore, by Theorem~\ref{Windep}, 
we deduce that both sides of \eqref{class-eq} represent
the same class in $K'_{ \ti{\Iw} } ( \DP )$. 
Thus, we have proved the ``if'' part of the assertion.
This proves the corollary. 
\end{proof}
%
%
\begin{thm}[Pieri-Chevalley formula for semi-infinite flag manifolds] \label{thm:PC}
For each $\lambda \in P^{+}$ and $x \in W_{\af}^{\ge 0}$, 
there holds the equality
\begin{equation*}
[\CO_{\DP} ( \lambda )] \cdot [\CO_{\DP ( x )}] = 
 \sum_{\eta \in \SLS_{\sige x} (- w_{\circ} \lambda )} 
 e^{\fwt(\wt(\eta))}q^{\qwt(\wt(\eta))} \cdot [ \CO_{\DP ( \Deo{\eta}{x} )}]
\end{equation*}
in $K'_{\ti{\Iw}} ( \DP )$.
\end{thm}

\begin{proof}
By Theorem~\ref{thm:Dem}, we have
\begin{equation*}
\gch V^{-}_{x}( - w_{\circ}(\lambda + \mu) ) = 
\sum_{\eta \in \SLS_{\sige x} ( - w_{\circ}\lambda )} 
e^{\fwt(\wt(\eta))}q^{\qwt(\wt(\eta))} \cdot 
 \gch V^{-}_{\Deo{\eta}{x}}( -w_{\circ} \mu )
\end{equation*}
for each $\mu \in P^{+}$. 
Taking into account the fact that the LHS is zero 
if $\lambda + \mu \not\in P^{+}$, and the RHS is zero if $\mu \not \in P^{+}$, 
we conclude the above equation for $\mu \gg 0$. 
Here we see from Section~\ref{subsec:extremal} that 
$\qwt(\wt(\eta)) \in \BZ_{\le 0}$ 
for each $\eta \in \SLS_{\sige x}(-w_{\circ}\lambda)$. 
Also, we deduce from \eqref{eq:gch1} that 
for each $m \in \BZ_{\le 0}$, 
there exist only finitely many 
$\eta \in \SLS_{\sige x} ( - w_{\circ}\lambda )$ 
such that $\qwt(\wt(\eta)) \ge m$. 
Because $\gdim H^{0} ( \DP, \, \CO_{\DP ( \Deo{\eta}{x} )} ( \mu ) ) 
\in \BZ\bra{q^{-1}}$ by Corollary~\ref{coh}, we deduce that
\begin{equation*}
\sum_{\eta \in \SLS_{\sige x} ( - w_{\circ} \lambda )} 
e ^{\fwt(\wt(\eta))}q^{\qwt(\wt(\eta))} \cdot 
[ \CO_{\DP ( \Deo{\eta}{x} )}] \in K_{\ti{\Iw}} ( \DP ).
\end{equation*}
From this, by applying Corollary~\ref{critK}, 
we conclude the desired result.
This proves the theorem. 
\end{proof}
%
%
\section{nil-DAHA action on $K_{\ti{\Iw}} ( \RDP )$.}
\label{sec:nDAHA}

\begin{dfn}[{cf. \cite[Sect.~1.2]{CF13}}]
The nil-DAHA $\dH$ (of adjoint type) is 
the unital $\BZ [\bq^{\pm 1}]$-algebra generated by 
$T_i$, $i \in I_{\af}$, and $\be(\nu)$, $\nu \in P$, 
subject to the following relations:
\begin{equation} \label{eq:nDAHA}
\begin{cases}
T_i ( T_i + 1 ) = 0 & \text{for each $i \in I_{\af}$}; \\[3mm]
\text{if 
$\overbrace{s_i s_j \cdots}^{\text{$m_{ij}$ times}} = 
 \overbrace{s_j s_i \cdots}^{\text{$m_{ij}$ times}}$, then
$\overbrace{T_i T_j \cdots}^{\text{$m_{ij}$ times}} = 
 \overbrace{T_j T_i \cdots}^{\text{$m_{ij}$ times}}$} & 
\text{for each $i,\,j \in I_{\af}$}; \\[3mm]
\text{$\be ( \nu_{1} ) \be ( \nu_{2} ) = 
\be ( \nu_{1} + \nu_{2} )$ and $\be( 0 ) = 1$} & 
\text{for each $\nu_{1},\,\nu_{2} \in P$}; \\[3mm]
T_i \be( \nu ) - \be( s_i \nu ) T_i = 
\dfrac{\be( s_i \nu ) - \be( \nu )}{1 - \be( \alpha_i )}
& \text{for each $\nu \in P$ and $i \in I_{\af}$}. 
\end{cases}
\end{equation}
We define $\CH$ to be the $\BZ[\bq^{-1}]$-subalgebra of 
$\dH$ generated by $T_i$, $i \in I$, and $\be ( \nu )$, $\nu \in P$.
\end{dfn}
%
%
\begin{prop}\label{KK-reint}
The assignment
\begin{align*}
\te ( \vpi_{i} ) & : 
[\BC _{\lambda} \otimes_{\BC} \CO_{\DP} ( \mu )] \mapsto 
[\BC _{- \vpi_{i} + \lambda} \otimes_{\BC} \CO_{\DP} ( \mu )], \\
\tT_{i} & : [\BC_{\lambda} \otimes _{\BC} \CO_{\DP} ( \mu )] \mapsto
\frac{\te(-\lambda) - \te(-s_{i}\lambda+\alpha_{i})}
     {1-\te (\alpha_{i})} 
     [\CO_{\DP} ( \mu )],
\end{align*}
for each $i \in I$ and $\lambda,\,\mu \in P$, equips 
$K'_{\ti{\Iw}} ( \DP )$ with an action of the subalgebra $\CH$ of $\dH$ 
through the identifications{\rm:}
\begin{equation*}
\bq \mapsto q^{-1}, \qquad 
T_{i} \mapsto \tT_{i} - 1 \quad \text{\rm for $i \in I$}, \qquad 
\be ( \nu ) \mapsto \te ( \nu ) \quad \text{\rm for $\nu \in P$}.
\end{equation*}
\end{prop}

\begin{proof}
By the construction, $K'_{\ti{\Iw}} ( \DP )$ contains 
a dense subset isomorphic to $(\BZ [P])\bra{q^{-1}} 
\otimes_{\BZ} K _{G} ( G / B )$ (see Proposition~\ref{line-bundles}). 
Also, we have a surjection
$(\BZ [P])[q^{-1}] \otimes_{\BZ} K _{G} ( G / B ) 
\twoheadrightarrow 
\BZ[q^{-1}] \otimes_{\BZ} K_{B}(G/B)$; 
see, e.g., \cite[(3.17)]{KK90}. 
Here, for each $i \in I$, the action of $\tT_i$ is 
identical to the action of the Demazure operator $D_{i}=D_{s_{i}}$, and 
the action of $\te(\vpi_{i})$ corresponds to 
the twist by the $B$-character $-\vpi_{i}$; 
these define an $\CH$-action on 
$K _{B} ( G / B )$ by \cite[Sect.~3]{KK90}. 
Notice that both of the actions of $\te(\,\cdot\,)$ and $T_i$, $i \in I$, 
are neutral with respect to tensoring with $\CO_{G/B} ( \lambda )$ 
for each $\lambda \in P$, and 
that they also commute with the $\Gm$-twist corresponding to $q^{-1}$. 
Therefore, the $\CH$-action on $K_B ( G / B ) \cong \BZ [P] [\CO_{G/B}]$ induces 
an $\CH$-action on $(\ZP)[q^{-1}] \otimes_{\BZ} K _{G} ( G / B )$ through
\begin{equation*}
(\ZP)[q^{-1}] \otimes_{\BZ} K _{G} ( G / B ) \cong 
(\ZP)[q^{-1}] \otimes_{\BZ} \BZ [P] [\CO_{G/B}] = 
\bigoplus _{\lambda \in P} (\ZP)[ q^{-1} ] [\CO_{G/B} ( \lambda )], 
\end{equation*}
where the second factor of the leftmost one is responsible for the factors 
$\{ [\CO_{G/B} ( \lambda )] \} _{\lambda \in P}$. 
Finally, we complete $(\BZ[P])[ q^{-1} ] \otimes_{\BZ} K_{B}(G/B)$ 
to obtain the desired assertion. This proves the proposition. 
\end{proof}

\begin{cor}
The $\CH$-action in Proposition~{\rm\ref{KK-reint}} is 
induced by the $\ti{\Iw}$-character twists 
and the convolution action of the structure sheaves 
through $q_{i,e}$ for $i \in I$ {\rm(}see \eqref{eq:qiy}{\rm)}. 
\end{cor}

\begin{proof}
The assertion holds for the actions of $\CH$ on $K_B ( G / B )$ and 
$\BZ [q^{-1}] \otimes_{\BZ} K _{B} ( G / B )$ by \cite[Sect.~3]{KK90}. 
Also, by Lemma~\ref{one-step}, for each $i \in I$, 
the map $q_{i,e}$ is a $\BP^{1}$-fibration, and hence
\begin{equation*}
\BR^{k} (q_{i,e})_{\ast} \CO_{\Iw ( i ) \times ^{\Iw} \DP} \cong 
\begin{cases} 
 \CO_{\DP} & \text{if $k=0$}, \\[1mm]
 \{ 0 \} & \text{if $k \neq 0$}. 
\end{cases} 
\end{equation*}
This implies that the convolution action of 
$\Iw ( i ) / \Iw$, $i \in I$, on $\DP$ fixes the classes of 
$[\CO_{\DP} ( \lambda )]$ for each $\lambda \in P$ by the projection formula. 
Taking into account the fact 
that the twist of $R ( \ti{\Iw} ) \cong R(B \times \Gm)$ has 
an effect through the fiber of $q_{i,e}$, we conclude that 
$T_{i}$, $i \in I$, is identical to the convolution action induced by 
$q_{i,e}$ through the inclusion $\BZ [q^{-1}] \otimes_{\BZ} K _{B} ( G / B ) 
\subset K'_{\ti{\Iw}} ( \DP )$. This proves the corollary. 
\end{proof}

For each $\xi \in Q^{\vee,+}$, the natural inclusion map 
$\imath_{\xi} : \DP \hookrightarrow \DP$ induces 
an inclusion 
$(\imath_{\xi})_{\ast} : 
K'_{\ti{\Iw}} ( \DP ) \hookrightarrow K'_{\ti{\Iw}} ( \DP )$ of 
$(\ZP)\bra{q^{-1}}$-modules such that 
$( \imath_{\xi} )_{\ast} [\CO_{\DP(x)} ( \lambda )] = 
 [\CO_{\DP(xt_{\xi})} ( \lambda )]$ for each $x \in W_{\af}$. 
We define
\begin{equation*}
K_{\ti{\Iw}} ( \RDP ) := 
 \BZ\pra{q^{-1}} \otimes_{\BZ\sbra{q^{-1}}} 
 \varinjlim K'_{\ti{\Iw}}(\DP).
\end{equation*}
%
%
\begin{thm}[\cite{BFU}] \label{thm:BFU}
The assignment
\begin{align*}
\te(\vpi_{i}) & : 
 [\BC _{\lambda} \otimes_{\BC} \CO_{\DP ( t_{\xi} )} ( \mu )] \mapsto 
 [\BC _{-\vpi_{i} + \lambda} \otimes_{\BC} \CO_{\DP ( t_{\xi} )} ( \mu )]
 \quad \text{\rm for $i \in I$}, \\[1mm]
\tT_i & : [\BC_{\lambda} \otimes _{\BC} \CO_{\DP ( t_{\xi} )} ( \mu ) ] \\[3mm]
& \hspace*{5mm} \mapsto 
 \begin{cases} 
   \dfrac{\te ( - \lambda ) - \te ( - s_0 \lambda )}
     {1 - \te ( \alpha_0 )} [\CO_{\DP ( t_{\xi} )} ( \mu )] + 
     \te ( - s_0 \lambda ) [ \CO_{\DP ( s_0 t_{\xi} )} ( \mu )] 
     & \text{\rm for $i = 0$}, \\[7mm]
   \dfrac{\te ( - \lambda ) - \te ( - s_i \lambda + \alpha_i )}
     {1 - \te ( \alpha_i )} [\CO_{\DP ( t_{\xi} )} ( \mu )]
     & \text{\rm for $i \neq 0$}, 
  \end{cases}
\end{align*}
where $\xi \in Q^{\vee,+}$, 
and $\lambda, \mu \in P$, equips $K_{\ti{\Iw}} (\RDP)$ 
with an action of $\dH$ through the identifications{\rm:}
\begin{equation*}
\bq \mapsto q^{-1}, \qquad
T_{i} \mapsto \tT_{i} - 1 \quad \text{\rm for $i \in I_{\af}$}, \qquad
\be ( \nu ) \mapsto \te ( \nu ) \quad \text{\rm for $\nu \in P$}.
\end{equation*}
\end{thm}

\begin{proof}
Thanks to \cite[Sect.~3]{KK90} (and Lemma~\ref{one-step}), for each $i \in I$, 
the action of $\tT_i$ is induced by the pushforward of 
an $\ti{\Iw}$-equivariant inflated sheaf through $q_{i,e}$ 
(see Section~\ref{subsec:BSDH}), and 
the action of $\te(\vpi_{i})$ is induced by an $\ti{\Iw}$-character twist. 
Because these geometric counterparts commute with the pullback 
through $\imath_{\xi}$ for each $\xi \in Q^{\vee, +}$, 
our formulas define an action of $\CH$ on $K_{\ti{\Iw}} ( \RDP )$ 
induced by Proposition~\ref{KK-reint}.

Now, we have
%
%
\begin{equation} \label{D-factor}
\frac{f g - e^{- \alpha_0} s_0 (f g )}{1 - e^{- \alpha_{0}}} = 
\frac{f - s_{0}(f)}{1 - e^{- \alpha_{0}}} g + s_{0} ( f ) 
\frac{g - e^{- \alpha_{0}} s_{0} ( g )}{1 - e^{- \alpha_{0}}} 
\quad \text{for $f,\,g \in (\BC [P])\pra{q^{-1}}$}.
\end{equation}
Let $p_{0,t_{\xi}}:\Iw(0) \times^{\Iw} \DP (t_{\xi}) \rightarrow \BP^{1}$
be the inflation of the structure map of $\DP(t_{\xi})$, and 
let $\CE(W)$ denote the vector bundle over 
$\BP^{1} \cong \tilde{\Iw}(0)/\tilde{\Iw}$ 
associated to an $\tilde{\Iw}$-module $W$.
Then, by taking into account equation~\eqref{D-factor}, 
Corollary~\ref{coh} and Theorem~\ref{cvb}, and 
\cite[Corollary~4.8]{Kat16}, we deduce that 
for each $\lambda,\,\mu \in P$, $\nu \in P$, 
and $\xi \in Q^{\vee,+}$ such that $s_{0} t_{\xi} \in W_{\af}^{\ge 0}$, 
\begin{align}
\sum_{m,\,n \ge 0} (-1)^{m+n} & 
  \gch H^{m} \Bigl( \DP, \, \BR^{n} ( q_{0,t_{\xi}} ) _{\ast} 
  \bigl( \BC_{- \nu} \otimes_{\BC} \CO_{\DP ( t_{\xi})} ( \lambda ) \bigr) 
    \otimes_{\CO_{\DP}} \CO_{\DP} ( \mu ) \Bigr) \nonumber \\[2mm]
= & \sum_{m \ge 0} (-1)^{m}
  \gch H^{m} \bigl( \Iw ( 0 ) \times^{\Iw} \DP, \, 
      \BC_{- \nu} \otimes_{\BC} \CO_{\DP ( t_{\xi})} ( \lambda + \mu ) \bigr) \nonumber \\[2mm]
= & \sum_{m \ge 0} (-1)^{m} 
  \gch H^{0} \Bigl( \BP^{1}, \, \BR^{m} ( p_{0, t_{\xi}} )_{\ast} 
      \bigl( \BC_{- \nu} \otimes_{\BC} 
             \CO_{\DP ( t_{\xi})} ( \lambda + \mu ) \bigr) \Bigr) \nonumber \\[2mm]
= & \sum_{m \ge 0} (-1)^{m} 
  \gch H^{0} \Bigl( \BP^{1}, \, \BC_{- \nu} \otimes_{\BC} 
      \CE \bigl( H^{0} ( \DP ( t_{\xi} ), \, 
      \CO_{\DP ( t_{\xi} )} ( \lambda + \mu ) )^{\ast} \bigr) \Bigr) \nonumber \\[2mm]
= & \dfrac{e^{-\nu} \cdot \gch V^{-}_{t_\xi} (-w_{\circ}(\lambda + \mu)) - 
    e^{- \alpha_0} s_0 ( e^{-\nu} \cdot \gch V^{-}_{t_\xi} (-w_{\circ}(\lambda + \mu)) )}
       {1 - e^{- \alpha_0}} \nonumber \\[2mm]
= & \dfrac{ e^{-\nu} - s_0 ( e^{-\nu} )}
          {1 - e^{- \alpha_0}} 
    \gch V^{-}_{t_\xi} (-w_{\circ}(\lambda + \mu)) + 
    e^{- s_0 \nu} \gch V^{-}_{s_0 t_\xi} (-w_{\circ}(\lambda + \mu)) \nonumber \\[2mm]
= & \dfrac{ e^{-\nu} - s_0 ( e^{-\nu} )}{1 - e^{- \alpha_0}} 
    \Psi ( [ \CO_{\DP ( t_{\xi} )} ( \lambda ) ] ) ( \mu ) + 
    e^{- s_0 \nu} \Psi ( [ \CO_{\DP ( s_0 t_{\xi} )} ( \lambda ) ] ) ( \mu ), \label{conv-eqn}
\end{align}
where the first and fourth equalities 
follow by the Leray spectral sequence.
In particular, the term \eqref{conv-eqn} represents 
the image under $\Psi$ of the convolution of 
$[\BC_{-\mu} \otimes_{\BC} \CO_{\DP (t_{\xi})} ( \lambda )]$ 
with respect to $q_{0,t_{\xi}}$.
Therefore, from the injectivity of $\Psi$, 
we conclude that $\tT_{0}$ is induced by the pushforward of 
an $\ti{\Iw}$-equivariant inflated sheaf 
through $q_{0, t_{\xi}}$ for some $\xi \in Q^{\vee, +}$.

From the above, we deduce that 
the actions $\tT_i$, $i \in I_{\af}$, and 
$\be(\nu)$, $\nu \in P$, generate 
the convolution action of Schubert cells and 
the $\ti{\Iw}$-character twists of the (thin) affine flag manifold 
$G\pra{z} / \Iw$ on $\RDP$ (or rather, on $\DP$). 
In particular, the $T_{i}$, $i \in I_{\af}$, generate 
the nil-Hecke algebra of affine type by \cite[Sect.~3]{KK90}. 
Therefore, their commutation relations with $\te ( \vpi_i )$, $i \in I$, 
imply that the $T_{i}$, $i \in I_{\af}$, and 
the $\te ( \vpi_{i} )$, $i \in I$, satisfy the relations for $\dH$ 
(see also \cite[Sect.~3.4]{BFU}); we remark that 
their convention differs from ours by the twist by the Serre duality 
and line bundle twist \cite[Sects.~3.1 and 3.21]{BFU}. 

Finally, we complete the proof by observing that 
$\tT_0$ preserves $K_{\ti{\Iw}} ( \RDP )$ by inspection.
\end{proof}

%
%
\section{Proof of Theorem~\ref{thm:SMT}.}
\label{sec:prf-SMT}

%
\subsection{Affine Weyl group action.}
\label{subsec:Weyl}

Let $\CB$ be a regular crystal in the sense of \cite[Sect.~2.2]{Kas02} 
(or, a normal crystal in the sense of \cite[p.\,389]{HK});
for example, $\SLS(\lambda)$ for $\lambda \in P^{+}$ is a regular crystal
by Theorem~\ref{thm:isom}, and hence so is 
$\SLS(\lambda) \otimes \SLS(\mu)$ for $\lambda,\,\mu \in P^{+}$.
Then we know from \cite[Sect.~7]{Kas94} that 
the affine Weyl group $W_{\af}$ acts on $\CB$ as follows: 
for $b \in \CB$ and $i \in I_{\af}$, 
%
%
\begin{equation} \label{eq:W-act}
s_{i} \cdot b := 
\begin{cases}
f_{i}^{n}b & \text{if $n:=\pair{\wt(b)}{\alpha_{i}^{\vee}} \ge 0$}, \\[1.5mm]
e_{i}^{-n}b & \text{if $n:=\pair{\wt(b)}{\alpha_{i}^{\vee}} \le 0$}. 
\end{cases}
\end{equation}

Also, for $b \in \CB$ and $i \in I_{\af}$, 
we define $e_{i}^{\max}b=e_{i}^{\ve_{i}(b)}b$ and 
$f_{i}^{\max}b=f_{i}^{\vp_{i}(b)}b$, 
where $\ve_{i}(b):=\max\bigl\{n \ge 0 \mid e_{i}^{n}b \ne \bzero\bigr\}$ and 
$\vp_{i}(b):=\max\bigl\{n \ge 0 \mid f_{i}^{n}b \ne \bzero\bigr\}$; 
note that if $b \in \CB$ satisfies $e_{i}b = \bzero$ (resp., $f_{i}b = \bzero$), 
i.e., $\ve_{i}(b)=0$ (resp., $\vp_{i}(b)=0$), then 
$f_{i}^{\max}b = s_{i} \cdot b$ (resp., $e_{i}^{\max}b = s_{i} \cdot b$).

%
\subsection{Connected components of $\SLS(\lambda)$.}
\label{subsec:conn}

Let $\lambda \in P^{+}$, and write it as 
$\lambda = \sum_{i \in I} m_{i} \vpi_{i}$, with $m_{i} \in \BZ_{\ge 0}$; 
note that 
$J=\bigl\{i \in I \mid \pair{\lambda}{\alpha_{i}^{\vee}}=0\bigr\}= 
\bigl\{i \in I \mid m_{i} = 0\bigr\}$. 
We define $\Par(\lambda)$ to be the set of $I$-tuples of partitions 
$\brho = (\rho^{(i)})_{i \in I}$ such that $\rho^{(i)}$ is a partition of 
length (strictly) less than $m_{i}$ for each $i \in I$; 
a partition of length less than $0$ (or $1$)
is understood to be the empty partition $\emptyset$. 
Also, for $\brho = (\rho^{(i)})_{i \in I} \in \Par(\lambda)$, we set 
$|\brho|:=\sum_{i \in I} |\rho^{(i)}|$, where for a partition 
$\rho = (\rho_{1} \ge \rho_{2} \ge \cdots \ge \rho_{m})$, 
we set $|\rho| := \rho_{1}+\cdots+\rho_{m}$. 
We endow the set $\Par(\lambda)$ with a crystal structure as follows: 
for $\brho \in \Par(\lambda)$ and $i \in I_{\af}$, 
\begin{equation*}
e_{i} \brho = f_{i} \brho := \bzero, \quad 
\ve_{i} (\brho) = \vp_{i} (\brho) := -\infty, \quad  
\wt(\brho) := - |\brho| \delta.
\end{equation*}

We recall from \cite[Sect.~7]{INS} 
the relation between $\Par(\lambda)$ and 
the set $\Conn(\SLS(\lambda))$ of connected components 
of $\SLS(\lambda)$. 
We set
$\Turn(\lambda):=
 \bigl\{k/m_{i} \mid i \in I \setminus \J \text{ and }
 0 \le k \le m_{i}\bigr\}$. 
By \cite[Proposition~7.1.2]{INS}, 
each connected component of $\SLS(\lambda)$ 
contains a unique element of the form: 
%
%
\begin{equation} \label{eq:ext}
\bigl( \PJ(t_{\xi_{1}}),\,\dots,\,\PJ(t_{\xi_{s-1}}),\,e \,;\, 
  a_{0},\,a_{1},\,\dots,\,a_{s} \bigr), 
\end{equation}
where $s \ge 1$, $\xi_{1},\,\dots,\,\xi_{s-1} \in Q^{\vee}_{I \setminus \J}$ 
such that $\xi_{1} > \cdots > \xi_{s-1} > 0=:\xi_{s}$, and 
$a_{u} \in \Turn(\lambda)$ for all $0 \le u \le s$. 
For each element of the form \eqref{eq:ext} 
(or equivalently, each connected component of $\SLS(\lambda)$), 
we define an element $\brho = (\rho^{(i)})_{i \in I} \in \Par(\lambda)$ as follows. 
First, let $i \in I \setminus \J$; note that $m_{i} \ge 1$.  
For each $1 \le k \le m_{i}$, take $0 \le u \le s$ such that 
$a_{u}$ is contained in the interval $\bigl( (k-1)/m_{i},\,k/m_{i} \bigr]$. 
Then we define $\rho^{(i)}_{k}$ to be $\pair{\vpi_{i}}{\xi_{u}}$, 
the coefficient of $\alpha_{i}^{\vee}$ in $\xi_{u}$; 
we know from (the proof of) \cite[Proposition~7.2.1]{INS} that 
$\rho^{(i)}_{k}$ does not depend on the choice of $u$ above. 
Since $\xi_{1} > \cdots > \xi_{s-1} > 0=\xi_{s}$, we see that
$\rho^{(i)}_{1} \ge \cdots \ge 
\rho^{(i)}_{m_{i}-1} \ge \rho^{(i)}_{m_{i}}=0$. 
Thus, for each $i \in I \setminus \J$, 
we obtain a partition of length less than $m_{i}$. 
For $i \in \J$, we set $\rho^{(i)}:=\emptyset$. 
Thus we obtain an element 
$\brho = (\rho^{(i)})_{i \in I} \in \Par(\lambda)$, and hence 
a map from $\Conn(\SLS(\lambda))$ to $\Par(\lambda)$. 
Moreover, we know from \cite[Proposition~7.2.1]{INS} that 
this map is bijective; 
we denote by $\pi_{\brho} \in \SLS(\lambda)$ 
the element of the form \eqref{eq:ext} 
corresponding to $\brho \in \Par(\lambda)$ under this bijection. 
%
%
\begin{rem} \label{rem:ext}
Let $\brho = (\rho^{(i)})_{i \in I} \in \Par(\lambda)$, 
with $\rho^{(i)}=(\rho^{(i)}_{1} \ge \cdots )$ for $i \in I$; 
note that $\rho^{(i)}_{1}=0$ if $\rho^{(i)}=\emptyset$. 
It follows from the definition that 
%
%
\begin{equation} \label{eq:par0}
\iota(\pi_{\brho}) = \PJ(t_{\xi_{1}}), \quad 
\text{where} \quad \xi_{1} = \sum_{i \in I} 
\rho^{(i)}_{1} \alpha_{i}^{\vee} \in Q^{\vee}_{I \setminus \J}. 
\end{equation}
\end{rem}

For $\brho \in \Par(\lambda)$, we denote by 
$\SLS_{\brho}(\lambda)$ the connected component of $\SLS(\lambda)$ 
containing $\pi_{\brho}$. Also, we denote by 
$\SLS_{0}(\lambda)$ the connected component of $\SLS(\lambda)$ 
containing $\pi_{\lambda}=(e\,;\,0,1)$; 
note that $\pi_{\lambda}=\pi_{\brho}$ for $\brho=(\emptyset)_{i \in I}$. 
We know from \cite[Proposition~3.2.4]{INS} (and its proof) that
for each $\brho \in \Par(\lambda)$, 
there exists an isomorphism $\SLS_{\brho}(\lambda) \stackrel{\sim}{\rightarrow}
\bigl\{\brho\bigr\} \otimes \SLS_{0}(\lambda)$ of crystals, which maps 
$\pi_{\brho}$ to $\brho \otimes \pi_{\lambda}$. Hence we have
%
%
\begin{equation} \label{eq:isom}
\SLS(\lambda) = 
\bigsqcup_{\brho \in \Par(\lambda)} \SLS_{\brho}(\lambda) \cong 
\bigsqcup_{\brho \in \Par(\lambda)}
\bigl\{\brho\bigr\} \otimes \SLS_{0}(\lambda) \quad \text{as crystals}.
\end{equation}

The following lemma is shown by induction on 
the (ordinary) length $\ell(x)$ of $x$; for part (1), 
see also \cite[Remark~3.5.2]{NS16}
%
%
\begin{lem} \label{lem:Weyl}
\mbox{}
\begin{enu}
\item Let $\lambda \in P^{+}$. 
If $\pi \in \SLS(\lambda)$ is of the form \eqref{eq:ext}, 
then for $x \in W_{\af}$, 
%
%
\begin{equation} \label{eq:tx0}
x \cdot \pi = 
\bigl( \PJ(xt_{\xi_{1}}),\,\dots,\,\PJ(xt_{\xi_{s-1}}),\,\PJ(x) \,;\, 
  a_{0},\,a_{1},\,\dots,\,a_{s} \bigr). 
\end{equation}

\item Let $\lambda,\,\mu \in P^{+}$. 
Let $\brho \in \Par(\lambda)$, $\bchi \in \Par(\mu)$, and 
$\xi,\zeta \in Q^{\vee}$. Then, for $x \in W_{\af}$, 
%
%
\begin{equation} \label{eq:Weyl}
x \cdot \bigl( (t_{\xi} \cdot \pi_{\brho}) \otimes (t_{\zeta} \cdot \pi_{\bchi}) \bigr)
= (xt_{\xi} \cdot \pi_{\brho}) \otimes (xt_{\zeta} \cdot \pi_{\bchi}). 
\end{equation}
\end{enu}
\end{lem}

Let $\xi \in Q^{\vee}$. It follows from Lemma~\ref{lem:SiB}\,(3) that 
if $\pi = (x_{1},\,\dots,\,x_{s}\,;\,\ba) \in \SLS(\lambda)$, then 
%
%
\begin{equation} \label{eq:Txi}
T_{\xi}\pi:=
\bigl( \PJ(x_{1}t_{\xi}),\,\dots,\,\PJ(x_{s}t_{\xi})\,;\,\ba \bigr) \in \SLS(\lambda); 
\end{equation}
the map $T_{\xi} : \SLS(\lambda) \rightarrow \SLS(\lambda)$ 
is clearly bijective, with $T_{\xi}^{-1}=T_{-\xi}$. 
We can verify by the definitions that 
%
%
\begin{equation} \label{eq:Tx}
\begin{cases}
T_{\xi}e_{i}\pi= e_{i}T_{\xi}\pi, \quad 
T_{\xi}f_{i}\pi= f_{i}T_{\xi}\pi 
 & \text{for $\pi \in \SLS(\lambda)$ and $i \in I_{\af}$}, \\[1.5mm]
\ve_{i}(T_{\xi}\pi) = \ve_{i}(\pi), \quad 
\vp_{i}(T_{\xi}\pi) = \vp_{i}(\pi)
 & \text{for $\pi \in \SLS(\lambda)$ and $i \in I_{\af}$}, \\[1.5mm]
\wt (T_{\xi}\pi) = \wt (\pi) - \pair{\lambda}{\xi}\delta,
\end{cases}
\end{equation}
where $T_{\xi}\bzero$ is understood to be $\bzero$.
%
%
\begin{rem} \label{rem:Tx}
Let $\brho \in \Par(\lambda)$, and assume that 
$\pi_{\brho}$ is of the form \eqref{eq:ext}. 
For $\xi \in Q^{\vee}$, 
we see from \eqref{eq:tx0} and \eqref{eq:Txi} that 
\begin{equation*}
T_{\xi}\pi_{\brho} = 
\bigl( \PJ(t_{\xi_{1}+\xi}),\,\dots,\,\PJ(t_{\xi_{s-1}+\xi}),\,\PJ(t_{\xi}) \,;\, 
  a_{0},\,a_{1},\,\dots,\,a_{s} \bigr) = t_{\xi} \cdot \pi_{\brho}, 
\end{equation*}
which implies that $T_{\xi}\pi_{\brho} \in \SLS_{\brho}(\lambda)$. 
Therefore, it follows from \eqref{eq:Tx} that 
$T_{\xi}(\SLS_{\brho}(\lambda)) = \SLS_{\brho}(\lambda)$. 
\end{rem}
%
%
\subsection{Quantum Lakshmibai-Seshadri paths.}
\label{subsec:QLS}

Let $\lambda \in P^{+}$. 
Let $\cl : \BR \otimes_{\BZ} P_{\af} \twoheadrightarrow 
(\BR \otimes_{\BZ} P_{\af})/\BR\delta$ denote the canonical projection. 
For an element 
$\pi=(x_{1},\,\dots,\,x_{s}\,;\,a_{0},\,a_{1},\,\dots,\,a_{s}) \in \SLS(\lambda)$, 
we define a piecewise-linear, continuous map 
$\cl(\pi) : [0,1] \rightarrow (\BR \otimes_{\BZ} P_{\af})/\BR\delta$ 
by $(\cl(\pi))(t):=\cl(\ol{\pi}(t))$ for $t \in [0,1]$ 
(for $\ol{\pi}$, see \eqref{eq:olpi}). 
As explained in \cite[Sect.~6.2]{NS16}, 
the set $\bigl\{\cl(\pi) \mid \pi \in \SLS(\lambda)\bigr\}$ 
is identical to the set $\BB(\lambda)_{\cl}$ of all 
``projected (by $\cl$)'' LS paths of shape $\lambda$, 
introduced in \cite[(3.4)]{NS05} and \cite[page 117]{NS06} 
(see also \cite[Sect.~2.2]{LNSSS2}). 
Also, by \cite[Theorem~3.3]{LNSSS2}, $\BB(\lambda)_{\cl}$ is identical to 
the set $\QLS(\lambda)$ of all quantum LS paths of shape $\lambda$, 
introduced in \cite[Sect.~3.2]{LNSSS2}. 
We can endow the set $\BB(\lambda)_{\cl} = \QLS(\lambda)$ 
with a crystal structure with weights in $\cl(P_{\af})$ 
in such a way that
%
%
\begin{equation} \label{eq:cl}
\begin{cases}
e_{i}\cl(\pi) = \cl(e_{i}\pi), \quad 
f_{i}\cl(\pi) = \cl(f_{i}\pi) 
  & \text{for $\pi \in \SLS(\lambda)$ and $i \in I_{\af}$}, \\[1.5mm]
\wt(\cl(\pi))=\cl(\wt(\pi))
  & \text{for $\pi \in \SLS(\lambda)$}, \\[1.5mm]
\ve_{i}(\cl(\pi))=\ve_{i}(\pi), \quad 
\vp_{i}(\cl(\pi))=\vp_{i}(\pi)
  & \text{for $\pi \in \SLS(\lambda)$ and $i \in I_{\af}$}, 
\end{cases}
\end{equation}
where we understand that $\cl(\bzero) = \bzero$. 
The next theorem follows from 
\cite[Proposition~3.23 and Theorem~3.2]{NS05}. 
%
%
\begin{thm} \label{thm:QLS}
\mbox{}
\begin{enu}
\item For every $\lambda \in P^{+}$, 
the crystal $\QLS(\lambda)=\BB(\lambda)_{\cl}$ is connected. 

\item For every $\lambda,\,\mu \in P^{+}$, there exists an isomorphism 
$\QLS(\lambda) \otimes \QLS(\mu) \cong \QLS(\lambda+\mu)$ of crystals. 
In particular, $\QLS(\lambda) \otimes \QLS(\mu)$ is connected. 
\end{enu}
\end{thm}
%
%
\begin{lem} \label{lem:conn}
Let $\lambda,\,\mu \in P^{+}$, and set 
$\J:=\bigl\{ i \in I \mid \pair{\lambda}{\alpha_{i}^{\vee}}=0\bigr\}$, 
$\K:=\bigl\{ i \in I \mid \pair{\mu}{\alpha_{i}^{\vee}}=0\bigr\}$. 
Each connected component of $\SLS(\lambda) \otimes \SLS(\mu)$ contains 
an element of the form{\rm:} $(t_{\xi} \cdot \pi_{\brho}) \otimes \pi_{\bchi}$ 
for some $\xi \in Q^{\vee}_{I \setminus (\J \cup \K)}$, 
$\brho \in \Par(\lambda)$, and $\bchi \in \Par(\mu)$. 
\end{lem}

\begin{proof}
Let $\pi \in \SLS(\lambda)$ and $\eta \in \SLS(\mu)$. 
By Theorem~\ref{thm:QLS}\,(2), 
there exists a monomial $X$ in root operators on 
$\QLS(\lambda) \otimes \QLS(\mu)$ such that
$X \bigl( \cl(\pi) \otimes \cl(\eta) \bigr) = 
 \cl(\pi_{\lambda}) \otimes \cl(\pi_{\mu})$; recall that 
$\pi_{\lambda}=(e\,;\,0,\,1) \in \SLS(\lambda)$ and 
$\pi_{\mu}=(e\,;\,0,\,1) \in \SLS(\mu)$. 
It follows from \eqref{eq:cl} and the tensor product rule for crystals that 
$X(\pi \otimes \eta)$ is of the form $\pi_{1} \otimes \eta_{1}$ 
for some $\pi_{1} \in \SLS(\lambda)$ such that $\cl(\pi_{1})=\cl(\pi_{\lambda})$ and 
$\eta_{1} \in \SLS(\mu)$ such that $\cl(\eta_{1})=\cl(\pi_{\mu})$. 
Here, we see from \cite[Lemma~6.2.2]{NS16} that 
\begin{equation}
\begin{split}
& \bigl\{ \pi \in \SLS(\lambda) \mid \cl(\pi)=\cl(\pi_{\lambda})\bigr\} = 
  \bigl\{ t_{\xi} \cdot \pi_{\brho} \mid \brho \in \Par(\lambda),\,\xi \in Q^{\vee}\bigr\}, \\
& \bigl\{ \eta \in \SLS(\mu) \mid \cl(\eta)=\cl(\pi_{\mu})\bigr\} = 
  \bigl\{ t_{\zeta} \cdot \pi_{\bchi} \mid \bchi \in \Par(\mu),\,\zeta \in Q^{\vee}\bigr\}. 
\end{split}
\end{equation}
Therefore, $X(\pi \otimes \eta) = 
\bigl(t_{\xi_1} \cdot \pi_{\brho}\bigr) \otimes \bigl(t_{\zeta_1} \cdot \pi_{\bchi}\bigr)$ 
for some $\brho \in \Par(\lambda)$, $\xi_1 \in Q^{\vee}$ and 
$\bchi \in \Par(\mu)$, $\zeta_1 \in Q^{\vee}$. Also, by \eqref{eq:Weyl}, we have
$t_{-\zeta_1} \cdot \bigl( (t_{\xi_1} \cdot \pi_{\brho}) \otimes 
(t_{\zeta_1} \cdot \pi_{\bchi}) \bigr) = 
(t_{\xi_1-\zeta_1} \cdot \pi_{\brho}) \otimes \pi_{\bchi}$; 
we deduce from \eqref{eq:tx0} and \eqref{eq:PiJ1} that 
$t_{\xi_1-\zeta_1} \cdot \pi_{\brho} = t_{\xi_2} \cdot \pi_{\brho}$, 
with $\xi_2=[\xi_1-\zeta_1]^{\J}$, 
where $[\,\cdot\,]^{\J}:Q^{\vee} \twoheadrightarrow Q^{\vee}_{I \setminus \J}$ is 
the projection in \eqref{eq:prj}. 
We set $\gamma:=[\xi_{2}]_{\K} \in Q^{\vee}_{K}$, where 
$[\,\cdot\,]_{\K}:Q^{\vee} \twoheadrightarrow Q^{\vee}_{\K}$ is 
the projection defined as in \eqref{eq:prj}; 
we deduce from \eqref{eq:tx0} and \eqref{eq:PiJ1} that
$t_{-\gamma} \cdot \pi_{\bchi} = \pi_{\bchi}$.
In addition, we set $\xi:=\xi_{2}-\gamma$; 
notice that $\xi \in Q^{\vee}_{I \setminus (\J \cup \K)}$. 
Summarizing the above, we have
\begin{align*}
(t_{-\gamma}t_{-\zeta_1}) \cdot X(\pi \otimes \eta) & = 
(t_{-\gamma}t_{-\zeta_1}) \cdot 
  \bigl( (t_{\xi_1} \cdot \pi_{\brho}) \otimes 
         (t_{\zeta_1} \cdot \pi_{\bchi}) \bigr) = 
t_{-\gamma} \cdot 
  \bigl( (t_{\xi_1-\zeta_1} \cdot \pi_{\brho}) \otimes \pi_{\bchi} \bigr) \\
& = t_{-\gamma} \cdot 
  \bigl( (t_{\xi_2} \cdot \pi_{\brho}) \otimes \pi_{\bchi} \bigr) 
  = (t_{\xi_2-\gamma} \cdot \pi_{\brho}) \otimes (t_{-\gamma} \cdot \pi_{\bchi}) \\
& = (t_{\xi} \cdot \pi_{\brho}) \otimes \pi_{\bchi}. 
\end{align*}
Because the action of the affine Weyl group $W_{\af}$ 
on $\SLS(\lambda) \otimes \SLS(\mu)$ is defined 
by means of root operators (see \eqref{eq:W-act}), 
we conclude that $\pi \otimes \eta$ and 
$(t_{\xi} \cdot \pi_{\brho}) \otimes \pi_{\bchi}$ above are 
in the same connected component of $\SLS(\lambda) \otimes \SLS(\mu)$. 
This proves the lemma. 
\end{proof}

%
\subsection{Proof of Theorem~\ref{thm:SMT}.}
\label{subsec:prf-SMT}

Recall that $\lambda,\,\mu \in P^{+}$, and 
$\J=\bigl\{ i \in I \mid \pair{\lambda}{\alpha_{i}^{\vee}}=0\bigr\}$, 
$\K=\bigl\{ i \in I \mid \pair{\mu}{\alpha_{i}^{\vee}}=0\bigr\}$. 

%
\begin{prop} \label{prop:SMT1}
The set $\SM(\lambda+\mu) \sqcup \{\bzero\}$ is stable under the action of 
the root operators $e_{i}$, $f_{i}$, $i \in I_{\af}$, 
on $\SLS(\lambda) \otimes \SLS(\mu)$.
\end{prop}

\begin{proof}
We give a proof of the assertion only for $e_{i}$, $i \in I_{\af}$; 
the proof for $f_{i}$, $i \in I_{\af}$, is similar. 
Let $\pi \otimes \eta \in \SM(\lambda+\mu)$, and $i \in I_{\af}$. 
We may assume that $e_{i}(\pi \otimes \eta) \ne \bzero$. Then 
it follows from the tensor product rule for crystals that
\begin{equation*}
e_{i}(\pi \otimes \eta) = 
  \begin{cases}
  (e_{i}\pi) \otimes \eta & \text{if $\vp_{i}(\pi) \ge \ve_{i}(\eta)$}, \\[1.5mm]
  \pi \otimes (e_{i}\eta) & \text{if $\vp_{i}(\pi) < \ve_{i}(\eta)$};
  \end{cases}
\end{equation*}
recall from Remark~\ref{rem:SLS} that 
%
%
\begin{equation} \label{eq:vevp}
\ve_{i}(\eta) = - m^{\eta}_{i} \quad \text{and} \quad 
\vp_{i}(\pi) = H^{\pi}_{i}(1)-m^{\pi}_{i}.
\end{equation}
Let $x,y \in W_{\af}$ be such that 
$x \sige y$ and $\PJ(x) = \kappa(\pi)$, 
$\PK(y) = \iota(\eta)$ (see \eqref{eq:SM}); 
we write $x$ and $y$ as: 
\begin{equation} \label{eq:xy}
\begin{cases}
x=\kappa(\pi)x_1 & \text{with $x_1 \in (\WJ)_{\af}$}, \\[1mm]
y=\iota(\eta)y_1 & \text{with $y_1 \in (\WK)_{\af}$}.
\end{cases}
\end{equation}

\paragraph{Case 1.}
%
Assume that $\vp_{i}(\pi) \ge \ve_{i}(\eta)$, 
i.e., $e_{i}(\pi \otimes \eta) = (e_{i}\pi) \otimes \eta$. 
Note that $\kappa(e_{i}\pi)$ is equal either to $\kappa(\pi)$ or to $s_{i}\kappa(\pi)$ 
by the definition of the root operator $e_{i}$. 
If $\kappa(e_{i}\pi)=\kappa(\pi)$, then there is nothing to prove. 
Hence we may assume that $\kappa(e_{i}\pi) = s_{i} \kappa(\pi)$. 
Then we deduce from the definition of the root operator $e_{i}$ that 
the point $t_{1} = \min \bigl\{ t \in [0,1] \mid H^{\pi}_{i}(t)=m^{\pi}_{i} \bigr\}$ is 
equal to $1$, and hence
\begin{equation} \label{eq:case1a}
H^{\pi}_{i}(1) = m^{\pi}_{i} \quad \text{and} \quad
\pair{\kappa(\pi)\lambda}{\alpha_{i}^{\vee}} < 0.
\end{equation}
By \eqref{eq:vevp}, the equality in \eqref{eq:case1a}, and our assumption that 
$\vp_{i}(\pi) \ge \ve_{i}(\eta)$, we see that $m^{\eta}_{i} \ge 0$, and hence 
$m^{\eta}_{i}=0$; in particular, we obtain $\pair{\iota(\eta)\mu}{\alpha_{i}^{\vee}} \ge 0$.
In addition, it follows from \eqref{eq:case1a} that $\kappa(\pi)^{-1}\alpha_{i} \in 
- \bigl(\Delta^{+} \setminus \DeJ^{+}\bigr) + \BZ \delta$. 
Since $x_{1} \in (\WJ)_{\af}$ by \eqref{eq:xy}, we have 
%
%
\begin{equation} \label{eq:x-}
x^{-1}\alpha_{i}=
x_1^{-1}\kappa(\pi)^{-1}\alpha_{i} \in 
- \bigl(\Delta^{+} \setminus \DeJ^{+}\bigr) + \BZ \delta \subset -\Delta^{+} + \BZ \delta.
\end{equation}
Also, since $s_{i}\kappa(\pi) = \kappa(e_{i}\pi) \in (\WJu)_{\af}$ 
and $x_{1} \in (\WJ)_{\af}$, we have 
\begin{equation*}
\PJ(s_{i}x) = 
\PJ(s_{i}\kappa(\pi)x_{1}) = 
s_{i}\kappa(\pi) = \kappa(e_{i}\pi).
\end{equation*}
If $y^{-1}\alpha_{i} \in \Delta^{+} + \BZ \delta$, then
we see from Lemma~\ref{lem:dmd}\,(2) (applied to the case $J=\emptyset$) 
and \eqref{eq:x-} that $s_{i}x \sige y$ in $W_{\af}$. 
Therefore, $s_{i}x,\,y \in W_{\af}$ satisfy 
condition \eqref{eq:SM} for $e_{i}(\pi \otimes \eta) = (e_{i}\pi) \otimes \eta$. 
If $y^{-1}\alpha_{i} \in -\Delta^{+} + \BZ \delta$, then we see 
from Lemma~\ref{lem:dmd}\,(3) (applied to the case $J=\emptyset$) 
and \eqref{eq:x-} that $s_{i}x \sige s_{i}y$ in $W_{\af}$. 
We now claim that $\PK(s_{i}y)=\iota(\eta)$. 
Indeed, since $\pair{\iota(\eta)\mu}{\alpha_{i}^{\vee}} \ge 0$ as seen above, 
we have $\iota(\eta)^{-1}\alpha_{i} \in 
\bigl(\Delta^{+} \sqcup (-\DeK^{+})\bigr)+\BZ\delta$. 
In addition, since $y_{1} \in (\WK)_{\af}$, we deduce that 
$y^{-1}\alpha_{i} = y_{1}^{-1}\iota(\eta)^{-1}\alpha_{i}$ is contained in 
$\bigl(\Delta^{+} \sqcup (-\DeK^{+})\bigr)+\BZ\delta$. 
However, since $y^{-1}\alpha_{i} \in -\Delta^{+} + \BZ \delta$ by our assumption, 
we have $y^{-1}\alpha_{i} \in -\Delta_{K}^{+}+\BZ\delta$, 
which implies that $s_{y^{-1}\alpha_{i}} \in (\WK)_{\af}$. 
Therefore, we obtain
$\PK(s_{i}y) = 
\PK(ys_{y^{-1}\alpha_{i}}) = 
\PK(y) = \iota(\eta)$,
as desired. Thus, $s_{i}x,\,s_{i}y \in W_{\af}$ satisfy condition \eqref{eq:SM} 
for $e_{i}(\pi \otimes \eta) = (e_{i}\pi) \otimes \eta$. 

\paragraph{Case 2.}
%
Assume that $\vp_{i}(\pi) < \ve_{i}(\eta)$, i.e., 
$e_{i}(\pi \otimes \eta) = \pi \otimes (e_{i}\eta)$.
If $\iota(e_{i}\eta)=\iota(\eta)$, then there is nothing to prove. 
Hence we may assume that $\iota(e_{i}\eta) = s_{i} \iota(\eta)$. 
Then we deduce from the definition of the root operator $e_{i}$ that 
$\pair{\iota(\eta)\mu}{\alpha_{i}^{\vee}} < 0$; 
observe that with notation in \eqref{eq:t-e} and \eqref{eq:epi}, 
$t_{0}=0$, and $H^{\eta}_{i}(t)$ is strictly decreasing on 
$[t_{0},\,t_{1}] = [0,\,t_{1}]$. 
Thus we obtain $\iota(\eta)^{-1}\alpha_{i} \in 
- \bigl(\Delta^{+} \setminus \DeK^{+}\bigr) + \BZ \delta$. 
In addition, since $y_{1} \in (\WK)_{\af}$ (see \eqref{eq:xy}), 
we deduce that
\begin{equation*}
y^{-1}\alpha_{i} = y_1^{-1}\iota(\eta)^{-1}\alpha_{i} \in 
- \bigl(\Delta^{+} \setminus \DeK^{+}\bigr) + \BZ \delta \subset 
-\Delta^{+} + \BZ \delta, 
\end{equation*}
which implies that $y \sig s_{i}y$ by Lemma~\ref{lem:si}. 
Since $x \sige y$, we get $x \sig s_{i}y$. 
Also, since $s_{i}\iota(\eta) = \iota(e_{i}\eta) \in (\WKu)_{\af}$ 
and $y_{1} \in (\WK)_{\af}$, we have
\begin{equation*}
\PK(s_{i}y) = 
\PK(s_{i}\iota(\eta)y_{1}) = 
s_{i}\iota(\eta) = \iota(e_{i}\eta).
\end{equation*}
Thus, $x,\,s_{i}y \in W_{\af}$ satisfy 
condition \eqref{eq:SM} for $e_{i}(\pi \otimes \eta) = \pi \otimes (e_{i}\eta)$. 

This completes the proof of the proposition. 
\end{proof}

By this proposition, 
$\SM(\lambda+\mu)$ is a subcrystal of $\SLS(\lambda) \otimes \SLS(\mu)$. 
Hence our remaining task is 
to prove that $\SM(\lambda+\mu) \cong \SLS(\lambda+\mu)$ as crystals. 
Recall from Lemma~\ref{lem:conn} that 
each connected component of $\SLS(\lambda) \otimes \SLS(\mu)$ contains 
an element of the form: $(t_{\xi} \cdot \pi_{\brho}) \otimes \pi_{\bchi}$ 
for some $\xi \in Q^{\vee}_{I \setminus (\J \cup \K)}$, 
$\brho \in \Par(\lambda)$, and $\bchi \in \Par(\mu)$. 
%
%
\begin{prop} \label{prop:tx}
Let $\brho=(\rho^{(i)}) \in \Par(\lambda)$, 
$\bchi=(\chi^{(i)}) \in \Par(\mu)$, and 
$\xi=\sum c_{i}\alpha_{i}^{\vee} 
\in Q^{\vee}_{I \setminus (\J \cup \K)}$. 
Then, 
%
%
\begin{equation} \label{eq:tx1}
(t_{\xi} \cdot \pi_{\brho}) \otimes \pi_{\bchi} \in \SM(\lambda+\mu)
\end{equation}
if and only if
\begin{equation} \label{eq:tx2}
c_{i} \ge \chi^{(i)}_{1} \quad \text{\rm for all 
$i \in I \setminus (\J \cup \K)$}. 
\end{equation}
\end{prop}

\begin{proof}
We first show the ``only if'' part; assume that \eqref{eq:tx1} holds. 
We see that 
\begin{align*}
& \kappa(t_{\xi} \cdot \pi_{\brho}) = \PJ(t_{\xi}) \quad \text{by \eqref{eq:tx0}}; \\[1mm]
& \iota(\pi_{\bchi}) = \PK(t_{\zeta_{1}}), \quad 
\text{with} \quad \zeta_{1} = \sum_{i \in I} 
\chi^{(i)}_{1} \alpha_{i}^{\vee} \in Q^{\vee}_{I \setminus \K} \quad \text{by \eqref{eq:par0}}. 
\end{align*}
Since \eqref{eq:tx1} holds, there exist $x,\,y \in W_{\af}$ 
such that $x \sige y$ in $W_{\af}$, and such that
\begin{equation*}
\begin{cases}
\PJ(x) = 
\kappa(t_{\xi} \cdot \pi_{\brho}) =\PJ(t_{\xi}), & \\[1mm]
\PK(y) = 
\iota(\pi_{\bchi}) = \PK(t_{\zeta_{1}});
\end{cases}
\end{equation*}
we write $x$ and $y$ as:
\begin{equation*}
\begin{cases}
x=\PJ(t_{\xi})x_{1} 
  & \text{with $x_{1} \in (\WJ)_{\af}$}, \\[1mm]
y=\PK(t_{\zeta_{1}})y_{1} 
  & \text{with $y_{1} \in (\WK)_{\af}$}.
\end{cases}
\end{equation*}
Because $x \in W_{\af}$ is a lift of $\PJ(t_{\xi}) \in (\WJu)_{\af}$, 
it follows from Lemma~\ref{lem:lift} that $x = v t_{\xi+\gamma}$ 
for some $v \in \WJ$ and $\gamma \in \QJv$. 
Similarly, we have $y = v't_{\zeta_1+\gamma'}$ 
for some $v' \in \WK$ and $\gamma' \in \QKv$. 
Because $x \sige y$ in $W_{\af}$, we deduce from Lemma~\ref{lem:SiB}\,(1) 
(applied to the case $J=\emptyset$) that $\xi+\gamma \ge \zeta_1+\gamma'$, 
which implies \eqref{eq:tx2} since $\gamma,\,\gamma' \in Q^{\vee}_{\J \cup \K}$. 

We next show the ``if'' part; assume that \eqref{eq:tx2} holds. 
Recall that 
\begin{align*}
& \kappa(t_{\xi} \cdot \pi_{\brho}) = \PJ(t_{\xi}) \quad \text{by \eqref{eq:tx0}}; \\[1mm]
& \iota(\pi_{\bchi}) = \PK(t_{\zeta_{1}}), \quad 
\text{with} \quad \zeta_{1} = \sum_{i \in I} 
\chi^{(i)}_{1} \alpha_{i}^{\vee} \in Q^{\vee}_{I \setminus \K} \quad \text{by \eqref{eq:par0}}. 
\end{align*}
We set $\gamma:=\sum_{i \in \J \setminus \K} 
\chi^{(i)}_{1} \alpha_{i}^{\vee} \in \QJv$. 
Then it follows from \eqref{eq:tx2} that $\xi+\gamma \ge \zeta_{1}$ 
since $I \setminus K = (I \setminus (J \cup K)) \sqcup (J \setminus K)$. 
Hence we deduce from Lemma~\ref{lem:SiB}\,(2) that 
$x:=t_{\xi+\gamma} \sige t_{\zeta_1}=:y$ in $W_{\af}$. 
It is obvious that $\PK(y) = \iota(\pi_{\bchi})$. 
Also, since $\gamma \in \QJv$, we see from \eqref{eq:PiJ1} that 
$\PJ(x) = \PJ(t_{\xi}) = \kappa(t_{\xi} \cdot \pi_{\brho})$. 
Thus, $x$ and $y$ satisfy condition \eqref{eq:SM} 
for $(t_{\xi} \cdot \pi_{\brho}) \otimes \pi_{\bchi}$, 
which implies \eqref{eq:tx1}. 
This proves the proposition. 
\end{proof}
%
%
\begin{prop} \label{prop:conn}
Each connected component of $\SM(\lambda+\mu)$ 
contains a unique element of the form{\rm:}
$(t_{\xi} \cdot \pi_{\brho}) \otimes \pi_{\bchi}$
for some $\brho \in \Par(\lambda)$, 
$\bchi \in \Par(\mu)$, and 
$\xi \in Q^{\vee}_{I \setminus (\J \cup \K)}$ 
satisfying condition \eqref{eq:tx2} in Proposition~\ref{prop:tx}. 
Therefore, there exists a one-to-one correspondence between 
the set $\Conn(\SM(\lambda+\mu))$ of connected components of
$\SM(\lambda+\mu)$ and the set of triples 
$(\brho,\,\bchi,\,\xi) \in \Par(\lambda) \times \Par(\mu) \times 
Q^{\vee}_{I \setminus (\J \cup \K)}$ 
satisfying condition \eqref{eq:tx2} in Proposition~\ref{prop:tx}. 
\end{prop}

\begin{proof}
The ``existence'' part follows from 
Lemma~\ref{lem:conn} and Proposition~\ref{prop:tx}. 
Hence it suffices to prove the ``uniqueness'' part. 
Let $(\brho,\,\bchi,\,\xi)$ and $(\brho',\,\bchi',\,\xi')$ be elements 
in $\Par(\lambda) \times \Par(\mu) \times 
Q^{\vee}_{I \setminus (\J \cup \K)}$ 
satisfying condition \eqref{eq:tx2} in Proposition~\ref{prop:tx}, and 
suppose that 
$(t_{\xi} \cdot \pi_{\brho}) \otimes \pi_{\bchi}$ and 
$(t_{\xi'} \cdot \pi_{\brho'}) \otimes \pi_{\bchi'}$ are contained in the same 
connected component of $\SM(\lambda+\mu)$. 
Then there exists a monomial $X$ in root operators such that 
$X ( (t_{\xi} \cdot \pi_{\brho}) \otimes \pi_{\bchi} ) = 
(t_{\xi'} \cdot \pi_{\brho'}) \otimes \pi_{\bchi'}$. 
By the tensor product rule for crystals, we see that 
$X ( (t_{\xi} \cdot \pi_{\brho}) \otimes \pi_{\bchi} ) 
= X_{1} (t_{\xi} \cdot \pi_{\brho}) \otimes X_{2}\pi_{\bchi}$ 
for some monomials $X_{1}$, $X_{2}$ in root operators. 
Then we have 
$X_{1} (t_{\xi} \cdot \pi_{\brho}) = t_{\xi'} \cdot \pi_{\brho'}$, 
which implies that $t_{\xi} \cdot \pi_{\brho}$ and 
$t_{\xi'} \cdot \pi_{\brho'}$ are contained in the same connected component of 
$\SLS(\lambda)$, and hence so are $\pi_{\brho}$ and $\pi_{\brho'}$. 
Therefore, by the uniqueness of an element of 
the form \eqref{eq:ext} in a connected 
component of $\SLS(\lambda)$ (see Section~\ref{subsec:conn}), 
we deduce that $\brho = \brho'$. Similarly, we obtain $\bchi=\bchi'$. 
Suppose, for a contradiction, that $\xi \ne \xi'$; we may assume that 
for some $k \in I \setminus (\J \cup \K)$, 
the coefficient of $\alpha_{k}^{\vee}$ in $\xi$ is 
greater than that in $\xi'$, i.e., 
the coefficient of $\alpha_{k}^{\vee}$ in $\xi'-\xi$ is a negative integer. 
Because $(t_{\xi} \cdot \pi_{\brho}) \otimes \pi_{\bchi}$ and 
$(t_{\xi'} \cdot \pi_{\brho'}) \otimes \pi_{\bchi'} = 
(t_{\xi'} \cdot \pi_{\brho}) \otimes \pi_{\bchi}$ are contained 
in the same connected component, there exists a monomial $Y$ in root operators 
such that $Y \bigl((t_{\xi} \cdot \pi_{\brho}) \otimes \pi_{\bchi}\bigr) 
   = (t_{\xi'} \cdot \pi_{\brho}) \otimes \pi_{\bchi}
   = (t_{\xi+(\xi'-\xi)} \cdot \pi_{\brho}) \otimes \pi_{\bchi}$. 
Here, the same argument as in the proof of \cite[Lemma~7.1.4]{INS} 
(or, as in the proof of \cite[Proposition~7.1.2]{INS}) shows that 
\begin{equation*}
Y^{N} \Bigl((t_{\xi} \cdot \pi_{\brho}) \otimes \pi_{\bchi}\Bigr) 
 = (t_{\xi+N(\xi'-\xi)} \cdot \pi_{\brho}) \otimes \pi_{\bchi}
\quad \text{for all $N \in \BZ_{\ge 1}$}. 
\end{equation*}
Since this element is contained in $\SM(\lambda+\mu)$ 
for all $N \in \BZ_{\ge 1}$ by Proposition~\ref{prop:SMT1},
it follows from Proposition~\ref{prop:tx} that 
the coefficient of $\alpha_{k}^{\vee}$ in $\xi+N(\xi'-\xi)$ 
is greater than or equal to $\chi^{(k)}_{1}$ for all $N \ge 1$. 
This contradicts the fact that 
the coefficient of $\alpha_{k}^{\vee}$ in $\xi'-\xi$ is a negative integer. 
This proves the proposition. 
\end{proof}

Now, we write $\lambda$ and $\mu$ as: 
$\lambda=\sum_{i \in I} m_{i} \vpi_{i}$ and 
$\mu=\sum_{i \in I} n_{i} \vpi_{i}$.
For $(\brho,\,\bchi,\,\xi) \in \Par(\lambda) \times \Par(\mu) \times 
 Q^{\vee}_{I \setminus (\J \cup \K)}$ satisfying
\eqref{eq:tx2}, define $\bvrho = (\vrho^{(i)})_{i \in I} \in \Par(\lambda+\mu)$ 
as follows. Write $\brho$, $\bchi$, and $\xi$ as: 
\begin{align*}
& \brho = (\rho^{(i)})_{i \in I}, \quad \text{with} \quad
  \rho^{(i)}=(\rho^{(i)}_{1} \ge \cdots \ge \rho^{(i)}_{m_i-1} \ge 0) \quad
  \text{for $i \in I$}, \\
& \bchi = (\chi^{(i)})_{i \in I}, \quad \text{with} \quad
  \chi^{(i)}=(\chi^{(i)}_{1} \ge \cdots \ge \chi^{(i)}_{n_i-1} \ge 0) \quad
  \text{for $i \in I$}, \\
& \xi=\sum_{i \in I \setminus (\J \cup \K)} 
  c_{i}\alpha_{i}^{\vee}; \quad \text{recall that $c_{i} \ge \chi^{(i)}_{1}$ 
  for all $i \in I \setminus (\J \cup \K)$}.
\end{align*}
Let $i \in I$. 
\begin{itemize}
\item If $i \in \J \cap \K$ (note that $m_{i}=n_{i}=0$), 
we set $\vrho^{(i)} := \emptyset$; 

\item if $i \in \J \setminus \K$ (note that $m_{i}=0$), 
we set $\vrho^{(i)}:=\chi^{(i)}$, which is a partition of 
length less than $n_{i} = 0+n_{i}= m_{i}+n_{i}$; 

\item if $i \in \K \setminus \J$ (note that $n_{i}=0$), 
we set $\vrho^{(i)}:=\rho^{(i)}$, which is a partition of length less 
than $m_{i} =m_{i}+0= m_{i}+n_{i}$;

\item if $i \in I \setminus (\J \cup \K)$, we set 
\begin{equation*}
\vrho^{(i)}=
(\rho^{(i)}_{1}+c_{i} \ge \cdots \ge \rho^{(i)}_{m_i-1}+c_{i} \ge c_{i} \ge 
\chi^{(i)}_{1} \ge \cdots \ge \chi^{(i)}_{n_i-1}),
\end{equation*}
which is a partition of length less than 
$(m_{i}-1)+1+(n_{i}-1) + 1=m_{i}+n_{i}$; 
note that $|\vrho^{(i)}| = | \rho^{(i)} | + |\chi^{(i)} | + m_{i}c_{i}$. 
\end{itemize}
It follows that $\bvrho=(\vrho^{(i)})_{i \in I} \in \Par(\lambda+\mu)$. 
Thus we obtain a map $\Theta$ from the set of 
those $(\brho,\,\bchi,\,\xi) \in \Par(\lambda) \times \Par(\mu) \times 
 Q^{\vee}_{I \setminus (\J \cup \K)}$ satisfying
\eqref{eq:tx2} to the set $\Par(\lambda+\mu)$; 
we can easily deduce that the map $\Theta$ is bijective.
Also, by direct calculation, we have
\begin{align*}
\wt\Bigl( \underbrace{(t_{\xi} \cdot \pi_{\brho}) \otimes \pi_{\bchi}}_{\text{satisfying \eqref{eq:tx2}}} \Bigr)
 & = \wt (t_{\xi} \cdot \pi_{\brho}) + \wt (\pi_{\bchi}) 
   = t_{\xi}\bigl(\lambda-|\brho|\delta\bigr) + \bigl(\mu-|\bchi|\delta\bigr) \\
 & = (\lambda+\mu) - \bigl( |\brho| + \pair{\lambda}{\xi} + |\bchi| \bigr)\delta \\[2mm]
 & = (\lambda+\mu) - 
   \left( |\brho| + \sum_{i \in I \setminus (J \cup K)} m_{i}c_{i} + |\bchi| \right)\delta \\[2mm]
 & = (\lambda+\mu) - \bigl| \Theta(\brho,\,\bchi,\,\xi) \bigr|\delta.
\end{align*}

We claim that the connected component of $\SM(\lambda+\mu)$ containing 
$(t_{\xi} \cdot \pi_{\brho}) \otimes \pi_{\bchi}$ is isomorphic, as a crystal, to 
$\bigl\{\Theta(\brho,\,\bchi,\,\xi)\bigr\} \otimes \SM_{0}(\lambda+\mu)$, 
where $\SM_{0}(\lambda+\mu)$ denotes the connected component of $\SM(\lambda+\mu)$ 
containing $\pi_{\lambda} \otimes \pi_{\mu} = 
(e\,;\,0,1) \otimes (e\,;\,0,1) \in \SLS(\lambda) \otimes \SLS(\mu)$. 
Indeed, let us consider the composite of the following bijections: 
\begin{align*}
& \SLS_{\brho}(\lambda) \otimes \SLS_{\bchi}(\mu) 
  \stackrel{T_{-\xi} \otimes \id}{\longrightarrow} 
  \SLS_{\brho}(\lambda) \otimes \SLS_{\bchi}(\mu) 
  \quad \text{(see \eqref{eq:Tx} and Remark~\ref{rem:Tx})} \\
& \quad \stackrel{\sim}{\rightarrow} 
   \bigl(\bigl\{\brho\bigr\} \otimes \SLS_{0}(\lambda)\bigr) \otimes 
   \bigl(\bigl\{\bchi\bigr\} \otimes \SLS_{0}(\mu)\bigr) \quad
   \text{(see the comment preceding \eqref{eq:isom})} \\
& \quad \stackrel{\sim}{\rightarrow}
   \bigl(\bigl\{\brho\bigr\} \otimes \bigl\{\bchi\bigr\}\bigr) \otimes 
   \bigl(\SLS_{0}(\lambda) \otimes \SLS_{0}(\mu)\bigr) \quad 
   \text{(by the tensor product rule for crystals)} \\
& \quad \stackrel{\sim}{\rightarrow}
   \bigl\{\Theta(\brho,\,\bchi,\,\xi)\bigr\} \otimes 
   \bigl(\SLS_{0}(\lambda) \otimes \SLS_{0}(\mu)\bigr),
\end{align*}
where the last map sends 
$(\brho \otimes \bchi) \otimes (\pi \otimes \eta)$ to 
$\Theta(\brho,\,\bchi,\,\xi) \otimes (\pi \otimes \eta)$ 
for each $\pi \in \SLS_{0}(\lambda)$ and 
$\eta \in \SLS_{0}(\mu)$. We deduce 
by \eqref{eq:Tx} and the tensor product rule for crystals 
that the composite of these bijections is 
an isomorphism of crystals, which sends 
$(t_{\xi} \cdot \pi_{\brho}) \otimes \pi_{\bchi}$ to 
$\Theta(\brho,\,\bchi,\,\xi) \otimes (\pi_{\lambda} \otimes \pi_{\mu})$. 
Therefore, the connected component of $\SM(\lambda+\mu)$ containing 
$(t_{\xi} \cdot \pi_{\brho}) \otimes \pi_{\bchi}$ is mapped to 
$\bigl\{\Theta(\brho,\,\bchi,\,\xi)\bigr\} \otimes \SM_{0}(\lambda+\mu)$ 
under this isomorphism of crystals. 
It follows from Proposition~\ref{prop:conn} and 
the bijectivity of $\Theta$ that 
%
%
\begin{equation} \label{eq:conn}
\SM(\lambda+\mu) \cong \bigsqcup_{\bvrho \in \Par(\lambda+\mu)} 
\bigl\{\bvrho\bigr\} \otimes \SM_{0}(\lambda+\mu). 
\end{equation}
%
%
\begin{prop} \label{prop:BN}
As crystals, $\SM_{0}(\lambda+\mu) \cong \SLS_{0}(\lambda+\mu)$.
\end{prop}

\begin{proof}
Write $\lambda$ and $\mu$ as: 
$\lambda = \sum_{i \in I} m_{i}\vpi_{i}$ 
with $m_{i} \in \BZ_{\ge 0}$, and
$\mu = \sum_{i \in I} n_{i}\vpi_{i}$ 
with $n_{i} \in \BZ_{\ge 0}$, respectively. 
We know from \cite[Conjecture~13.1\,(iii)]{Kas02}, 
which is proved in \cite[Remark~4.17]{BN}, that 
there exists an isomorphism 
$\CB(\lambda+\mu) \stackrel{\sim}{\rightarrow} 
\bigotimes_{i \in I} \CB((m_{i}+n_{i})\vpi_{i})$ of crystals, 
which maps $u_{\lambda+\mu}$ to 
$\bigotimes_{i \in I} u_{(m_{i}+n_{i})\vpi_{i}}$; 
the restriction of this isomorphism to 
$\CB_{0}(\lambda+\mu) \subset \CB(\lambda+\mu)$ 
gives an embedding $\CB_{0}(\lambda+\mu) \hookrightarrow 
\bigotimes_{i \in I} \CB_{0}((m_{i}+n_{i})\vpi_{i})$ of crystals. 
Also, we know from \cite[Conjecture~13.2\,(iii)]{Kas02}, 
which is proved in \cite[Remark~4.17]{BN}, that 
for each $i \in I$, there exists an embedding 
$\CB_{0}((m_{i}+n_{i})\vpi_{i}) \hookrightarrow 
\CB(\vpi_{i})^{\otimes (m_{i}+n_{i})}$ of crystals, 
which maps $u_{(m_{i}+n_{i})\vpi_{i}}$ to 
$u_{\vpi_{i}}^{\otimes (m_{i}+n_{i})}$; recall from 
\cite[Proposition~5.4]{Kas02} that $\CB(\vpi_{i})$ is connected. 
Thus we obtain an embedding 
\begin{equation*}
\CB_{0}(\lambda+\mu) \hookrightarrow 
\bigotimes_{i \in I} \CB_{0}((m_{i}+n_{i})\vpi_{i})
\hookrightarrow 
\bigotimes_{i \in I} \CB(\vpi_{i})^{\otimes (m_{i}+n_{i})}
\end{equation*}
of crystals, which maps $u_{\lambda+\mu}$ to 
$\bigotimes_{i \in I} u_{\vpi_{i}}^{\otimes (m_{i}+n_{i})}$. 
Here, we recall from \cite[Sect.~10]{Kas02} that 
for each $j,\,k \in I$, 
there exists an isomorphism $\CB(\vpi_{j}) \otimes \CB(\vpi_{k}) 
\stackrel{\sim}{\rightarrow} 
\CB(\vpi_{k}) \otimes \CB(\vpi_{j})$ of crystals, which maps 
$u_{\vpi_{j}} \otimes u_{\vpi_{k}}$ to 
$u_{\vpi_{k}} \otimes u_{\vpi_{j}}$. 
Hence we obtain an isomorphism of crystals 
\begin{equation*}
\bigotimes_{i \in I} \CB(\vpi_{i})^{\otimes (m_{i}+n_{i})} 
\stackrel{\sim}{\rightarrow} 
\left(\bigotimes_{i \in I} \CB(\vpi_{i})^{\otimes m_{i}}\right)
\otimes 
\left(\bigotimes_{i \in I} \CB(\vpi_{i})^{\otimes n_{i}}\right) = : \CB, 
\end{equation*}
which maps $\bigotimes_{i \in I} u_{\vpi_{i}}^{\otimes (m_{i}+n_{i})}$
to $\left(\bigotimes_{i \in I} u_{\vpi_{i}}^{\otimes m_{i}}\right) \otimes 
\left(\bigotimes_{i \in I} u_{\vpi_{i}}^{\otimes n_{i}}\right)=:b$. 
From these, we obtain an embedding $\CB_{0}(\lambda+\mu) \hookrightarrow \CB$ of crystals, 
which maps $u_{\lambda+\mu}$ to $b$. 
Similarly, we obtain an embedding 
$\CB_{0}(\lambda) \otimes \CB_{0}(\mu) \hookrightarrow \CB$ of crystals, 
which maps $u_{\lambda} \otimes u_{\mu}$ to $b$. 
Consequently, there exists an isomorphism of crystals from 
$\CB_{0}(\lambda+\mu)$ to the connected component 
(denoted by $\CS_{0}(\lambda+\mu)$) of $\CB_{0}(\lambda) \otimes \CB_{0}(\mu)$ 
containing $u_{\lambda} \otimes u_{\mu}$, which maps $u_{\lambda+\mu}$ to 
$u_{\lambda} \otimes u_{\mu}$. 
Now, by Theorem~\ref{thm:isom}, we have an isomorphism 
$\CB_{0}(\lambda+\mu) \stackrel{\sim}{\rightarrow} \SLS_{0}(\lambda+\mu)$ of crystals, 
which maps $u_{\lambda+\mu}$ to $\pi_{\lambda+\mu}$. In addition, we have 
an isomorphism $\CB_{0}(\lambda) \otimes \CB_{0}(\mu) \stackrel{\sim}{\rightarrow} 
\SLS_{0}(\lambda) \otimes \SLS_{0}(\mu)$ of crystals, 
which maps $u_{\lambda} \otimes u_{\mu}$ to $\pi_{\lambda} \otimes \pi_{\mu}$; 
by restriction, we obtain an isomorphism of crystals from 
$\CS_{0}(\lambda+\mu)$ to $\SM_{0}(\lambda+\mu)$. 
Summarizing, we obtain the following 
isomorphism of crystals: 
\begin{equation*}
\begin{array}{ccccccc}
\SLS_{0}(\lambda+\mu) & \stackrel{\sim}{\rightarrow} &
\CB_{0}(\lambda+\mu) & \stackrel{\sim}{\rightarrow} &
\CS_{0}(\lambda+\mu) & \stackrel{\sim}{\rightarrow} &
\SM_{0}(\lambda+\mu), \\[3mm]
\pi_{\lambda+\mu} & \mapsto & 
u_{\lambda+\mu} & \mapsto & 
u_{\lambda} \otimes u_{\mu} & \mapsto &
\pi_{\lambda} \otimes \pi_{\mu}.
\end{array}
\end{equation*}
This proves the proposition. 
\end{proof}

By using \eqref{eq:conn}, Proposition~\ref{prop:BN}, and 
\eqref{eq:isom} (with $\lambda$ replaced by $\lambda+\mu$), 
we conclude that 
\begin{align*}
\SM(\lambda+\mu) & \cong 
\bigsqcup_{\bvrho \in \Par(\lambda+\mu)} 
\bigl\{\bvrho\bigr\} \otimes \SM_{0}(\lambda+\mu) \\[3mm]
& \cong 
\bigsqcup_{\bvrho \in \Par(\lambda+\mu)} 
\bigl\{\bvrho\bigr\} \otimes \SLS_{0}(\lambda+\mu)
\cong \SLS(\lambda+\mu)
\end{align*}
as crystals. This completes the proof of Theorem~\ref{thm:SMT}. 
%
%
\begin{cor} \label{cor:Sext}
For each $\bvrho \in \Par(\lambda+\mu)$, 
the element $\pi_{\bvrho} \in \SLS(\lambda+\mu)$ 
is mapped to $(t_{\xi} \cdot \pi_{\brho}) \otimes \pi_{\bchi} \in \SM(\lambda+\mu)$ 
for some $\xi \in Q^{\vee}_{I \setminus (\J \cup \K)}$ and 
$\brho \in \Par(\lambda)$, $\bchi \in \Par(\mu)$ 
satisfying \eqref{eq:tx2} under the isomorphism 
$\SLS(\lambda+\mu) \cong \SM(\lambda+\mu)$ of crystals in Theorem~\ref{thm:SMT}. 
\end{cor}
%
%
\section{Proof of Propositions~\ref{prop:DC} and \ref{prop:DC2}.}
\label{sec:prf-DC}
Recall that $\lambda,\,\mu \in P^{+}$, and that
$\J = \bigl\{ i \in I \mid \pair{\lambda}{\alpha_{i}^{\vee}}=0\bigr\}$, 
$\K = \bigl\{ i \in I \mid \pair{\mu}{\alpha_{i}^{\vee}}=0\bigr\}$. 
%
%
\subsection{Proof of Proposition~\ref{prop:DC}.}
\label{subsec:prf-DC}

The ``if'' part is obvious from 
the definition of defining chains and condition \eqref{eq:SM}. 
Let us prove the ``only if'' part. 
Assume that $\pi \otimes \eta \in \SM(\lambda+\mu)$, and 
write $\pi$ and $\eta$ as: 
$\pi=(x_{1},\,\dots,\,x_{s}\,;\,\ba) \in \SLS(\lambda)$ and 
$\eta=(y_{1},\,\dots,\,y_{p}\,;\,\bb) \in \SLS(\mu)$, respectively. 
It follows from \eqref{eq:SM} that 
there exist $x_{s}',\,y_{1}' \in W_{\af}$ 
such that $x_{s}' \sige y_{1}'$ in $W_{\af}$, 
and such that $\PJ(x_{s}')=x_{s}$, 
$\PK(y_{1}')=y_{1}$;
we write $x_{s}'=x_{s}z_{1}$ for some $z_{1} \in (\WJ)_{\af}$, and 
$y_{1}'=y_{1}z_{2}$ for some $z_{2} \in (\WK)_{\af}$. Now we set
\begin{equation*}
\begin{cases}
 x_{u}':=x_{u}z_{1} & \text{for $1 \le u \le s$}, \\[1.5mm]
 y_{q}':=y_{q}z_{2} & \text{for $1 \le q \le p$}. 
\end{cases}
\end{equation*}
Because $x_{1} \sige x_{2} \sige \cdots \sige x_{s}$ in $(\WJu)_{\af}$ 
by the definition of semi-infinite LS paths, 
it follows from Lemma~\ref{lem:SiL} that 
$x_{1}' \sige x_{2}' \sige \cdots \sige x_{s}'$ in $W_{\af}$. 
Similarly, we see that 
$y_{1}' \sige y_{2}' \sige \cdots \sige y_{p}'$ in $W_{\af}$.
Combining these inequalities with the inequality $x_{s}' \sige y_{1}'$, 
we obtain $x_{1}' \sige \cdots \sige x_{s}' \sige 
y_{1}' \sige \cdots \sige y_{p}'$ in $W_{\af}$. 
Since $\PJ(x_{u}')=x_{u}$ for all $1 \le u \le s$, 
and $\PJ(y_{q}')=y_{q}$ for all $1 \le q \le p$, 
the sequence $x_{1}',\,\dots,\,x_{s}',\,y_{1}',\,\dots,\,y_{p}'$ is 
a defining chain for $\pi \otimes \eta$.  
This completes the proof of Proposition~\ref{prop:DC}. 
%
%
\subsection{Proof of Proposition~\ref{prop:DC2}.}
\label{subsec:prf-DC2}

We write $\pi$ and $\eta$ as:
$\pi=(x_{1},\,\dots,\,x_{s}\,;\,\,\ba) \in \SLS(\lambda)$ and 
$\eta=(y_{1},\,\dots,\,y_{p}\,;\,\,\bb) \in \SLS(\mu)$, respectively. 
First, we prove the ``only if'' part. 
Take an (arbitrary) defining chain 
$x_{1}',\,\dots,\,x_{s}',\,y_{1}',\,\dots,\,y_{p}' \in W_{\af}$ 
for $\pi \otimes \eta$. It suffices to show the following claim. 

\begin{claim} \label{c:DC2-1}
Let $x \in W_{\af}$ be such that $y_{p}' \sige x${\rm;} note that 
$\kappa(\eta)=y_{p}=\PK(y_{p}') \sige \PK(x)$ by Lemma~\ref{lem:611}. 
Then, $\kappa(\pi) \sige \PJ(\Deo{\eta}{x})$.
\end{claim}

\noindent
{\it Proof of Claim~\ref{c:DC2-1}.}
For the given $x \in W_{\af}$, 
we define $\ti{y}_{p},\,\ti{y}_{p-1},\,\dots,\,\ti{y}_{1} = \Deo{\eta}{x}$ 
as in \eqref{eq:lift0}. We show by descending induction on $q$ that 
\begin{equation} \label{eq:DC2-1}
y_{q}' \sige \ti{y}_{q} \quad \text{for all $1 \le q \le p$}.
\end{equation}
Because $y_{p}' \sige x$ and $\PK(y_{p}')=y_{p}$, it follows that 
$y_{p}' \in \Lif{x}{y_{p}}$, 
and hence $y_{p}' \sige \min \Lif{x}{y_{p}} = \ti{y}_{p}$. 
Assume that $q < p$. Since $y_{q}' \sige y_{q+1}'$ 
by the definition of defining chains, and 
since $y_{q+1}' \sige \ti{y}_{q+1}$ by our induction hypothesis, 
we obtain $y_{q}' \sige \ti{y}_{q+1}$. 
In addition, we have $\PK(y_{q}')=y_{q}$. 
From these, we deduce that 
$y_{q}' \in \Lif{\ti{y}_{q+1}}{y_{q}}$, and hence 
$y_{q}' \sige \min \Lif{\ti{y}_{q+1}}{y_{q}}
= \ti{y}_{q}$. Thus we have shown \eqref{eq:DC2-1}. 
Hence we have $x_{s}' \sige y_{1}' \sige \ti{y}_{1} = \Deo{\eta}{x}$ 
by the assumption. Therefore, it follows from Lemma~\ref{lem:611} that
%
%
\begin{equation} \label{eq:kappa-Deo}
\kappa(\pi) = x_{s} = \PJ(x_{s}') \sige \PJ(\Deo{\eta}{x}).
\end{equation}
This proves Claim~\ref{c:DC2-1}. \bqed

Next, we prove the ``if'' part. 
We define $\ti{y}_{p}$, $\ti{y}_{p-1}$, $\dots$, $\ti{y}_{1} = \Deo{\eta}{x}$ 
as in \eqref{eq:lift0}. By the definitions, we have
%
%
\begin{equation} \label{eq:DC2-2}
\begin{cases}
\Deo{\eta}{x}
 = \ti{y}_{1} \sige \ti{y}_{2} \sige \cdots \sige \ti{y}_{p} \sige x, & \\[1mm]
\PK(\ti{y}_{q}) = y_{q} 
\quad \text{for $1 \le q \le p$}.
\end{cases}
\end{equation}
Write $\Deo{\eta}{x} \in W_{\af}$ as: 
$\Deo{\eta}{x} = \PJ(\Deo{\eta}{x})z$ with $z \in (\WJ)_{\af}$. 
Since $\kappa(\pi) \sige \PJ(\Deo{\eta}{x})$ by the assumption, 
we deduce from Lemma~\ref{lem:SiL} that $x_{s}':=\kappa(\pi)z \sige 
\PJ(\Deo{\eta}{x})z = \Deo{\eta}{x} = \ti{y}_{1}$. 
Similarly, if we set $x_{u}':=x_{u}z$ for $1 \le u \le s$, then 
we have 
%
%
\begin{equation} \label{eq:DC2-3}
\begin{cases}
x_{1}' \sige x_{2}' \sige \cdots \sige x_{s}' \ 
 (\sige \Deo{\eta}{x} = \ti{y}_{1}), & \\[1mm]
\PJ(x_{u}') = x_{u} 
\quad \text{for $1 \le u \le s$}.
\end{cases}
\end{equation}
Concatenating the sequences in 
\eqref{eq:DC2-2} and \eqref{eq:DC2-3}, 
we obtain a defining chain
%
%
\begin{equation} \label{eq:minDC}
x_{1}' \sige x_{2}' \sige \cdots \sige x_{s}' \sige 
\ti{y}_{1} \sige \ti{y}_{2} \sige \cdots \sige \ti{y}_{p}
\end{equation}
for $\pi \otimes \eta \in \SLS(\lambda) \otimes \SLS(\mu)$. 
This completes the proof of Proposition~\ref{prop:DC2}. 
%
%
\section{Proof of Theorem~\ref{thm:Dem}.}
\label{sec:prf-Dem}

Recall that $\lambda,\,\mu \in P^{+}$, and that
$\J = \bigl\{ i \in I \mid \pair{\lambda}{\alpha_{i}^{\vee}}=0\bigr\}$, 
$\K = \bigl\{ i \in I \mid \pair{\mu}{\alpha_{i}^{\vee}}=0\bigr\}$, and 
$S = \bigl\{ i \in I \mid \pair{\lambda+\mu}{\alpha_{i}^{\vee}}=0\bigr\} = \J \cap \K$.
%
%
\subsection{Proof of \eqref{eq:Db} $\Leftrightarrow$ \eqref{eq:Dc}.}
\label{subsec:23}

We prove the implication \eqref{eq:Db} $\Rightarrow$ \eqref{eq:Dc}. 
Let $x_{1}',\,\dots,\,x_{s}',\,y_{1}',\,\dots,\,y_{p}'=:y$ 
be a defining chain for $\pi \otimes \eta$ 
such that $\PS(y) \sige \PS(x)$. 
Write $x$ as: $x = \PS(x)z$ for some $z \in (\WS)_{\af}$; 
note that $(\WS)_{\af} \subset (\WJ)_{\af} \cap (\WK)_{\af}$. 
We deduce from Lemma~\ref{lem:SiL} that 
\begin{equation*}
\PS(x_{1}')z,\,\dots,\,\PS(x_{s}')z,\,
\PS(y_{1}')z,\,\dots,\,\PS(y_{p}')=\PS(y)z
\end{equation*}
is also a defining chain for $\pi \otimes \eta$ 
such that $\PS(\PS(y)z) = \PS(y) \sige \PS(x)$. 
Hence we may assume from the beginning that $y \sige x$. 
We deduce from Lemma~\ref{lem:611} that $\kappa(\eta) = \PK(y) \sige \PK(x)$. 
Also, the inequality $\kappa(\pi) \sige \PJ(\Deo{\eta}{x})$ was
shown in Claim~\ref{c:DC2-1} in the proof of Proposition~\ref{prop:DC2}. 

The implication \eqref{eq:Dc}\ $\Rightarrow$ \eqref{eq:Db} 
follows from the fact that the defining chain 
\eqref{eq:minDC} for $\pi \otimes \eta$ 
satisfies the desired condition in \eqref{eq:Db}. 

%
\subsection{Proof of \eqref{eq:Da} $\Leftrightarrow$ \eqref{eq:Db}.}
\label{subsec:12}

Let $\DC_{\sige x}(\lambda+\mu)$ denote
the set of elements in $\SLS(\lambda) \otimes \SLS(\mu)$ 
satisfying condition \eqref{eq:Db} 
(or equivalently, condition \eqref{eq:Dc}); by Proposition~\ref{prop:DC}, 
we see that $\DC_{\sige x}(\lambda+\mu) \subset \SM(\lambda+\mu)$. 
%
%
\begin{lem}[{cf. \cite[Lemma~5.3.1 and Proposition~5.3.2]{NS16}}] \label{lem:Dem1}
\mbox{}
\begin{enu}

\item The set $\DC_{\sige x}(\lambda+\mu) \cup \{\bzero\}$ 
is stable under the action of the root operator $f_{i}$ for all $i \in I_{\af}$. 

\item The set $\DC_{\sige x}(\lambda+\mu) \cup \{\bzero\}$ is stable under 
the action of the root operator $e_{i}$ for those $i \in I_{\af}$ such that 
$\pair{x(\lambda+\mu)}{\alpha_{i}^{\vee}} \ge 0$. 

\item Let $i \in I_{\af}$ be such that 
$\pair{x(\lambda+\mu)}{\alpha_{i}^{\vee}} \ge 0$. Then, 
%
%
\begin{equation} \label{eq:Dem1-3a}
\DC_{\sige x}(\lambda+\mu) = \bigl\{e_{i}^{n}(\pi \otimes \eta) \mid 
\pi \otimes \eta \in \DC_{\sige s_{i}x}(\lambda+\mu),\,n \ge 0\bigr\} \setminus \{\bzero\}.
\end{equation}
\end{enu}
\end{lem}
\begin{proof}
(1) \ Let $\pi \otimes \eta \in \DC_{\sige x}(\lambda+\mu)$, and let $i \in I_{\af}$;
we may assume that $f_{i}(\pi \otimes \eta) \ne \bzero$. 
We give a proof only for the case that $f_{i}(\pi \otimes \eta) = \pi \otimes f_{i}\eta$; 
the proof for the case that $f_{i}(\pi \otimes \eta) = f_{i}\pi \otimes \eta$ is similar. 
We write $\pi$ and $\eta$ as: 
$\pi=(x_{1},\,\dots,\,x_{s}\,;\,\ba) \in \SLS(\lambda)$ and 
$\eta=(y_{1},\,\dots,\,y_{p}\,;\,\bb) \in \SLS(\mu)$, respectively.  
Let $x_{1}' \sige \cdots \sige x_{s}' \sige y_{1}' \sige \cdots \sige y_{p}'$ be 
a defining chain for $\pi \otimes \eta$ such that $\PS(y_{p}') \sige \PS(x)$. 
Take $0 \le t_{0} < t_{1} \le 1$ as in \eqref{eq:t-f} 
(with $\pi$ replaced by $\eta$); note that $H^{\eta}_{i}(t)$ 
is strictly increasing on the interval $[t_{0},\,t_{1}]$. 
We see from \eqref{eq:fpi} that $f_{i}\eta$ is of the form: 
\begin{equation*}
f_{i} \eta := ( y_{1},\,\ldots,\,y_{k},\,s_{i}y_{k+1},\,\dots,\,
  s_{i} y_{m},\,s_{i} y_{m+1},\,y_{m+1},\,\ldots,\,y_{p}\,;\,\bb') 
\end{equation*}
for some $0 \le k \le m \le p-1$ and some 
increasing sequence $\bb'$ of rational numbers in $[0,1]$. 
Here, since $H^{\eta}_{i}(t)$ is strictly increasing on 
the interval $[t_{0},\,t_{1}]$, 
it follows that $\pair{y_{n}\mu}{\alpha_{i}^{\vee}} > 0$, 
and hence that $y_{n}^{-1}\alpha_{i} \in (\Delta^{+} \setminus \DeK^{+}) + \BZ\delta$ 
for all $k+1 \le n \le m+1$. Hence we deduce that 
%
%
\begin{equation} \label{eq:yn}
(y_{n}')^{-1}\alpha_{i} 
\in (\Delta^{+} \setminus \DeK^{+}) + \BZ \delta \subset \Delta^{+} + \BZ\delta
\quad \text{for all $k+1 \le n \le m+1$}
\end{equation}
since $y_{n}' = y_{n}z_{n}$ for some $z_{n} \in (\WK)_{\af}$. 
Therefore, it follows from Lemma~\ref{lem:dmd}\,(3) (applied to the case $J=\emptyset$) 
that $s_{i}y_{k+1}' \sige \cdots \sige s_{i}y_{m+1}'$. 
Also, we see from \eqref{eq:simple} that 
$s_{i}y_{m+1}' \sige y_{m+1}'$. Thus we obtain
\begin{equation} \label{eq:Dem1-0}
s_{i}y_{k+1}' \sige \cdots \sige s_{i}y_{m+1}' \sige y_{m+1}' \sige \cdots \sige y_{p}'; 
\end{equation}
note that $\PK(s_{i}y_{n}') = s_{i}y_{n}$ 
for all $k+1 \le n \le m+1$ by Lemma~\ref{lem:si} since $s_{i}y_{n} \in (\WKu)_{\af}$. 
If $t_{1} \ne 1$, then $\kappa(f_{i}\eta) = \kappa(\eta) = y_{p}$, and 
the final element of the sequence \eqref{eq:Dem1-0} is $y_{p}'$, 
which satisfies $\PS(y_{p}') \sige \PS(x)$ by our assumption. 
If $t_{1} = 1$, then $m+1=p$, $\kappa(f_{i}\eta) = s_{i}\kappa(\eta) = s_{i}y_{p}$, and 
the final element of the sequence \eqref{eq:Dem1-0} is $s_{i}y_{m+1}' = s_{i}y_{p}'$. 
Since $s_{i}y_{m+1}' \sige y_{m+1}'$ as shown above, 
we deduce from Lemma~\ref{lem:611} that $\PS(s_{i}y_{p}') = \PS(s_{i}y_{m+1}') 
\sige \PS(y_{m+1}') = \PS(y_{p}') \sige \PS(x)$. 
In what follows, we will give a defining chain for 
$f_{i}(\pi \otimes \eta) = \pi \otimes (f_{i}\eta)$ 
in which the sequence \eqref{eq:Dem1-0} lies at the tail. 

\paragraph{Case 1.}
%
Assume that the set 
$\bigl\{1 \le n \le k \mid \pair{y_{n}\mu}{\alpha_{i}^{\vee}} \ne 0\bigr\}$ 
is nonempty, and let $k_{0}$ be the maximum element of this set. 
Because the function $H^{\eta}_{i}(t)$ attains 
its minimum value $m^{\eta}_{i}$ at $t=t_{0}$, 
it follows that $\pair{y_{k}\mu}{\alpha_{i}^{\vee}} = 
\pair{y_{k-1}\mu}{\alpha_{i}^{\vee}} = \cdots = 
\pair{y_{k_{0}+1}\mu}{\alpha_{i}^{\vee}} = 0$, and 
$\pair{y_{k_{0}}\mu}{\alpha_{i}^{\vee}} < 0$, 
which implies that 
$y_{n}^{-1}\alpha_{i} \in \DeK + \BZ\delta$ for all $k_{0}+1 \le n \le k$, and 
$y_{k_{0}}^{-1}\alpha_{i} \in -(\Delta^{+} \setminus \DeK^{+}) + \BZ\delta$. 
Hence we deduce that 
$(y_{n}')^{-1}\alpha_{i} \in \DeK + \BZ\delta$ for all $k_{0}+1 \le n \le k$, and 
$(y_{k_{0}}')^{-1}\alpha_{i} \in -(\Delta^{+} \setminus \DeK^{+}) + \BZ\delta$. 
Therefore, there exists $k_{0} \le k_{1} \le k$ such that 
$(y_{n}')^{-1}\alpha_{i} \in \DeK^{+} + \BZ\delta$ for all $k_{1}+1 \le n \le k$, and 
such that $(y_{k_{1}}')^{-1}\alpha_{i} \in -\Delta^{+} + \BZ\delta$; 
recall from \eqref{eq:yn} that 
$(y_{k+1}')^{-1}\alpha_{i} \in \Delta^{+} + \BZ\delta$.
Hence, in this case, we deduce from Lemma~\ref{lem:dmd}\,(1) and (3) that 
$y_{k_{1}}' \sige s_{i}y_{k_{1}+1}' \sige \cdots \sige s_{i}y_{k}' \sige s_{i}y_{k+1}'$;
since $(y_{n}')^{-1}\alpha_{i} \in \DeK + \BZ\delta$ for all $k_{1}+1 \le n \le k$, 
we see by Remark~\ref{rem:si} that 
$\PK(s_{i}y_{n}') = \PK(y_{n}') = y_{n}$ for all $k_{1}+1 \le n \le k$. 
Thus, we obtain a defining chain 
%
%
\begin{equation} \label{eq:Dem1-1}
x_{1}' \sige \cdots \sige x_{s}' \sige y_{1}' \sige \cdots \sige 
y_{k_{1}}' \sige s_{i}y_{k_{1}+1}' \sige \cdots \sige s_{i}y_{m+1}' 
\sige y_{m+1}' \sige \cdots \sige y_{p}'
\end{equation}
for $f_{i}(\pi \otimes \eta) = \pi \otimes (f_{i}\eta)$. 

\paragraph{Case 2.}
%
Assume that the set 
$\bigl\{1 \le n \le k \mid \pair{y_{n}\mu}{\alpha_{i}^{\vee}} \ne 0\bigr\}$ is empty, 
i.e., $\pair{y_{n}\mu}{\alpha_{i}^{\vee}} = 0$ for all $1 \le n \le k$; 
note that $(y_{n}')^{-1}\alpha_{i} \in \DeK + \BZ\delta$ for all $1 \le n \le k$. 
If there exists $1 \le k_{1} \le k$ such that 
$(y_{n}')^{-1}\alpha_{i} \in \DeK^{+} + \BZ\delta$ for all $k_{1}+1 \le n \le k$, and 
$(y_{k_{1}}')^{-1}\alpha_{i} \in -\DeK^{+} + \BZ\delta$, then 
we obtain a defining chain of the form \eqref{eq:Dem1-1} 
for $f_{i}(\pi \otimes \eta) = \pi \otimes (f_{i}\eta)$ 
in exactly the same way as in Case 1. Hence we may assume that 
$(y_{n}')^{-1}\alpha_{i} \in \DeK^{+} + \BZ\delta$ for all $1 \le n \le k$. 
It follows from Lemma~\ref{lem:dmd}\,(3) and \eqref{eq:yn} that 
\begin{equation} \label{eq:Dem1-2}
s_{i}y_{1}' \sige \cdots \sige 
s_{i}y_{k}' \sige s_{i}y_{k+1}' \sige \cdots \sige s_{i}y_{m+1}' 
\sige y_{m+1}' \sige \cdots \sige y_{p}';
\end{equation}
note that by Remark~\ref{rem:si}, 
$\PK(s_{i}y_{n}') = \PK(y_{n}') = y_{n}$ for all $1 \le n \le k$. 
Now, we define $u_{0}$ to be the maximum element of the set 
$\bigl\{1 \le u \le s \mid \pair{x_{u}\lambda}{\alpha_{i}^{\vee}} \ne 0\bigr\} \cup \{0\}$.
We claim that if $u_{0} \ge 1$, then 
$\pair{x_{s}\lambda}{\alpha_{i}^{\vee}} = 
\pair{x_{s-1}\lambda}{\alpha_{i}^{\vee}} = \cdots = 
\pair{x_{u_{0}+1}\lambda}{\alpha_{i}^{\vee}} = 0$, 
and that if $u_{0} \ge 1$, then 
$\pair{x_{u_{0}}\lambda}{\alpha_{i}^{\vee}} < 0$; 
this would imply that 
$x_{u}^{-1}\alpha_{i} \in \DeJ + \BZ\delta$ for all $u_{0}+1 \le u \le s$, and 
that if $u_{0} \ge 1$, then 
$x_{u_{0}}^{-1}\alpha_{i} \in -(\Delta^{+} \setminus \DeJ^{+}) + \BZ\delta$. 
Indeed, since $\pair{y_{n}\mu}{\alpha_{i}^{\vee}} = 0$ 
for all $1 \le n \le k$ by our assumption, 
we see that $H^{\eta}_{i}(t)$ is identically zero 
on the interval $[0,\,t_{0}]$, and hence 
$m^{\eta}_{i} = 0$, from which it follows that 
$\ve_{i}(\eta) = - m^{\eta}_{i} = 0$ by Remark~\ref{rem:SLS}. 
Here we recall that 
$f_{i}(\pi \otimes \eta) = \pi \otimes f_{i}\eta$ (if and) only if 
$\vp_{i}(\pi) \le \ve_{i}(\eta)$ by 
the tensor product rule for crystals. 
Hence we see that 
$\vp_{i}(\pi) = H^{\pi}_{i}(1) - m^{\pi}_{i} = 0$ 
by Remark~\ref{rem:SLS}. 
Since $\pair{x_{s}\lambda}{\alpha_{i}^{\vee}} = 
\pair{x_{s-1}\lambda}{\alpha_{i}^{\vee}} = \cdots = 
\pair{x_{u_{0}+1}\lambda}{\alpha_{i}^{\vee}} = 0$ by our assumption, 
we obtain $\pair{x_{u_{0}}\lambda}{\alpha_{i}^{\vee}} < 0$ 
if $u_{0} \ge 1$, as desired. 
Therefore, by the same argument as in Case 1, 
we get $0 \le u_{0} \le u_{1} \le s$ such that 
$(x_{u}')^{-1}\alpha_{i} \in \DeJ^{+} + \BZ\delta$ for all $u_{1}+1 \le u \le s$, and 
such that $(x_{u_{1}}')^{-1}\alpha_{i} \in -\Delta^{+} + \BZ\delta$ if $u_{1} \ge 1$; 
recall that $(y_{1}')^{-1}\alpha_{i} \in \Delta_{K}^{+} + \BZ\delta$.
Also we note that by Remark~\ref{rem:si}, 
$\PJ(s_{i}x_{u}') = \PJ(x_{u}') = x_{u}$ for all $u_{1}+1 \le i \le s$.
In this case, by Lemma~\ref{lem:dmd}\,(1) and (3), 
together with \eqref{eq:Dem1-2}, we obtain a defining chain  
\begin{equation*}
x_{1}' \sige \cdots \sige x_{u_{1}}' \sige s_{i}x_{u_{1}+1}' \sige \cdots \sige s_{i}x_{s}' \sige 
s_{i}y_{1}' \sige \cdots \sige s_{i}y_{m+1}' \sige y_{m+1}' \sige \cdots \sige y_{p}'
\end{equation*}
for $f_{i}(\pi \otimes \eta) = \pi \otimes (f_{i}\eta)$. 
This proves part (1). 

(2) \ Let $\pi \otimes \eta \in \DC_{\sige x}(\lambda+\mu)$, 
and let $i \in I_{\af}$ be such that $\pair{x(\lambda+\mu)}{\alpha_{i}^{\vee}} \ge 0$;
we may assume that $e_{i}(\pi \otimes \eta) \ne \bzero$. 
Since $\pi \otimes \eta \in \DC_{\sige x}(\lambda+\mu)$, 
there exists a defining chain for $\pi \otimes \eta$ whose final element, 
say $y \in W_{\af}$, satisfies $\PS(y) \sige \PS(x)$. We can show 
the following claims by arguments similar to those in part (1).
%
%
\begin{claim} \label{c:Dem1-1} \mbox{}
\begin{enu}

\item[\rm (i)] If $e_{i}(\pi \otimes \eta) = e_{i}\pi \otimes \eta$, or if 
$e_{i}(\pi \otimes \eta) = \pi \otimes (e_{i}\eta)$ and $\kappa(e_{i}\eta)=\kappa(\eta)$, 
then there exists a defining chain for $e_{i}(\pi \otimes \eta)$ 
whose final element is $y$. 

\item[\rm (ii)] If $e_{i}(\pi \otimes \eta) = \pi \otimes (e_{i}\eta)$ and 
$\kappa(e_{i}\eta)=s_{i}\kappa(\eta)$, then there exists a defining chain 
for $e_{i}(\pi \otimes \eta)$ whose final element is $s_{i}y$. 

\end{enu}
\end{claim}

In case (i) of Claim~\ref{c:Dem1-1}, 
it is obvious that $e_{i}(\pi \otimes \eta) \in \DC_{\sige x}(\lambda+\mu)$. 
In case (ii) of Claim~\ref{c:Dem1-1}, 
we see by the definition of the root operator $e_{i}$ 
that with notation in \eqref{eq:t-e} and \eqref{eq:epi}, 
$t_{1} = 1$, and the function $H^{\eta}_{i}(t)$ is strictly 
decreasing on $[t_{0},t_{1}] = [t_{0},1]$. 
Hence we have $\pair{\kappa(\eta)\mu}{\alpha_{i}^{\vee}} < 0$, 
which implies that $\kappa(\eta)^{-1}\alpha_{i} \in 
-(\Delta^{+} \setminus \DeK^{+}) + \BZ \delta$. 
Since $\PK(\PS(y)) = \PK(y) = \kappa(\eta)$, we see that 
$(\PS(y))^{-1}\alpha_{i} \in - (\Delta^{+} \setminus \DeK^{+}) + \BZ\delta 
\subset - (\Delta^{+} \setminus \DeS^{+}) + \BZ\delta$, 
which implies that $\pair{\PS(y)(\lambda+\mu)}{\alpha_{i}^{\vee}} < 0$. 
Also, it follows from Lemma~\ref{lem:si} that 
$s_{i}\PS(y) \in (\WSu)_{\af}$ and $\PS(s_{i}y) = s_{i}\PS(y)$. 
Here, by the assumption, we have $\pair{\PS(x)(\lambda+\mu)}{\alpha_{i}^{\vee}} = 
\pair{x(\lambda+\mu)}{\alpha_{i}^{\vee}} \ge 0$. 
Therefore, we deduce from 
Lemma~\ref{lem:dmd}\,(2), together with $\PS(y) \sige \PS(x)$, 
that $\PS(s_{i}y) = s_{i}\PS(y) \sige \PS(x)$. 
Thus, we conclude that 
$e_{i}(\pi \otimes \eta) \in \DC_{\sige x}(\lambda+\mu)$. 
This proves part (2). 

(3) \ If $\pair{x(\lambda+\mu)}{\alpha_{i}^{\vee}} = 0$, 
then $\PS(s_{i}x)=\PS(x)$ by Remark~\ref{rem:si}, and hence 
$\DC_{\sige s_{i}x}(\lambda+\mu) = \DC_{\sige x}(\lambda+\mu)$. 
Hence the assertion is obvious from part (2). 

Assume that $\pair{x(\lambda+\mu)}{\alpha_{i}^{\vee}} > 0$. 
Then we see from Lemma~\ref{lem:si} that 
$\PS(s_{i}x) = s_{i}\PS(x) \in (\WSu)_{\af}$ 
and $\PS(s_{i}x) \sige \PS(x)$, which implies that 
$\DC_{\sige x}(\lambda+\mu) \supset \DC_{\sige s_{i}x}(\lambda+\mu)$. 
Therefore, by part (2), 
we obtain the inclusion $\supset$ in \eqref{eq:Dem1-3a}. 
In order to show the opposite inclusion $\subset$ in \eqref{eq:Dem1-3a}, 
it suffices to show that $f_{i}^{\max}(\pi \otimes \eta) \in 
\DC_{\sige s_{i}x}(\lambda+\mu)$ for all $\pi \otimes \eta \in \DC_{\sige x}(\lambda+\mu)$. 
In view of part (1), this assertion itself follows from the following claim.
%
%
\begin{claim} \label{c:Dem1-2}
Let $\pi \otimes \eta \in \DC_{\sige x}(\lambda+\mu)$. 
If $f_{i}(\pi \otimes \eta) = \bzero$, i.e., 
$\vp_{i}(\pi \otimes \eta) = 0$, then 
$\pi \otimes \eta \in \DC_{\sige s_{i}x}(\lambda+\mu)$.
\end{claim}

\noindent
{\it Proof of Claim~\ref{c:Dem1-2}.}
We write $\pi$ and $\eta$ as: 
$\pi=(x_{1},\,\dots,\,x_{s}\,;\,\ba)$ and 
$\eta=(y_{1},\,\dots,\,y_{p}\,;\,\bb)$, respectively. 
Let $x_{1}' \sige \cdots \sige x_{s}' \sige y_{1}' \sige \cdots \sige y_{p}'$ be 
a defining chain for $\pi \otimes \eta$ such that $\PS(y_{p}') \sige \PS(x)$. 
We see from Lemma~\ref{lem:611} that 
$\PS(x_{1}') \sige \cdots \sige \PS(x_{s}') \sige 
 \PS(y_{1}') \sige \cdots \sige \PS(y_{p}')$ is also 
a defining chain for $\pi \otimes \eta$ 
satisfying $\PS(\PS(y_{p}')) = \PS(y_{p}') \sige \PS(x)$. 
Hence we may assume from the beginning that 
$x_{1}',\,\dots,\,x_{s}',\,y_{1}',\,\dots,\,y_{p}' \in (\WSu)_{\af}$. 

Assume first that the set 
\begin{equation} \label{eq:set1}
\begin{split}
& \bigl\{1 \le q \le p \mid 
(y_{q}')^{-1}\alpha_{i} \not\in (\DeK^{+} \setminus \DeS^{+}) + \BZ\delta\bigr\} \\
& \hspace*{20mm}
= \bigl\{1 \le q \le p \mid 
(y_{q}')^{-1}\alpha_{i} \in 
\bigl((\Delta \setminus \DeK^{+}) \sqcup \DeS^{+}\bigr) + \BZ\delta\bigr\}
\end{split}
\end{equation}
is nonempty. 
Let $q_{1}$ be the maximum element of this set;
notice that $\pair{y_{q}'(\lambda+\mu)}{\alpha_{i}^{\vee}} > 0$ and 
$\pair{y_{q}\mu}{\alpha_{i}^{\vee}} = 
\pair{y_{q}'\mu}{\alpha_{i}^{\vee}} = 0$ for all $q_{1} < q \le p$. 
Note that $f_{i}(\pi \otimes \eta)=\bzero$ implies $f_{i}\eta = \bzero$ 
by the tensor product rule for crystals, and hence 
$H^{\eta}_{i}(1) - m^{\eta}_{i}=0$. 
From this it follows that 
$\pair{y_{q_{1}}'\mu}{\alpha_{i}^{\vee}} = 
\pair{y_{q_{1}}\mu}{\alpha_{i}^{\vee}} \le 0$, and hence 
$(y_{q_{1}}')^{-1}\alpha_{i} \in ((-\Delta^{+}) \sqcup \DeS^{+}) + \BZ\delta$ 
by the definition of $q_{1}$, 
which implies that $\pair{y_{q_{1}}'(\lambda+\mu)}{\alpha_{i}^{\vee}} \le 0$. 
Therefore, we see from Lemma~\ref{lem:dmd}\,(1) and (3) that
\begin{equation} \label{eq:Dem1-2-1}
x_{1}' \sige \cdots \sige x_{s}' \sige y_{1}' \sige \cdots \sige y_{q_{1}}' \sige 
s_{i}y_{q_{1}+1}' \sige \cdots \sige s_{i}y_{p}'; 
\end{equation}
note that $\PK(s_{i}y_{q}')= \PK(y_{q}')$ for all $q_{1} < q \le p$ 
since $\pair{y_{q}'\mu}{\alpha_{i}^{\vee}} =  0$. 
Thus the sequence \eqref{eq:Dem1-2-1} is 
also a defining chain for $\pi \otimes \eta$. 
If $q_{1} = p$, then the final element of \eqref{eq:Dem1-2-1} is $y_{p}'$,
and $\pair{y_{p}'(\lambda+\mu)}{\alpha_{i}^{\vee}} \le 0$. 
Hence it follows from Lemma~\ref{lem:dmd}\,(1) that 
$y_{p}' \sige s_{i}\PS(x) = \PS(s_{i}x)$. 
If $q_{1} < p$, then the final element of \eqref{eq:Dem1-2-1} is $s_{i}y_{p}'$, 
and $\pair{y_{p}'(\lambda+\mu)}{\alpha_{i}^{\vee}} > 0$. 
This implies that $s_{i}y_{p}' \in (\WSu)_{\af}$ by Lemma~\ref{lem:si}, 
and that $s_{i}y_{p}' \sige s_{i}\PS(x) = \PS(s_{i}x)$ by Lemma~\ref{lem:dmd}\,(3). 
Hence we conclude that $\pi \otimes \eta \in \DC_{\sige s_{i}x}(\lambda+\mu)$. 

Assume next that the set in \eqref{eq:set1} is empty, 
that is, $(y_{q}')^{-1}\alpha_{i} \in (\DeK^{+} \setminus \DeS^{+}) + \BZ\delta$
for all $1 \le q \le p$; notice that 
$\pair{y_{q}'(\lambda+\mu)}{\alpha_{i}^{\vee}} > 0$ for all $1 \le  q \le p$. 
Also, since $\pair{y_{q}\mu}{\alpha_{i}^{\vee}} = 
\pair{y_{q}'\mu}{\alpha_{i}^{\vee}} = 0$ for all $1 \le  q \le p$, 
we have $H^{\eta}_{i}(t) = 0$ for all $t \in [0,1]$, and hence $\ve_{i}(\eta)=0$. 
Since $f_{i}(\pi \otimes \eta)=\bzero$ by the assumption, we obtain 
$f_{i}\pi = \bzero$ by the tensor product rule for crystals.  
Let $u_{1}$ be the maximum element of the set $\bigl\{ 1 \le u \le s \mid 
(x_{u}')^{-1}\alpha_{i} \not\in (\DeJ^{+} \setminus \DeS^{+}) + \BZ\delta\bigr\} \cup \{0\}$. 
Then we have $\pair{x_{u}'(\lambda+\mu)}{\alpha_{i}^{\vee}} > 0$ and 
$\pair{x_{u}'\lambda}{\alpha_{i}^{\vee}} = 
\pair{x_{u}\lambda}{\alpha_{i}^{\vee}} = 0$ for all $u_{1} < u \le s$. 
In addition, we can show by the same argument as above that 
if $u_{1} \ge 1$, then $\pair{x_{u_{1}}'(\lambda+\mu)}{\alpha_{i}^{\vee}} \le 0$. 
Therefore, it follows from Lemma~\ref{lem:dmd}\,(1) and (3) that
\begin{equation} \label{eq:Dem1-2-2}
x_{1}' \sige \cdots \sige x_{u_{1}}' \sige 
s_{i}x_{u_{1}+1} \sige \cdots \sige s_{i}x_{s}' \sige s_{i}y_{1}' \sige 
\cdots \sige s_{i}y_{p}'.
\end{equation}
In the same way as for \eqref{eq:Dem1-2-1}, 
we can verify that the sequence \eqref{eq:Dem1-2-2}  is 
a defining chain for $\pi \otimes \eta$ 
satisfying the condition in \eqref{eq:Db}. 
This proves Claim~\ref{c:Dem1-2}. \bqed

This completes the proof of Lemma~\ref{lem:Dem1}. 
\end{proof}
%
%
\begin{cor}[{cf. \cite[Corollary~5.3.3]{NS16}}] \label{cor:Dem2}
Let $x \in W_{\af}$, and $i \in I_{\af}$. 
For every $\pi \otimes \eta \in \DC_{\sige x}(\lambda+\mu)$, 
we have $f_{i}^{\max}(\pi \otimes \eta) \in \DC_{\sige s_{i}x}(\lambda+\mu)$. 
\end{cor}
\begin{proof}
If $\pair{x(\lambda+\mu)}{\alpha_{i}^{\vee}} \ge 0$, 
then the assertion follows from the proof of Lemma~\ref{lem:Dem1}\,(3). 
If $\pair{x(\lambda+\mu)}{\alpha_{i}^{\vee}} < 0$, then 
we have $\PS(s_{i}x) = s_{i}\PS(x) \sile \PS(x)$ by Lemma~\ref{lem:si}, 
and hence $\DC_{\sige s_{i}x}(\lambda+\mu) \supset \DC_{x}(\lambda+\mu)$. 
Therefore, the assertion follows from Lemma~\ref{lem:Dem1}\,(1). 
This proves the corollary. 
\end{proof}
%
%
\begin{lem} \label{lem:Dem3}
Let $x,\,y \in W_{\af}$.
\begin{enu}

\item If $\PS(y) \sige \PS(x)$ in $(\WSu)_{\af}$, then 
$y \cdot \bigl( t_{\xi} \cdot \pi_{\brho} \otimes \pi_{\bchi}) \in \DC_{\sige x}(\lambda+\mu)$
for all $(\brho,\,\bchi,\,\xi) \in \Par(\lambda) \times \Par(\mu) \times 
Q^{\vee}_{I \setminus (\J \cup \K)}$ satisfying \eqref{eq:tx2}. 

\item If $y \cdot \bigl( t_{\xi} \cdot \pi_{\brho} \otimes \pi_{\bchi}) 
\in \DC_{\sige x}(\lambda+\mu)$ for some $(\brho,\,\bchi,\,\xi) \in 
\Par(\lambda) \times \Par(\mu) \times Q^{\vee}_{I \setminus (\J \cup \K)}$ 
satisfying \eqref{eq:tx2}, then $\PS(y) \sige \PS(x)$. 

\end{enu}
\end{lem}

\begin{proof}
(1) By the definitions (see \eqref{eq:ext}), 
$\pi_{\brho} \in \SLS(\lambda)$ and $\pi_{\bchi} \in \SLS(\mu)$ are of the form: 
%
%
\begin{equation} \label{eq:Dem3b}
\begin{cases}
\pi_{\brho}=(\PJ(t_{\xi_1}),\,\dots,\,\PJ(t_{\xi_{s-1}}),\,e\,;\,\ba), & \\[1mm]
\pi_{\bchi}=(\PK(t_{\zeta_1}),\,\dots,\,\PK(t_{\zeta_{p-1}}),\,e\,;\,\bb)
\end{cases}
\end{equation}
for some $\xi_{1},\,\dots,\,\xi_{s-1} \in Q^{\vee}_{I \setminus \J}$ such that 
$\xi_{1} > \cdots > \xi_{s-1} > 0$, and 
$\zeta_{1},\,\dots,\,\zeta_{p-1} \in Q^{\vee}_{I \setminus \K}$ such that 
$\zeta_{1} > \cdots > \zeta_{p-1} > 0$, respectively. 
Also, recall from \eqref{eq:tx0} that 
\begin{equation}
t_{\xi} \cdot \pi_{\brho}=
(\PJ(t_{\xi_1+\xi}),\,\dots,\,
 \PJ(t_{\xi_{s-1}+\xi}),\,\PJ(t_{\xi})\,;\,\ba), 
\end{equation}
and from Lemma~\ref{lem:Weyl} that 
%
%
\begin{equation} \label{eq:Dem3c}
\begin{split}
& y \cdot \bigl( t_{\xi} \cdot \pi_{\brho} \otimes \pi_{\bchi}) = 
(y \cdot (t_{\xi} \cdot \pi_{\brho})) \otimes (y \cdot \pi_{\bchi}), \quad \text{with} \\
& 
\begin{cases}
y \cdot (t_{\xi} \cdot \pi_{\brho}) = 
 (\PJ(yt_{\xi_1+\xi}),\,\dots,\,\PJ(yt_{\xi_{s-1}+\xi}),\,\PJ(yt_{\xi})\,;\,\ba), & \\[1mm]
y \cdot \pi_{\bchi}=(\PK(yt_{\zeta_1}),\,\dots,\,\PK(yt_{\zeta_{p-1}}),\,\PK(y)\,;\,\bb). 
\end{cases}
\end{split}
\end{equation}

Now, if $\bchi=(\chi^{(i)})_{i \in I} \in \Par(\mu)$, with 
$\chi^{(i)}=(\chi^{(i)}_{1} \ge \chi^{(i)}_{2} \ge \cdots)$ for $i \in I$, 
then we have $\zeta_{1} = \sum_{i \in I} \chi_{1}^{(i)} \alpha_{i}^{\vee}$ 
by \eqref{eq:par0}; 
we set $\gamma:=\sum_{i \in \J \setminus K}
\chi_{1}^{(i)}\alpha_{i}^{\vee} \in \QJv$. 
Since $(\brho,\,\bchi,\,\xi)$ satisfies \eqref{eq:tx2}, and 
$\chi_{1}^{(i)}=0$ for all $i \in K$, we deduce that 
$\xi + \gamma \ge \zeta_{1}$, and hence that 
\begin{equation*}
\xi_{1}+\xi+\gamma > \cdots > \xi_{s-1}+\xi+\gamma > \xi+\gamma \ge 
\zeta_{1} > \cdots > \zeta_{p-1} > 0. 
\end{equation*}
Therefore, it follows from Lemma~\ref{lem:SiB}\,(2) 
(applied to the case $J=\emptyset$) that
%
%
\begin{equation} \label{eq:dc63}
yt_{\xi_{1}+\xi+\gamma} \sig \cdots \sig yt_{\xi_{s-1}+\xi+\gamma} \sig yt_{\xi+\gamma} \sige
yt_{\zeta_{1}} \sig \cdots \sig yt_{\zeta_{p-1}} \sig y = yt_{0} \quad \text{in $W_{\af}$}; 
\end{equation} 
note that by Lemma~\ref{lem:PiJ}, 
$\PJ(yt_{\xi_{u}+\xi+\gamma}) = \PJ(yt_{\xi_{u}+\xi})$ 
for all $1 \le u \le s$ since $\gamma \in \QJv$. 
Hence the sequence \eqref{eq:dc63} is 
a defining chain for $y \cdot \bigl( t_{\xi} \cdot \pi_{\brho} \otimes \pi_{\bchi})$. 
Since $\PS(y) \sige \PS(x)$ in $(\WSu)_{\af}$ by the assumption, 
we conclude that 
$y \cdot \bigl( t_{\xi} \cdot \pi_{\brho} \otimes \pi_{\bchi}) 
\in \DC_{\sige x}(\lambda+\mu)$. This proves part (1). 

(2) We divide the proof into several steps. 
%
\paragraph{Step 1.}
%
Assume that $x = t_{\zeta'}$ 
for some $\zeta' \in Q^{\vee}$, and $y = t_{\xi'}$ 
for some $\xi' \in Q^{\vee}$; in this case, 
in order to prove that $\PS(y) \sige \PS(x)$ in $(\WSu)_{\af}$, 
it suffices to show that $[\xi']^{S} \ge [\zeta']^{S}$ (see Lemma~\ref{lem:SiB}\,(2)). 
In the same way as for \eqref{eq:Dem3c}, we obtain
%
%
\begin{equation} \label{eq:Dem3e}
\begin{split}
& y \cdot \bigl( t_{\xi} \cdot \pi_{\brho} \otimes \pi_{\bchi}) = 
(y \cdot (t_{\xi} \cdot \pi_{\brho})) \otimes (y \cdot \pi_{\bchi}), \quad \text{with} \\
& 
\begin{cases}
y \cdot (t_{\xi} \cdot \pi_{\brho}) = 
 (\PJ(t_{\xi_1+\xi+\xi'}),\,\dots,\,\PJ(t_{\xi_{s-1}+\xi+\xi'}),\,\PJ(t_{\xi+\xi'})\,;\,\ba), & \\[1mm]
y \cdot \pi_{\bchi}=(\PK(t_{\zeta_1+\xi'}),\,\dots,\,\PK(t_{\zeta_{p-1}+\xi'}),\,\PK(t_{\xi'})\,;\,\bb). 
\end{cases}
\end{split}
\end{equation}
Here, since $y \cdot \bigl( t_{\xi} \cdot \pi_{\brho} \otimes \pi_{\bchi}) 
\in \DC_{\sige x}(\lambda+\mu)$ satisfies condition \eqref{eq:Db} 
(or equivalently, condition \eqref{eq:Dc}; see Section~\ref{subsec:12}), 
we have 
\begin{equation} \label{eq:Dem3f}
\kappa(y \cdot \pi_{\bchi}) \sige \PK(x) \quad \text{and} \quad
\kappa \bigl(y \cdot (t_{\xi} \cdot \pi_{\brho}))\bigr) \sige \PJ(\Deo{y \cdot \pi_{\bchi}}{x}). 
\end{equation}
We deduce from the first inequality in \eqref{eq:Dem3f} that 
$\PK(t_{\xi'}) = \kappa(y \cdot \pi_{\bchi}) \sige \PK(x) = \PK(t_{\zeta'})$, 
which implies that $[\xi']^{\K} \ge [\zeta']^{\K}$ by Lemma~\ref{lem:SiB}\,(2). 
Since $I \setminus S = (I \setminus \K) \sqcup (\K \setminus S)$, 
it remains to show that $[\xi']_{\K \setminus S} \ge [\zeta']_{\K \setminus S}$. 
We define $\ti{y}_{p}$, $\ti{y}_{p-1}$, $\dots$, $\ti{y}_{1}$ 
by the recursive procedure \eqref{eq:lift0}, that is, 
\begin{equation*}
\begin{split}
& \ti{y}_{p}:=\min \Lif{x}{\PK(t_{\xi'})}, \quad 
  \text{with} \quad x = t_{\zeta'}, \\
& \ti{y}_{p-1}:=\min \Lif{\ti{y}_{p}}{\PK(t_{\zeta_{p-1}+\xi'})},  \\
& \qquad \vdots \\
& \ti{y}_{1}:=\min \Lif{\ti{y}_{2}}{\PK(t_{\zeta_{1}+\xi'})}
  = \Deo{y \cdot \pi_{\bchi}}{x}.
\end{split}
\end{equation*}

\begin{claim} \label{c:Dem3-1}
The elements $\ti{y}_{q}$, $1 \le q \le p$, are of the form{\rm:} 
$\ti{y}_{q} = t_{\zeta_{q}+\xi'+\gamma_{q}}$ 
for some $\gamma_{q} \in \QKv$, where we set $\zeta_{p}:=0$. 
\end{claim}

\noindent
{\it Proof of Claim~\ref{c:Dem3-1}.}
We show the assertion by descending induction on $1 \le q \le p$. 
Assume that $q = p$. We see from Lemma~\ref{lem:lift} that
$\ti{y}_{p} = z_{p} t_{\xi'+\gamma_{p}}$ 
for some $z_{p} \in \WK$ and $\gamma_{p} \in \QKv$. 
Since $z_{p} t_{\xi'+\gamma_{p}} = \ti{y}_{p} \sige x = t_{\zeta'}$, 
it follows from Lemma~\ref{lem:SiB}\,(1) and(2) that 
$t_{\xi'+\gamma_{p}} \sige t_{\zeta'} = x$ in $W_{\af}$. 
Also, we have 
$t_{\xi'+\gamma_{p}} \in \Lift(\PK(t_{\xi'}))$ by Lemma~\ref{lem:lift}. 
Combining these, we obtain $t_{\xi'+\gamma_{p}} \in 
\Lif{x}{\PK(t_{\xi'})}$. 
Since $\ti{y}_{p} = z_{p} t_{\xi'+\gamma_{p}} \sige t_{\xi'+\gamma_{p}}$ 
in $W_{\af}$ by Remark~\ref{rem:SiB},
we deduce that $\ti{y}_{p} = t_{\xi'+\gamma_{p}}$ 
by the minimality of $\ti{y}_{p}$. 

Assume that $q < p$; by our induction hypothesis, 
we have $\ti{y}_{q+1} = t_{\zeta_{q+1}+\xi'+\gamma_{q+1}}$ 
for some $\gamma_{q+1} \in \QKv$. 
Also, we see from Lemma~\ref{lem:lift} that
$\ti{y}_{q} = z_{q} t_{\zeta_{q}+\xi'+\gamma_{q}}$ 
for some $z_{q} \in \WK$ and $\gamma_{q} \in \QKv$. 
Now, the same argument as above shows that
\begin{equation*}
\ti{y}_{q} = z_{q} t_{\zeta_{q}+\xi'+\gamma_{q}} \sige 
t_{\zeta_{q}+\xi'+\gamma_{q}} \sige 
t_{\zeta_{q+1}+\xi'+\gamma_{q+1}} = \ti{y}_{q+1}, 
\end{equation*}
and that $t_{\zeta_{q}+\xi'+\gamma_{q}} \in \Lift(\PK(t_{\zeta_{q}+\xi'}))$. 
Hence we obtain $\ti{y}_{q} = t_{\zeta_{q}+\xi'+\gamma_{q}}$ 
by the minimality of $\ti{y}_{q}$. 
This proves Claim~\ref{c:Dem3-1}. \bqed

\vsp

Because $\ti{y}_{1} \sige \cdots \sige \ti{y}_{p} \sige x$ in $W_{\af}$, 
it follows from Lemma~\ref{lem:SiB}\,(2) and Claim~\ref{c:Dem3-1} that
\begin{equation*}
\zeta_{1}+\xi'+\gamma_{1} \ge \zeta_{2}+\xi'+\gamma_{2} \ge \cdots \ge 
\zeta_{p}+\xi'+\gamma_{p} \ge \zeta'; 
\end{equation*}
in particular, we have
%
%
\begin{equation} \label{eq:step1-4}
[\zeta_{1}+\xi'+\gamma_{1}]_{\K \setminus S} \ge 
[\zeta']_{\K \setminus S}. 
\end{equation}
Also, we see by the second inequality in \eqref{eq:Dem3f} 
and Claim~\ref{c:Dem3-1} that 
\begin{equation*}
\PJ(t_{\xi+\xi'}) = 
\kappa(y \cdot (t_{\xi} \cdot \pi_{\brho})) \sige 
\PJ(\Deo{y \cdot \pi_{\bchi}}{x}) = 
\PJ(t_{\zeta_{1}+\xi'+\gamma_{1}}),
\end{equation*}
which implies that $[\xi+\xi']^{\J} \ge 
[\zeta_{1}+\xi'+\gamma_{1}]^{\J}$ by Lemma~\ref{lem:SiB}\,(2); 
in particular, we have
$[\xi+\xi']_{\K \setminus S} \ge 
[\zeta_{1}+\xi'+\gamma_{1}]_{\K \setminus S}$. 
Here, since $\xi \in Q^{\vee}_{I \setminus (\J \cup \K)}$, 
we have $[\xi+\xi']_{\K \setminus S} = 
[\xi']_{\K \setminus S}$. Therefore, we deduce that
%
%
\begin{equation} \label{eq:step1-5}
[\xi']_{\K \setminus S} \ge 
[\zeta_{1}+\xi'+\gamma_{1}]_{\K \setminus S}. 
\end{equation}
Combining \eqref{eq:step1-4} and \eqref{eq:step1-5}, 
we obtain $[\xi']_{\K \setminus S} \ge 
[\zeta']_{\K \setminus S}$, as desired. 

\paragraph{Step 2.}
%
Assume that $x = t_{\zeta'}$ for some $\zeta' \in Q^{\vee}$, and 
write $y \in W_{\af}$ as $y = v t_{\xi'}$ for some $v \in W$ and $\xi' \in Q^{\vee}$.
Let us show the assertion by induction on $\ell(v)$. 
If $\ell(v) = 0$, i.e., $v = e$, then the assertion follows from Step 1. 
Assume that $\ell(v) > 0$. We take $i \in I$ such that 
$\ell(s_{i}v) = \ell(v) -1$; note that $y^{-1}\alpha_{i} \in -\Delta^{+} + \BZ \delta$. 
Since $\pair{y\lambda}{\alpha_{i}^{\vee}} \le 0$ and 
$\pair{y\mu}{\alpha_{i}^{\vee}} \le 0$, 
we see by the definition of the root operator $f_{i}$ and \eqref{eq:Dem3c} that 
$f_{i}\bigl(y \cdot \bigl( t_{\xi} \cdot \pi_{\brho} \otimes \pi_{\bchi})\bigr) = \bzero$, 
and hence that
\begin{equation} \label{eq:step2-1}
e_{i}^{\max} \bigl( y \cdot \bigl( t_{\xi} \cdot \pi_{\brho} \otimes \pi_{\bchi}) \bigr)
= (s_{i}y) \cdot \bigl( t_{\xi} \cdot \pi_{\brho} \otimes \pi_{\bchi}).
\end{equation}
Since $x = t_{\zeta'}$, we have 
$\pair{x(\lambda+\mu)}{\alpha_{i}^{\vee}} = 
 \pair{\lambda+\mu}{\alpha_{i}^{\vee}} \ge 0$. 
Therefore, by Lemma~\ref{lem:Dem1}\,(2), 
together with \eqref{eq:step2-1}, we obtain 
$(s_{i}y) \cdot \bigl( t_{\xi} \cdot \pi_{\brho} \otimes \pi_{\bchi}) 
\in \DC_{\sige x}(\lambda+\mu)$. 
Hence, by our induction hypothesis, 
we have $\PS(s_{i}y) \sige \PS(x)$.
Here we recall that $\pair{y(\lambda+\mu)}{\alpha_{i}^{\vee}} \le 0$ since 
$y^{-1}\alpha_{i} \in -\Delta^{+} + \BZ \delta$. 
If $\pair{y(\lambda+\mu)}{\alpha_{i}^{\vee}} < 0$, then 
$\PS(y) \sige s_{i}\PS(y) = \PS(s_{i}y) \sige \PS(x)$ 
by Lemma~\ref{lem:si} and Remark~\ref{rem:si}. 
If $\pair{y(\lambda+\mu)}{\alpha_{i}^{\vee}} = 0$, then 
$\PS(y) = \PS(s_{i}y) \sige \PS(x)$ by Remark~\ref{rem:si}. 
In both cases, we obtain $\PS(y) \sige \PS(x)$, as desired. 

\paragraph{Step 3.}
%
Let $x,\,y \in W_{\af}$. We see from \cite{AK} that there exist 
$i_{1},\,\dots,\,i_{N} \in I_{\af}$ such that
$\pair{s_{i_{n-1}} \cdots s_{i_{1}}x(\lambda+\mu)}{\alpha_{i_n}^{\vee}} \ge 0$ 
for all $1 \le n \le N$, and such that
$s_{i_{N}} \cdots s_{i_{1}}x = t_{\zeta'}$ for some $\zeta' \in Q^{\vee}$. 
Let us show the assertion by induction on $N$. 
If $N=0$, i.e., $x = t_{\zeta'}$, 
then the assertion follows from Step 2. Assume that $N > 0$; 
for simplicity of notation, we set $i:=i_{1}$. 
It follows from Corollary~\ref{cor:Dem2} that 
\begin{equation} \label{eq:step3-1}
f_{i}^{\max} 
\bigl( y \cdot \bigl( t_{\xi} \cdot \pi_{\brho} \otimes \pi_{\bchi}) \bigr)
\in \DC_{\sige s_{i}x}(\lambda+\mu).
\end{equation}
%
\paragraph{Case 3.1.}
%
Assume that $\pair{y(\lambda+\mu)}{\alpha_{i}^{\vee}} \le 0$; 
note that $\pair{y\lambda}{\alpha_{i}^{\vee}} \le 0$ and 
$\pair{y\mu}{\alpha_{i}^{\vee}} \le 0$. 
We see by the definition of the root operator $f_{i}$ and \eqref{eq:Dem3c} that 
$f_{i}^{\max} 
\bigl( y \cdot \bigl( t_{\xi} \cdot \pi_{\brho} \otimes \pi_{\bchi}) \bigr)= 
y \cdot \bigl( t_{\xi} \cdot \pi_{\brho} \otimes \pi_{\bchi})$.
Hence, by our induction hypothesis, 
we have $\PS(y) \sige \PS(s_{i}x)$. Here we recall that 
$\pair{x(\lambda+\mu)}{\alpha_{i}^{\vee}} \ge 0$. 
If $\pair{x(\lambda+\mu)}{\alpha_{i}^{\vee}} = 0$, then 
we have $\PS(s_{i}x) = \PS(x)$ by Remark~\ref{rem:si}. 
If $\pair{x(\lambda+\mu)}{\alpha_{i}^{\vee}} > 0$, then 
it follows from Lemma~\ref{lem:si} that 
$\PS(s_{i}x) = s_{i}\PS(x) \sige \PS(x)$. 
In both cases, we obtain $\PS(y) \sige \PS(x)$, as desired. 
%
\paragraph{Case 3.2.}
%
Assume that $\pair{y(\lambda+\mu)}{\alpha_{i}^{\vee}} > 0$; 
note that $\pair{y\lambda}{\alpha_{i}^{\vee}} \ge 0$ and 
$\pair{y\mu}{\alpha_{i}^{\vee}} \ge 0$. 
We see by the definition of the root operator $e_{i}$ and 
\eqref{eq:Dem3c} that 
$e_{i} \bigl( y \cdot \bigl( t_{\xi} \cdot \pi_{\brho} \otimes \pi_{\bchi}) \bigr) = \bzero$, 
and hence
\begin{equation*}
f_{i}^{\max} 
\bigl( y \cdot \bigl( t_{\xi} \cdot \pi_{\brho} \otimes \pi_{\bchi}) \bigr)= 
(s_{i}y) \cdot \bigl( t_{\xi} \cdot \pi_{\brho} \otimes \pi_{\bchi}).
\end{equation*}
Hence, by our induction hypothesis, 
we have $\PS(s_{i}y) \sige \PS(s_{i}x)$. 
As in Case 3.1, we see that $\PS(s_{i}x) \sige \PS(x)$, 
and hence $\PS(s_{i}y) \sige \PS(x)$. 
Also, since $\pair{y(\lambda+\mu)}{\alpha_{i}^{\vee}} > 0$, 
it follows from Lemma~\ref{lem:si} that $\PS(s_{i}y) = s_{i}\PS(y)$, 
and hence from Lemma~\ref{lem:dmd}\,(2) that $\PS(y) \sige \PS(x)$. 

This proves part (2), and completes the proof of Lemma~\ref{lem:Dem3}. 
\end{proof}

Now, the equivalence \eqref{eq:Da} $\Leftrightarrow$ \eqref{eq:Db} 
follows from the next lemma. 
%
%
\begin{lem} \label{lem:DemX}
Let $\psi \in \SLS(\lambda+\mu)$, and assume that 
$\psi$ is mapped to $\pi \otimes \eta \in \SM(\lambda+\mu)$ 
under the isomorphism $\SLS(\lambda+\mu) \cong \SM(\lambda+\mu)$ 
in Theorem~\ref{thm:SMT}. Let $x \in W_{\af}$. 
\begin{enu}

\item If $\psi \in \SLS_{\sige x}(\lambda+\mu)$, then 
$\pi \otimes \eta \in \DC_{\sige x}(\lambda+\mu)$. 

\item If $\pi \otimes \eta \in \DC_{\sige x}(\lambda+\mu)$, then 
$\psi \in \SLS_{\sige x}(\lambda+\mu)$. 

\end{enu}
\end{lem}

\begin{proof}
By \cite[Lemma~5.4.1]{NS16}, 
there exist $i_{1},\,i_{2},\,\dots,\,i_{N} \in I_{\af}$ satisfying 
the conditions that 
\begin{equation*}
\begin{cases}
\pair{s_{i_{n-1}}s_{i_{n-2}} \cdots 
  s_{i_{2}}s_{i_{1}}x(\lambda+\mu)}{\alpha_{i_n}^{\vee}} \ge 0 
  & \text{for all $1 \le n \le N$, and} \\[1.5mm]
f_{i_{N}}^{\max}f_{i_{N-1}}^{\max} \cdots f_{i_{2}}^{\max}f_{i_{1}}^{\max}\psi = 
t_{\xi'} \cdot \pi_{\bvrho}
  & \text{for some $\xi' \in Q^{\vee}$ and $\bvrho \in \Par(\lambda+\mu)$}. 
\end{cases}
\end{equation*}
We prove part (1) by induction on $N$. Assume that $N=0$, i.e., 
$\psi = t_{\xi'} \cdot \pi_{\bvrho}$; recall from \eqref{eq:tx0} that 
$\kappa(\psi) = \PS(t_{\xi'})$. 
Since $\psi \in \SLS_{\sige x}(\lambda+\mu)$ 
by the assumption, we have
%
%
\begin{equation} \label{eq:DemX-1}
\PS(t_{\xi'}) = \kappa(\psi) \sige \PS(x). 
\end{equation}
By Corollary~\ref{cor:Sext}, 
$\psi = t_{\xi'} \cdot \pi_{\bvrho}$ is mapped to $t_{\xi'} \cdot 
\bigl(t_{\xi} \cdot \pi_{\brho} \otimes \pi_{\bchi}\bigr)$, 
which is $\pi \otimes \eta$, 
for some $\xi \in Q^{\vee}_{I \setminus (\J \cup \K)}$ and 
$\brho \in \Par(\lambda)$, $\bchi \in \Par(\mu)$ satisfying \eqref{eq:tx2} 
under the isomorphism $\SLS(\lambda+\mu) \cong \SM(\lambda+\mu)$ of crystals 
in Theorem~\ref{thm:SMT}.
Therefore, we deduce from Lemma~\ref{lem:Dem3}\,(1), 
together with \eqref{eq:DemX-1}, 
that $\pi \otimes \eta \in \DC_{\sige x}(\lambda+\mu)$. 

Assume that $N > 0$. For simplicity of notation, we set $i_{1}:=i$; 
note that $\pair{x(\lambda+\mu)}{\alpha_{i}^{\vee}} \ge 0$. 
We see from \cite[Corollary~5.3.3]{NS16} that 
$f_{i}^{\max}\psi \in \SLS_{\sige s_{i}x}(\lambda+\mu)$. 
By our induction hypothesis, we have $f_{i}^{\max}(\pi \otimes \eta) 
\in \DC_{\sige s_{i}x}(\lambda+\mu)$. 
Since $\pi \otimes \eta = e_{i}^{k} f_{i}^{\max}(\pi \otimes \eta)$ 
for some $k \ge 0$, we deduce from Lemma~\ref{lem:Dem1}\,(3) that 
$\pi \otimes \eta \in \DC_{\sige x}(\lambda+\mu)$. This proves part~(1). 

We can prove part~(2) similarly, using 
Lemma~\ref{lem:Dem3}\,(2)
instead of Lemma~\ref{lem:Dem3}\,(1).
\end{proof}

\appendix
%
\section*{Appendices.}
%
%
\section{Basic properties of the semi-infinite Bruhat order.}
\label{sec:basic}

We fix $J \subset I$ and 
$\lambda \in P^{+} \subset P_{\af}^{0}$ (see \eqref{eq:P-fin} and \eqref{eq:P}) 
such that $\big\{ i \in I \mid \pair{\lambda}{\alpha_{i}^{\vee}}= 0 \bigr\} = \J$. 
%
%
\begin{lem}[{\cite[Lemmas~2.3.3 and 2.3.5]{INS}}] \label{lem:PiJ}
\mbox{}
\begin{enu}
\item It holds that 
%
%
\begin{equation} \label{eq:PiJ2}
\begin{cases}
\PJ(w)=\mcr{w} 
  & \text{\rm for all $w \in W$}; \\[1mm]
\PJ(xt_{\xi})=\PJ(x)\PJ(t_{\xi}) 
  & \text{\rm for all $x \in W_{\af}$ and $\xi \in Q^{\vee}$};
\end{cases}
\end{equation}
in particular, 
%
%
\begin{equation} \label{eq:PiJ3}
(\WJu)_{\af} 
  = \bigl\{ w \PJ(t_{\xi}) \mid w \in W^J,\,\xi \in Q^{\vee} \bigr\}.
\end{equation}

\item For each $\xi \in Q^{\vee}$, the element $\PJ(t_{\xi})$ is 
of the form{\rm:} $\PJ(t_{\xi})=z_{\xi}t_{\xi+\phi_{\J}(\xi)}$ 
for {\rm(}a unique{\rm)} $z_{\xi} \in \WJ$ and $\phi_{\J}(\xi) \in \QJv$. 

\item For $\xi,\,\zeta \in Q^{\vee}$, 
%
%
\begin{equation} \label{eq:PiJ1}
\PJ(t_{\xi}) = \PJ(t_{\zeta}) \iff \xi-\zeta \in \QJv.
\end{equation}
\end{enu}
\end{lem}
%
%
%
%
\begin{lem}[{\cite[Remark~4.1.3]{INS}}] \label{lem:si}
Let $x \in (\WJu)_{\af}$, and $i \in I_{\af}$. Then, 
\begin{equation} \label{eq:si1}
s_{i}x \in (\WJu)_{\af} \iff 
\pair{x\lambda}{\alpha_{i}^{\vee}} \ne 0 \iff 
x^{-1}\alpha_{i} \in (\Delta \setminus \DeJ)+\BZ\delta.
\end{equation}
Moreover, in this case, 
%
%
\begin{equation} \label{eq:simple}
\begin{cases}
x \edge{\alpha_{i}} s_{i}x \iff
\pair{x\lambda}{\alpha_{i}^{\vee}} > 0 \iff 
x^{-1}\alpha_{i} \in (\Delta^{+} \setminus \DeJ^{+})+\BZ\delta, & \\[1.5mm]
s_{i}x \edge{\alpha_{i}} x  \iff 
\pair{x\lambda}{\alpha_{i}^{\vee}} < 0 \iff 
x^{-1}\alpha_{i} \in -(\Delta^{+} \setminus \DeJ^{+})+\BZ\delta. & 
\end{cases}
\end{equation}
\end{lem}
%
%
\begin{rem} \label{rem:si}
Keep the setting of Lemma~\ref{lem:si}. 
If $x^{-1}\alpha_{i} \in \DeJ+\BZ\delta$, i.e., 
$\pair{x\lambda}{\alpha_{i}^{\vee}} = 0$, then $\PJ(s_{i}x) = x$. 
\end{rem}
%
%
\begin{lem}[{\cite[Lemma~2.3.6]{NS16}}] \label{lem:dmd}
Let $x,\,y \in (\WJu)_{\af}$ be such that $x \sile y$, 
and let $i \in I_{\af}$. 
\begin{enu}

\item If $\pair{x\lambda}{\alpha_{i}^{\vee}} > 0$ and 
$\pair{y\lambda}{\alpha_{i}^{\vee}} \le 0$, then $s_{i}x \sile y$.

\item If $\pair{x\lambda}{\alpha_{i}^{\vee}} \ge 0$ and 
$\pair{y\lambda}{\alpha_{i}^{\vee}} < 0$, then $x \sile s_{i}y$. 

\item If $\pair{x\lambda}{\alpha_{i}^{\vee}} > 0$ and 
$\pair{y\lambda}{\alpha_{i}^{\vee}} > 0$, or 
if $\pair{x\lambda}{\alpha_{i}^{\vee}} < 0$ and 
$\pair{y\lambda}{\alpha_{i}^{\vee}} < 0$, then $s_{i}x \sile s_{i}y$.

\end{enu}
\end{lem}
%
%
\begin{lem}[{\cite[Lemmas 4.3.3--4.3.5]{NNS}}] \label{lem:SiB}
\mbox{}
\begin{enu}

\item 
Let $w,\,v \in \WJu$, and $\xi,\,\zeta \in Q^{\vee}$. 
If $w\PJ(t_{\xi}) \sige v\PJ(t_{\zeta})$, 
then $[\xi]^{\J} \ge [\zeta]^{\J}$, 
where $[\,\cdot\,]^{\J}:Q^{\vee} \twoheadrightarrow 
Q^{\vee}_{I \setminus \J}$ is the projection in \eqref{eq:prj}. 

\item 
Let $w \in \WJu$, and $\xi,\,\zeta \in Q^{\vee}$. 
Then, $w\PJ(t_{\xi}) \sige w\PJ(t_{\zeta})$ 
if and only if $[\xi]^{\J} \ge [\zeta]^{\J}$. 

\item 
Let $x,\,y \in (\WJu)_{\af}$ and $\beta \in \prr$ be such that 
$x \edge{\beta} y$ in $\SBJ$. Then, $\PJ(xt_{\xi}) \edge{\beta} \PJ(yt_{\xi})$ 
in $\SBJ$ for all $\xi \in Q^{\vee}$. Therefore, if $x \sige y$, then 
$\PJ(xt_{\xi}) \sige \PJ(yt_{\xi})$ for all $\xi \in Q^{\vee}$. 

\end{enu}
\end{lem}
%
%
\begin{rem} \label{rem:SiB}
Let $w \in W$. 
Since $w \ge e$ in the ordinary Bruhat order on $W$, 
we see that $w \sige e$ in the semi-infinite Bruhat order on $W_{\af}$. 
Hence it follows from Lemma~\ref{lem:SiB}\,(3) that 
$wt_{\xi} \sige t_{\xi}$ for all $\xi \in Q^{\vee}$. 
\end{rem}
%
%
\begin{lem} \label{lem:SiL}
Let $x,\,y \in (\WJu)_{\af}$ be such that 
$x \sile y$ in $(\WJu)_{\af}$. Then, 
$xz \sile yz$ in $W_{\af}$ for all $z \in (\WJ)_{\af}$. 
\end{lem}

\begin{proof}
Let $z \in (\WJ)_{\af}$. We know from \cite{Pet97} (see also \cite[Theorem~3.3]{A}) that 
$\sell(xz) = \sell(x) + \sell(z)$ and $\sell(yz) = \sell(y) + \sell(z)$. 
Also, we may assume that 
$x \edge{\beta} y$ in $\SBJ$ for some $\beta \in \prr$; 
by the definition of $\SBJ$, we have $y = s_{\beta}x$, with $\sell(y)=\sell(x)+1$. 
Therefore, $yz = s_{\beta}xz$, and 
\begin{equation*}
\sell(yz) 
 = \sell(y)+\sell(z)
 = \sell(x)+1+\sell(z) = \sell(xz)+1. 
\end{equation*}
Thus, we obtain $xz \edge{\beta} yz$ in $\SB$, as desired. 
\end{proof}
%
%
\begin{lem}[{\cite[Lemma~6.1.1]{INS}}] \label{lem:611}
If $x,\,y \in W_{\af}$ satisfy $x \sige y$, then 
$\PJ(x) \sige \PJ(y)$. 
\end{lem}

%
\section{Proof of Proposition~\ref{prop:Deo}.}
\label{sec:prf-Deo}
%
%
\begin{lem} \label{lem:lift}
Let $x \in (\WJu)_{\af}$, and write it as{\rm:} 
$x = w \PJ(t_{\xi}) \in (\WJu)_{\af}$ 
for some $w \in \WJu$ and 
$\xi \in Q^{\vee}$ {\rm(}see \eqref{eq:PiJ3}{\rm)}. Then, 
$\Lift(x) = \bigl\{ w' t_{\xi+\gamma} \mid 
 w' \in w\WJ,\, \gamma \in \QJv \bigr\}$. 
\end{lem}

\begin{proof}
We set $L := \bigl\{ w' t_{\xi+\gamma} \mid 
 w' \in w\WJ,\, \gamma \in \QJv \bigr\}$. 
We first prove that $L \subset \Lift(x)$. 
Let $w' t_{\xi+\gamma} \in L$. 
Then we see from \eqref{eq:PiJ2} and \eqref{eq:PiJ1} that
$\PJ(w' t_{\xi+\gamma}) = \PJ(w')\PJ(t_{\xi+\gamma}) 
 = \mcr{w'}\Pi(t_{\xi}) = w \Pi(t_{\xi}) = x$.
This proves the inclusion $L \subset \Lift(x)$. 

We next show that $L \supset \Lift(x)$. 
Each element of $\Lift(x)$ is of the form 
$x z$ for some $z \in (\WJ)_{\af}=\WJ \ltimes \QJv$; 
we write $z$ as $z = v_{1}t_{\gamma_{1}}$ 
for some $v_{1} \in \WJ$ and $\gamma_{1} \in \QJv$. 
Since $\PJ(t_{\xi}) = v_{2}t_{\xi + \gamma_{2}}$ 
for some $v_{2} \in \WJ$ and $\gamma_{2} \in \QJv$, 
we have $xz = (w \PJ(t_{\xi})) (v_{1}t_{\gamma_{1}}) = 
w (v_{2}t_{\xi + \gamma_{2}})(v_{1}t_{\gamma_{1}}) = 
wv_{2}v_{1}t_{ v_{1}^{-1}(\xi + \gamma_{2}) + \gamma_{1} }$, 
which is of the form $w v t_{\xi+\gamma}$ 
for some $v \in \WJ$ and $\gamma \in \QJv$. 
Thus, this element is contained in $L$. 
This proves the opposite inclusion, and hence the lemma. 
\end{proof}

Now, we give a proof of Proposition~\ref{prop:Deo}. 
If $J=I$, then the assertion is obvious. 
Hence we may assume that $J \subsetneqq I$. 

\paragraph{Step 1.}
%
Assume that $x = t_{\xi}$ for some $\xi \in Q^{\vee}$, and 
$y = \PJ(t_{\zeta})$ for some $\zeta \in Q^{\vee}$; 
since $\PJ(t_{\xi})=y \sige \PJ(x)=\PJ(t_{\xi})$ by 
the assumption, it follows from Lemma~\ref{lem:SiB}\,(1) 
that $[\zeta]^{\J} \ge [\xi]^{\J}$, 
where $[\,\cdot\,]^{\J}:Q^{\vee} \twoheadrightarrow Q_{I \setminus \J}^{\vee}$
is the projection in \eqref{eq:prj}. 
We set $\gamma:=[\zeta-\xi]_{\J} \in \QJv$, 
where $[\,\cdot\,]_{\J}:Q^{\vee} \twoheadrightarrow \QJv$ 
is the projection in \eqref{eq:prj}; 
note that $[\zeta-\gamma]_{\J} = [\xi]_{\J}$. 
We claim that $t_{\zeta-\gamma}$ is the minimum element of $\Lif{x}{y}$.
It is clear by Lemma~\ref{lem:lift} that $t_{\zeta-\gamma} \in \Lift(y)$. 
In addition, since $[\zeta-\gamma]^{\J} = [\zeta]^{\J} \ge [\xi]^{\J}$
and $[\zeta-\gamma]_{\J} = [\xi]_{\J}$, we have $\zeta - \gamma \ge \xi$, 
and hence $t_{\zeta - \gamma} \sige t_{\xi} = x$ by Lemma~\ref{lem:SiB}\,(2).
Thus, $t_{\zeta - \gamma} \in \Lif{x}{y}$. 
Now, by Lemma~\ref{lem:lift}, 
each element $y' \in \Lif{x}{y} \subset \Lift(y)$ is of the form 
$y'=v't_{\zeta-\gamma'}$ for some $v' \in \WJ$ and $\gamma' \in \QJv$. 
Since $v' t_{\zeta-\gamma'} = y' \sige x = t_{\xi}$ in $W_{\af}$ 
by the assumption, we deduce from Lemma~\ref{lem:SiB}\,(1) that 
$\zeta-\gamma' \ge \xi$; in particular, $[\zeta-\gamma']_{\J} \ge [\xi]_{\J}$. 
Here, since $\gamma=[\zeta-\xi]_{\J}$ by the definition, we have 
$[\zeta-\gamma']_{\J} \ge [\xi]_{\J}=[\zeta-\gamma]_{\J}$. 
Also, since $\gamma,\,\gamma' \in \QJv$, we have 
$[\zeta-\gamma']^{\J} = [\zeta-\gamma]^{\J}$. 
Combining these, 
we obtain $\zeta-\gamma' \ge \zeta-\gamma$. 
Therefore, by Remark~\ref{rem:SiB} and Lemma~\ref{lem:SiB}\,(2), 
$y'=v' t_{\zeta-\gamma'} \sige t_{\zeta-\gamma'} \sige t_{\zeta-\gamma}$. 
Thus, $t_{\zeta-\gamma}$ is 
the minimum element of $\Lif{x}{y}$. \bqed

\vspace{3mm}

In the following, 
we fix $\Lambda \in P^{+}$ and $\lambda \in P^{+}$ such that 
$\bigl\{ i \in I \mid \pair{\Lambda}{\alpha_{i}^{\vee}} = 0\bigr\} = \emptyset$ and 
$\bigl\{ i \in I \mid \pair{\lambda}{\alpha_{i}^{\vee}} = 0\bigr\} = \J$. 
Note that $\pair{\Lambda}{\beta^{\vee}} \ne 0$ for all $\beta \in \rr$. 
%
\paragraph{Step 2.}
%
Let $x \in W_{\af}$, and assume that 
$y = \PJ(t_{\zeta})$ for some $\zeta \in Q^{\vee}$.
We deduce from \cite{AK} 
that there exist $i_{1},\,\dots,\,i_{N} \in I_{\af}$ such that
$\pair{s_{i_{n-1}} \cdots s_{i_{1}}x\Lambda}{\alpha_{i_n}^{\vee}} > 0$ 
for all $1 \le n \le N$, and such that
$s_{i_{N}} \cdots s_{i_{1}}x = t_{\xi}$ for some $\xi \in Q^{\vee}$. 
We show the assertion of the proposition by induction on $N$. If $N=0$, 
then the assertion follows from Step~1. Assume that $N \ge 1$; 
for simplicity of notation, we set $i:=i_{1} \in I_{\af}$. 
Since 
\begin{equation} \label{eq:s2-1}
\pair{x\Lambda}{\alpha_{i}^{\vee}} > 0
\end{equation}
by the assumption, 
it follows that $x^{-1}\alpha_{i} \in \Delta^{+}$, and hence 
$\pair{\PJ(x)\lambda}{\alpha_{i}^{\vee}} = 
\pair{x\lambda}{\alpha_{i}^{\vee}} \ge 0$. Also, 
by Lemma~\ref{lem:si} and Remark~\ref{rem:si}, 
\begin{equation} \label{eq:Pix}
\PJ(s_{i}x) = 
 \begin{cases}
 s_{i}\PJ(x) & \text{if $\pair{x\lambda}{\alpha_{i}^{\vee}} > 0$}, \\[1mm]
 \PJ(x) & \text{if $\pair{x\lambda}{\alpha_{i}^{\vee}} = 0$}.
 \end{cases}
\end{equation}

\paragraph{Case 2.1.}
%
Assume that $i \in I \setminus \J$, and hence
$\pair{y\lambda}{\alpha_{i}^{\vee}}=
\pair{\lambda}{\alpha_{i}^{\vee}} > 0$; 
note that $s_{i}y \in (\WJu)_{\af}$ by Lemma~\ref{lem:si}. 
We first claim that $s_{i}y \sige \PJ(s_{i}x)$. 
Indeed, if $\pair{\PJ(x)\lambda}{\alpha_{i}^{\vee}} = 
\pair{x\lambda}{\alpha_{i}^{\vee}} > 0$, 
then it follows from Lemma~\ref{lem:dmd}\,(3), \eqref{eq:Pix}, and 
the assumption $y \sige \PJ(x)$ that 
$s_{i}y \sige s_{i}\PJ(x) = \PJ(s_{i}x)$. 
If $\pair{\PJ(x)\lambda}{\alpha_{i}^{\vee}} = \pair{x\lambda}{\alpha_{i}^{\vee}} = 0$, 
then it follows from Lemma~\ref{lem:si}, 
\eqref{eq:Pix}, and the assumption $y \sige \PJ(x)$ that 
$s_{i}y \sig y \sige \PJ(x) = \PJ(s_{i}x)$. In both cases, 
we obtain $s_{i}y \sige \PJ(s_{i}x)$, as desired. 
Hence, by our induction hypothesis, 
\begin{equation*}
\Lif{s_{i}x}{s_{i}y}=
\bigl\{ y'' \in W_{\af} \mid 
 \text{\rm $\PJ(y'') = s_{i}y$ and $y'' \sige s_{i}x$}
\bigr\}
\end{equation*}
has the minimum element $y''_{\min}$. 
We next claim that $s_{i}y''_{\min} \in \Lif{x}{y}$. 
Indeed, since $\pair{y''_{\min}\lambda}{\alpha_{i}^{\vee}} = 
\pair{\PJ(y''_{\min})\lambda}{\alpha_{i}^{\vee}} = 
\pair{s_{i}y\lambda}{\alpha_{i}^{\vee}} = - \pair{y\lambda}{\alpha_{i}} < 0$, 
it follows from Lemma~\ref{lem:si} that 
$\PJ(s_{i}y''_{\min}) = s_{i} \PJ(y''_{\min}) = s_{i}(s_{i}y) = y$, 
which implies that $s_{i}y''_{\min} \in \Lift(y)$. 
In addition, the inequality $\pair{y''_{\min}\lambda}{\alpha_{i}^{\vee}} < 0$ above
implies that $\pair{y''_{\min}\Lambda}{\alpha_{i}^{\vee}} < 0$. 
Since $\pair{s_{i}x\Lambda}{\alpha_{i}^{\vee}} < 0$ by \eqref{eq:s2-1},
we deduce from Lemma~\ref{lem:dmd}\,(3), together with 
the assumption $y''_{\min} \sige s_{i}x$, that 
$s_{i}y''_{\min} \sige x$. Thus we get 
$s_{i}y''_{\min} \in \Lif{x}{y}$, 
as desired. Finally, we claim that 
\begin{equation} \label{eq:min2-1}
\text{$s_{i}y''_{\min}$ is the minimum element of $\Lif{x}{y}$}.
\end{equation}
Let $y' \in \Lif{x}{y}$. 
Since $\pair{y'\lambda}{\alpha_{i}^{\vee}}=
\pair{y\lambda}{\alpha_{i}^{\vee}} = \pair{\lambda}{\alpha_{i}^{\vee}} > 0$ 
by our assumption, we see that 
$(y')^{-1}\alpha_{i} \in \Delta^{+} + \BZ \delta$, 
and hence $\pair{y'\Lambda}{\alpha_{i}^{\vee}} > 0$. 
Also, since $\pair{x\Lambda}{\alpha_{i}^{\vee}} > 0$ by \eqref{eq:s2-1}, 
it follows from Lemma~\ref{lem:dmd}\,(3) that 
$s_{i}y' \sige s_{i}x$, which implies that 
$s_{i}y' \in \Lif{s_{i}x}{s_{i}y}$. 
Therefore, we obtain $s_{i}y' \sige y''_{\min}$. 
Since $\pair{y''_{\min}\Lambda}{\alpha_{i}^{\vee}} < 0$ as seen above, 
and $\pair{s_{i}y'\Lambda}{\alpha_{i}^{\vee}} < 0$, 
we deduce from Lemma~\ref{lem:dmd}\,(3) that 
$y' \sige s_{i}y''_{\min}$. This shows \eqref{eq:min2-1}. 

\paragraph{\bf Case 2.2.}
%
Assume that $i \in \J$, and hence 
$\pair{y\lambda}{\alpha_{i}^{\vee}}=
\pair{\lambda}{\alpha_{i}^{\vee}} = 0$. 
We first claim that $y \sige \PJ(s_{i}x)$. 
Indeed, if $\pair{\PJ(x)\lambda}{\alpha_{i}^{\vee}} 
 = \pair{x\lambda}{\alpha_{i}^{\vee}} > 0$, 
then it follows from Lemma~\ref{lem:dmd}\,(1), \eqref{eq:Pix}, and 
the assumption $y \sige \PJ(x)$ that 
$y \sige s_{i}\PJ(x) = \PJ(s_{i}x)$. 
If $\pair{\PJ(x)\lambda}{\alpha_{i}^{\vee}} = \pair{x\lambda}{\alpha_{i}^{\vee}} = 0$, 
then $y \sige \PJ(x) = \PJ(s_{i}x)$ by \eqref{eq:Pix}. 
In both cases, we obtain $y \sige \PJ(s_{i}x)$, as desired. 
Hence, by our induction hypothesis, 
\begin{equation*}
\Lif{s_{i}x}{y}=
\bigl\{ y'' \in W_{\af} \mid 
 \text{\rm $\PJ(y'') = y$ and $y'' \sige s_{i}x$}
\bigr\}
\end{equation*}
has the minimum element $y''_{\min}$. We set
\begin{equation*}
y'_{\min}:=
 \begin{cases}
 y''_{\min} & 
 \text{if $\pair{ y''_{\min}\Lambda }{ \alpha_{i}^{\vee} } > 0$}, \\[1mm]
 s_{i}y''_{\min} & 
 \text{if $\pair{ y''_{\min}\Lambda }{ \alpha_{i}^{\vee} } < 0$}; 
 \end{cases}
\end{equation*}
remark that $y'_{\min} \sile y''_{\min}$ by Lemma~\ref{lem:si}. 
First, we show that $y'_{\min} \in \Lif{x}{y}$. 
Since $y''_{\min} \in \Lift(y)$ and $i \in \J$, 
it follows from Remark~\ref{rem:si} that 
$y'_{\min} \in \Lift(y)$. 
Also, since $\pair{x\Lambda}{\alpha_{i}^{\vee}} > 0$, 
we see by Lemma~\ref{lem:si} that $s_{i}x \sige x$. 
Hence we have $y_{\min}' = y_{\min}'' \sige s_{i}x \sige x$ 
if $\pair{ y''_{\min}\Lambda }{ \alpha_{i}^{\vee} } > 0$. 
If $\pair{ y''_{\min}\Lambda }{ \alpha_{i}^{\vee} } < 0$, 
then we deduce from Lemma~\ref{lem:dmd}\,(3) that 
$y_{\min}' = s_{i}y_{\min}'' \sige s_{i}(s_{i}x) = x$ 
since $y_{\min}'' \sige s_{i}x$. 
Thus, in both cases, we obtain 
$y'_{\min} \in \Lif{x}{y}$, as desired. 
Next, we show that 
\begin{equation} \label{eq:min2-2}
\text{$y'_{\min}$ is the minimal element of $\Lif{x}{y}$}.
\end{equation}
Let $y' \in \Lif{x}{y}$. 
If $\pair{y'\Lambda}{\alpha_{i}^{\vee}} < 0$, then 
it follows from Lemma~\ref{lem:dmd}\,(1) that 
$y' \sige s_{i}x$, and hence $y' \in \Lif{s_{i}x}{y}$. 
This implies that $y' \sige y_{\min}''$. 
If $\pair{ y''_{\min}\Lambda }{ \alpha_{i}^{\vee} } > 0$, 
then we have $y' \sige y_{\min}'' = y_{\min}'$ by the definition. 
If $\pair{ y''_{\min}\Lambda }{ \alpha_{i}^{\vee} } < 0$, 
then we see from Lemma~\ref{lem:si} that 
$y' \sige y_{\min}'' \sige s_{i}y_{\min}'' = y_{\min}'$. 
Assume now that $\pair{y'\Lambda}{\alpha_{i}^{\vee}} > 0$. 
Since $\pair{x\Lambda}{\alpha_{i}^{\vee}} > 0$, 
it follows from Lemma~\ref{lem:dmd}\,(3) that 
$s_{i}y' \sige s_{i}x$. 
In addition, since $i \in \J$ and $y' \in \Lift(y)$, 
we see from Remark~\ref{rem:si} that $\PJ(s_{i}y')=y$. 
Hence we obtain $s_{i}y' \in \Lif{s_{i}x}{y}$, 
so that $s_{i}y' \sige y''_{\min}$; note that 
$\pair{s_{i}y'\Lambda}{\alpha_{i}^{\vee}} < 0$ by our assumption. 
If $\pair{ y''_{\min}\Lambda }{ \alpha_{i}^{\vee} } > 0$, then 
we deduce from Lemma~\ref{lem:dmd}\,(2) that 
$y' \sige y''_{\min} = y'_{\min}$.
If $\pair{ y''_{\min}\Lambda }{ \alpha_{i}^{\vee} } < 0$, then 
we deduce from Lemma~\ref{lem:dmd}\,(3) that 
$y' \sige s_{i}y''_{\min} = y'_{\min}$. 
Thus, in all cases, we have shown that 
$y' \sige y'_{\min}$, as desired. 

\paragraph{\bf Case 2.3.}
%
Assume that $i = 0$. In this case, we have
$\pair{y\lambda}{\alpha_{0}^{\vee}}=
\pair{\lambda}{\alpha_{0}^{\vee}} < 0$
since $\alpha_{0}=-\theta+\delta$, 
where $\theta \in \Delta^{+}$ is the highest root. 
By the same argument as that at the beginning of Case 2.2, 
we see that $y \sige \PJ(s_{0}x)$. 
Hence, by the induction hypothesis, 
\begin{equation*}
\Lif{s_{0}x}{y}=
\bigl\{ y'' \in W_{\af} \mid 
 \text{\rm $\PJ(y'') = y$ and $y'' \sige s_{0}x$}
\bigr\}
\end{equation*}
has the minimum element $y''_{\min}$. 
Since $\pair{ x\Lambda }{ \alpha_{0}^{\vee} } > 0$ by \eqref{eq:s2-1}, 
it follows from Lemma~\ref{lem:si} that $s_{0}x \sige x$, 
which implies that $y''_{\min} \in \Lif{x}{y}$. 
Here we claim that 
\begin{equation*}
\text{$y''_{\min}$ is the minimal element of $\Lif{x}{y}$.}
\end{equation*}
Let $y' \in \Lif{x}{y}$. Then we have 
$\pair{ y'\Lambda }{ \alpha_{0}^{\vee} } < 0$. 
Indeed, we deduce from Lemma~\ref{lem:lift} that 
$y'=z t_{\zeta+\gamma}$ for some $z \in \WJ$ and $\gamma \in \QJv$. 
Since $z \in \WJ$ and $\theta \in \Delta^{+} \setminus \DeJ^{+}$ 
(recall that $J \subsetneqq I$), we see that
$z^{-1}\theta \in \Delta^{+} \setminus \DeJ^{+}$, 
and hence that $\pair{y'\Lambda}{\alpha_{0}^{\vee}} = 
\pair{\Lambda}{-z^{-1}\theta^{\vee}} < 0$. 
Since $\pair{ x\Lambda }{ \alpha_{0}^{\vee} } > 0$ by \eqref{eq:s2-1}, 
it follows from Lemma~\ref{lem:dmd}\,(1) that $y' \sige s_{0}x$, 
and hence $y' \in \Lif{s_{0}x}{y}$. 
This shows that $y' \sige y''_{\min}$. \bqed

\paragraph{Step 3.}
%
Let $x \in W_{\af}$, $y \in (\WJu)_{\af}$, 
and write $y$ as $y = v \PJ(t_{\zeta})$ 
for some $v \in \WJu$ and $\zeta \in Q^{\vee}$. 
We show the assertion by induction on $\ell(v)$. If $\ell(v)=0$, 
then the assertion follows from Step 2. Assume that $\ell(v) \ge 1$, 
and take $i \in I$ such that $\pair{v\lambda}{\alpha_{i}^{\vee}} < 0$; 
note that in this case, $v^{-1}\alpha_{i} \in - (\Delta^{+} \setminus \DeJ^{+})$, 
and $s_{i}v \in W^{J}$ (see, for example, \cite[Lemmas~5.8 and 5.9]{LNSSS}), 
and hence that $s_{i}y \in (\WJu)_{\af}$.
Also, for all $y' \in \Lift(y)$, we have 
$\pair{y'\lambda}{\alpha_{i}^{\vee}} = 
\pair{y\lambda}{\alpha_{i}^{\vee}} = 
\pair{v\lambda}{\alpha_{i}^{\vee}} < 0$, 
which implies that 
%
%
\begin{equation} \label{eq:S3-1}
\pair{y'\Lambda}{\alpha_{i}^{\vee}} < 0. 
\end{equation}

\paragraph{Case 3.1.}
%
Assume that $\pair{x\Lambda}{\alpha_{i}^{\vee}} > 0$; 
note that $\pair{\PJ(x)\lambda}{\alpha_{i}^{\vee}} = 
\pair{x\lambda}{\alpha_{i}^{\vee}} \ge 0$. 
Since $\pair{y\lambda}{\alpha_{i}^{\vee}} = \pair{v\lambda}{\alpha_{i}^{\vee}} < 0$, 
it follows from Lemma~\ref{lem:dmd}\,(2) that 
$s_{i}y \sige \PJ(x)$. Hence, by our induction hypothesis, 
\begin{equation*}
\Lif{x}{s_{i}y} =
\bigl\{ y'' \in W_{\af} \mid 
 \text{\rm $\PJ(y'') = s_{i}y$ and $y'' \sige x$}
\bigr\}
\end{equation*}
has the minimum element $y''_{\min}$. 
Since $\pair{y''_{\min}\lambda}{\alpha_{i}^{\vee}} = 
\pair{s_{i}y\lambda}{\alpha_{i}^{\vee}} > 0$, it follows that 
$(y''_{\min})^{-1}\alpha_{i} \in \Delta^{+}$, and hence 
$\pair{y''_{\min}\Lambda}{\alpha_{i}^{\vee}} > 0$.
This implies that $s_{i}y''_{\min} \sige y''_{\min} \sige x$
by Lemma~\ref{lem:si}. In addition, we see by Lemma~\ref{lem:si} 
that $\PJ(s_{i}y''_{\min}) = s_{i}\PJ(y''_{\min}) = s_{i}(s_{i}y) = y$. 
Therefore, we conclude that $s_{i}y''_{\min} \in \Lif{x}{y}$. 
Here we claim that 
%
%
\begin{equation} \label{eq:min3-1}
\text{$s_{i}y''_{\min}$ is the minimum element of $\Lif{x}{y}$}.
\end{equation}
Let $y' \in \Lif{x}{y}$. 
Since $\pair{x\Lambda}{\alpha_{i}^{\vee}} > 0$ by the assumption, and 
$\pair{y'\Lambda}{\alpha_{i}^{\vee}} < 0$ by \eqref{eq:S3-1}, 
we deduce from Lemma~\ref{lem:dmd}\,(2) that 
$s_{i}y' \sige x$. 
In addition, we see by Lemma~\ref{lem:si} that 
$\PJ(s_{i}y') = s_{i} \PJ(y') = s_{i}y$, 
which implies that $s_{i}y' \in \Lif{x}{s_{i}y}$, 
and hence $s_{i}y' \sige y''_{\min}$. 
Because $\pair{y''_{\min}\Lambda}{\alpha_{i}^{\vee}} > 0$ and 
$\pair{s_{i}y'\Lambda}{\alpha_{i}^{\vee}} > 0$, 
it follows from Lemma~\ref{lem:dmd}\,(3) that 
$y' \sige s_{i}y''_{\min}$. This shows \eqref{eq:min3-1}. 

\paragraph{Case 3.2.}
%
Assume that $\pair{x\Lambda}{\alpha_{i}^{\vee}} < 0$; 
note that $\pair{x\lambda}{\alpha_{i}^{\vee}} \le 0$. 
Since $\pair{y\lambda}{\alpha_{i}^{\vee}} < 0$,
it follows from Lemma~\ref{lem:dmd}\,(2) and (3), 
together with Lemma~\ref{lem:si} and Remark~\ref{rem:si},
that $s_{i}y \sige \PJ(s_{i}x)$. 
Hence, by our induction hypothesis, 
\begin{equation*}
\Lif{s_{i}x}{s_{i}y} = 
\bigl\{ y'' \in W_{\af} \mid 
 \text{\rm $\PJ(y'') = s_{i}y$ and $y'' \sige s_{i}x$}
\bigr\}
\end{equation*}
has the minimum element $y''_{\min}$; as in Case 3.1, we obtain 
$\pair{y''_{\min}\Lambda}{\alpha_{i}^{\vee}} > 0$.
Since $\pair{s_{i}x\Lambda}{\alpha_{i}^{\vee}} > 0$, 
we deduce from Lemma~\ref{lem:dmd}\,(3) that 
$s_{i}y''_{\min} \sige x$. 
In addition, we see by Lemma~\ref{lem:si}
that $\PJ(s_{i}y_{\min}'') = s_{i} \PJ(y_{\min}'') = s_{i} (s_{i}y) = y$, 
which implies that $s_{i}y''_{\min} \in \Lif{x}{y}$. 
Here we claim that 
%
%
\begin{equation} \label{eq:min3-2}
\text{$s_{i}y''_{\min}$ is the minimum element of $\Lif{x}{y}$}.
\end{equation}
Let $y' \in \Lif{x}{y}$. 
Since $\pair{y'\lambda}{\alpha_{i}^{\vee}} = 
\pair{y\lambda}{\alpha_{i}^{\vee}} < 0$, 
it follows that $\pair{y'\Lambda}{\alpha_{i}^{\vee}} < 0$.  
Also, since $\pair{x\Lambda}{\alpha_{i}^{\vee}} < 0$ 
by the assumption, we deduce from Lemma~\ref{lem:dmd}\,(3) that 
$s_{i}y' \sige s_{i}x$. 
In addition, we see by Lemma~\ref{lem:si} that
$\PJ(s_{i}y') = s_{i}\PJ(y') = s_{i}y$, 
which implies that 
$s_{i}y' \in \Lif{s_{i}x}{s_{i}y}$, 
and hence $s_{i}y' \sige y''_{\min}$. 
Because $\pair{y''_{\min}\Lambda}{\alpha_{i}^{\vee}} > 0$ and 
$\pair{s_{i}y'\Lambda}{\alpha_{i}^{\vee}} > 0$, 
it follows from Lemma~\ref{lem:dmd}\,(3) that 
$y' \sige s_{i}y''_{\min}$. This shows \eqref{eq:min3-2}. 

This completes the proof of Proposition~\ref{prop:Deo}. 

%
\section{Crystal structure on $\SLS(\lambda)$.}
\label{sec:crystal}

We fix $\lambda \in P^{+} \subset P_{\af}^{0}$ 
(see \eqref{eq:P-fin} and \eqref{eq:P}). Let 
\begin{equation} \label{eq:pi2}
\pi = (\bx \,;\, \ba) 
     = (x_{1},\,\dots,\,x_{s} \,;\, a_{0},\,a_{1},\,\dots,\,a_{s}) \in \SLS(\lambda). 
\end{equation}
Define $\ol{\pi}:[0,1] \rightarrow \BR \otimes_{\BZ} P_{\af}$ 
to be the piecewise-linear, continuous map 
whose ``direction vector'' on the interval 
$[a_{u-1},\,a_{u}]$ is $x_{u}\lambda \in P_{\af}$ 
for each $1 \le u \le s$, that is, 
%
%
\begin{equation} \label{eq:olpi}
\ol{\pi} (t) := 
\sum_{k = 1}^{u-1}(a_{k} - a_{k-1}) x_{k}\lambda + (t - a_{u-1}) x_{u}\lambda
\quad
\text{for $t \in [a_{u-1},\,a_u]$, $1 \le u \le s$}. 
\end{equation}
We know from \cite[Proposition~3.1.3]{INS} that $\ol{\pi}$ 
is an (ordinary) LS path of shape $\lambda$, introduced in \cite[Sect.~4]{Lit95}. 
We set
%
%
\begin{equation} \label{eq:wt}
\wt (\pi):= \ol{\pi}(1) = \sum_{u = 1}^{s} (a_{u}-a_{u-1})x_{u}\lambda \in P_{\af}.
\end{equation}

Now, we define root operators $e_{i}$, $f_{i}$, $i \in I_{\af}$. 
Set 
%
%
\begin{equation} \label{eq:H}
\begin{cases}
H^{\pi}_{i}(t) := \pair{\ol{\pi}(t)}{\alpha_{i}^{\vee}} \quad 
\text{for $t \in [0,1]$}, \\[1.5mm]
m^{\pi}_{i} := 
 \min \bigl\{ H^{\pi}_{i} (t) \mid t \in [0,1] \bigr\}. 
\end{cases}
\end{equation}
As explained in \cite[Remark~2.4.3]{NS16}, 
all local minima of the function $H^{\pi}_{i}(t)$, $t \in [0,1]$, 
are integers; in particular, 
the minimum value $m^{\pi}_{i}$ is a nonpositive integer 
(recall that $\ol{\pi}(0)=0$, and hence $H^{\pi}_{i}(0)=0$).

We define $e_{i}\pi$ as follows. 
If $m^{\pi}_{i}=0$, then we set $e_{i} \pi := \bzero$, 
where $\bzero$ is an additional element not 
contained in any crystal. 
If $m^{\pi}_{i} \le -1$, then we set
%
%
\begin{equation} \label{eq:t-e}
\begin{cases}
t_{1} := 
  \min \bigl\{ t \in [0,\,1] \mid 
    H^{\pi}_{i}(t) = m^{\pi}_{i} \bigr\}, \\[1.5mm]
t_{0} := 
  \max \bigl\{ t \in [0,\,t_{1}] \mid 
    H^{\pi}_{i}(t) = m^{\pi}_{i} + 1 \bigr\}; 
\end{cases}
\end{equation}
notice that $H^{\pi}_{i}(t)$ is 
strictly decreasing on the interval $[t_{0},\,t_{1}]$. 
Let $1 \le p \le q \le s$ be such that 
$a_{p-1} \le t_{0} < a_p$ and $t_{1} = a_{q}$. 
Then we define $e_{i}\pi$ to be
%
%
\begin{equation} \label{eq:epi}
\begin{split}
& e_{i} \pi := ( 
  x_{1},\,\ldots,\,x_{p},\,s_{i}x_{p},\,s_{i}x_{p+1},\,\ldots,\,
  s_{i}x_{q},\,x_{q+1},\,\ldots,\,x_{s} ; \\
& \hspace*{40mm}
  a_{0},\,\ldots,\,a_{p-1},\,t_{0},\,a_{p},\,\ldots,\,a_{q}=t_{1},\,
\ldots,\,a_{s});
\end{split}
\end{equation}
if $t_{0} = a_{p-1}$, then we drop $x_{p}$ and $a_{p-1}$, and 
if $s_{i} x_{q} = x_{q+1}$, then we drop $x_{q+1}$ and $a_{q}=t_{1}$.

Similarly, we define $f_{i}\pi$ as follows. 
Note that $H^{\pi}_{i}(1) - m^{\pi}_{i}$ is a nonnegative integer. 
If $H^{\pi}_{i}(1) - m^{\pi}_{i} = 0$, then we set $f_{i} \pi := \bzero$. 
If $H^{\pi}_{i}(1) - m^{\pi}_{i}  \ge 1$, 
then we set
%
%
\begin{equation} \label{eq:t-f}
\begin{cases}
t_{0} := 
 \max \bigl\{ t \in [0,1] \mid H^{\pi}_{i}(t) = m^{\pi}_{i} \bigr\}, \\[1.5mm]
t_{1} := 
 \min \bigl\{ t \in [t_{0},\,1] \mid H^{\pi}_{i}(t) = m^{\pi}_{i} + 1 \bigr\};
\end{cases}
\end{equation}
notice that $H^{\pi}_{i}(t)$ is 
strictly increasing on the interval $[t_{0},\,t_{1}]$. 
Let $0 \le p \le q \le s-1$ be such that $t_{0} = a_{p}$ and 
$a_{q} < t_{1} \le a_{q+1}$. Then we define $f_{i}\pi$ to be
%
%
\begin{equation} \label{eq:fpi}
\begin{split}
& f_{i} \pi := ( x_{1},\,\ldots,\,x_{p},\,s_{i}x_{p+1},\,\dots,\,
  s_{i} x_{q},\,s_{i} x_{q+1},\,x_{q+1},\,\ldots,\,x_{s} ; \\
& \hspace{40mm} 
  a_{0},\,\ldots,\,a_{p}=t_{0},\,\ldots,\,a_{q},\,t_{1},\,
  a_{q+1},\,\ldots,\,a_{s});
\end{split}
\end{equation}
if $t_{1} = a_{q+1}$, then we drop $x_{q+1}$ and $a_{q+1}$, and 
if $x_{p} = s_{i} x_{p+1}$, then we drop $x_{p}$ and $a_{p}=t_{0}$.
In addition, we set $e_{i} \bzero = f_{i} \bzero := \bzero$ 
for all $i \in I_{\af}$.
%
%
\begin{thm}[{see \cite[Theorem~3.1.5]{INS}}] \label{thm:SLS}
\mbox{}
\begin{enu}
\item The set $\SLS(\lambda) \sqcup \{ \bzero \}$ is 
stable under the action of the root operators 
$e_{i}$ and $f_{i}$, $i \in I_{\af}$.

\item For each $\pi \in \SLS(\lambda)$ 
and $i \in I_{\af}$, we set 
\begin{equation*}
\begin{cases}
\ve_{i} (\pi) := 
 \max \bigl\{ n \ge 0 \mid e_{i}^{n} \pi \neq \bzero \bigr\}, \\[1.5mm]
\vp_{i} (\pi) := 
 \max \bigl\{ n \ge 0 \mid f_{i}^{n} \pi \neq \bzero \bigr\}.
\end{cases}
\end{equation*}
Then, the set $\SLS(\lambda)$, 
equipped with the maps $\wt$, $e_{i}$, $f_{i}$, $i \in I_{\af}$, 
and $\ve_{i}$, $\vp_{i}$, $i \in I_{\af}$, 
defined above, is a crystal with weights in $P_{\af}$.
\end{enu}
\end{thm}
%
%
\begin{rem} \label{rem:SLS}
Let $\pi \in \SLS(\lambda)$, and $i \in I_{\af}$. 
If $e_{i}\pi \ne \bzero$, then we deduce from the definition of 
the root operator $e_{i}$ that $m^{e_{i}\pi}_{i} = m^{\pi}_{i}+1$. 
Hence it follows that $\ve_{i}(\pi)=-m^{\pi}_{i}$. 
Similarly, we have $\vp_{i}(\pi)=H^{\pi}_{i}(1)-m^{\pi}_{i}$. 
\end{rem}
%
%
\section{A formula for graded characters of Demazure submodules.}
\label{sec:tx-gch}
%
%
\begin{prop} \label{prop:gch-tx}
For each $x \in W_{\af}$ and $\xi \in Q^{\vee}$, 
there holds the equality
\begin{equation}
\gch V_{xt_{\xi}}^{-}(\lambda) = q^{-\pair{\lambda}{\xi}}\gch V_{x}^{-}(\lambda).
\end{equation}
\end{prop}

\begin{proof}
Let $\pi = (x_{1},\,\dots,\,x_{s}\,;\,\ba) \in \SLS(\lambda)$. We see that
\begin{align*}
\pi \in \SLS_{\sige x}(\lambda) 
 & \Rightarrow x_{s} \sige \PJ(x) 
   \Rightarrow \PJ(x_{s}t_{\xi}) \sige \PJ(\PJ(x)t_{\xi}) 
   \quad \text{by Lemma~\ref{lem:SiB}\,(3)} \\
 & \Rightarrow \PJ(x_{s}t_{\xi}) \sige \PJ(xt_{\xi}) 
   \quad \text{by \eqref{eq:PiJ2}} \\
 & \Rightarrow T_{\xi}(\pi) \in \SLS_{\sige xt_{\xi}}(\lambda);
\end{align*}
for the definition of $T_{\xi}$, see \eqref{eq:Tx}. 
From this, we conclude that $T_{\xi}(\SLS_{\sige x}(\lambda)) \subset
\SLS_{\sige xt_{\xi}}(\lambda)$. Replacing $x$ by $xt_{\xi}$, 
and $T_{\xi}$ by $T_{-\xi}$, we obtain 
$T_{-\xi}(\SLS_{\sige xt_{\xi}}(\lambda)) \subset
\SLS_{\sige x}(\lambda)$, and hence 
$\SLS_{\sige xt_{\xi}}(\lambda) \subset T_{\xi}(\SLS_{\sige x}(\lambda))$. 
Combining these, we conclude that 
$T_{\xi}(\SLS_{\sige x}(\lambda)) = \SLS_{\sige xt_{\xi}}(\lambda)$. 
Therefore, using \eqref{eq:gch2}, 
we compute: 
\begin{align*}
\gch V_{xt_{\xi}}^{-}(\lambda) 
 & = \sum_{\pi \in \SLS_{\sige xt_{\xi}}(\lambda)} e^{\fwt(\wt(\pi))} q^{\qwt(\wt(\pi))}
   = \sum_{\pi \in \SLS_{\sige x}(\lambda)} 
      e^{\fwt(\wt(T_{\xi}(\pi)))} 
      q^{\qwt(\wt(T_{\xi}(\pi)))} \\[3mm]
 & = \sum_{\pi \in \SLS_{\sige x}(\lambda)} 
      e^{\fwt(\wt(\pi)-\pair{\lambda}{\xi}\delta)} 
      q^{\qwt(\wt(\pi)-\pair{\lambda}{\xi}\delta)} 
      \quad \text{by \eqref{eq:Tx}} \\[3mm]
 & = \sum_{\pi \in \SLS_{\sige x}(\lambda)} 
      e^{\fwt(\wt(\pi))} 
      q^{\qwt(\wt(\pi))-\pair{\lambda}{\xi}}
   = q^{-\pair{\lambda}{\xi}} \gch V_{x}^{-}(\lambda). 
\end{align*}
This proves the proposition.
\end{proof}
%
%
{\small
 \setlength{\baselineskip}{10pt}

}

\end{document}